\documentclass[11pt]{amsproc}
 \usepackage[foot]{amsaddr}
 \usepackage[title]{appendix}
\usepackage{amsfonts}
\usepackage{fullpage, amsmath, amsthm,amsfonts,amssymb,stmaryrd, mathrsfs,amscd}
\usepackage{graphicx}
\usepackage{hyperref}
\makeatletter
\usepackage[pdftex]{color}
\usepackage[all]{xy}

\newtheorem{thm}[subsubsection]{Theorem}
\newtheorem{prop}[subsubsection]{Proposition}
\newtheorem{lem}[subsubsection]{Lemma}

\newtheorem{lem-def}[subsubsection]{Lemma-Definition}
\newtheorem{cor}[subsubsection]{Corollary}
\newtheorem{theorem}[subsubsection]{Theorem}
\newtheorem{proposition}[subsubsection]{Proposition}
\newtheorem{lemma}[subsubsection]{Lemma}

\theoremstyle{definition}
\newtheorem{example}[subsubsection]{Example}
\newtheorem{remark}[subsubsection]{Remark}
\theoremstyle{definition}
\newtheorem{assumption}[subsubsection]{Assumption}

\newtheorem{ex}[subsubsection]{Example}

\newtheorem{rmk}[subsubsection]{Remark}
\theoremstyle{definition}
\newtheorem{dfn}[subsubsection]{Definition}

\numberwithin{equation}{subsection}

\input xy
\xyoption{all}

\newcommand{\quash}[1]{}  
\newcommand{\nc}{\newcommand}
\nc{\on}{\operatorname}

\DeclareFontEncoding{OT2}{}{} 

\DeclareMathOperator{\Ker}{Ker}
\def\bbb{\mathbf{b}}

\def\AAA{\mathbb{A}}

\def\CC{\mathbb{C}}

\def\GG{\mathbb{G}}

\def\NN{\mathbb{N}}

\def\XX{\mathbb{X}}

\def\ZZ{\mathbb{Z}}

\def\calF{\mathcal{F}}

\def\calO{\mathcal{O}}

\def\bfJ{{\boldsymbol{J}}}

 \nc{\SL}{{\rm SL}}
 \nc{\hatQ}{{\hat Q}}
 \nc{\sgn}{{\rm sgn}}
 
 \nc{\seee}{\mathbb C}
 
 \newcommand{\id}{{\rm id}}

 \nc{\hatlambda}{{\hat\lambda}}
 
 \nc{\daggerlambda}{{\lambda^\dagger}}

\newcommand{\frakb}{{\mathfrak b}}

\newcommand{\frakd}{{\mathfrak d}}

\newcommand{\frakg}{{\mathfrak g}}

\newcommand{\frakk}{{\mathfrak k}}

\newcommand{\frakp}{{\mathfrak p}}

\newcommand{\frakt}{{\mathfrak t}}
\newcommand{\fraku}{{\mathfrak u}}
\newcommand{\frakv}{{\mathfrak v}}

\newcommand{\frakz}{{\mathfrak z}}

\newcommand{\frakM}{{\mathfrak M}}

\newcommand{\frakS}{{\mathfrak S}}

\newcommand{\frakX}{{\mathfrak X}}

\newcommand{\Sat}{\mathrm{Sat}}

\newcommand{\der}{\mathrm{der}}

\newcommand{\bA}{{\mathbb A}}
\newcommand{\bB}{{\mathbb B}}
\newcommand{\bC}{{\mathbb C}}

\newcommand{\bG}{{\mathbb G}}

\newcommand{\bM}{{\mathbb M}}
\newcommand{\bN}{{\mathbb N}}

\newcommand{\bP}{{\mathbb P}}
\newcommand{\bQ}{{\mathbb Q}}

\newcommand{\bV}{{\mathbb V}}

\newcommand{\bZ}{{\mathbb Z}}

\newcommand{\mA}{{\mathcal A}}
\newcommand{\mB}{{\mathcal B}}

\newcommand{\mE}{{\mathcal E}}

\newcommand{\mO}{{\mathcal O}}

\def\bfA{\mathbf{A}}
\def\sfA{\mathsf{A}}
\def\sfB{\mathsf{B}}
\def\sfC{\mathsf{C}}
\def\ttS{\mathtt{S}}
\def\ttW{\mathtt{W}}
\def\ttM{\mathtt{M}}

\newcommand{\red}{\mathrm{red}}

\newcommand{\fil}{\mathrm{fil}}

\newcommand{\Rep}{\mathrm{Rep}}
\nc{\Tate}{\mathrm{Tate}}

\nc{\al}{{\alpha}} \nc{\be}{{\beta}} \nc{\ga}{{\gamma}}
\nc{\ve}{{\varepsilon}} \nc{\Ga}{{\Gamma}} \nc{\la}{{\lambda}}
\nc{\La}{{\Lambda}}

\nc{\ad}{{\on{ad}}}
\newcommand{\Ad}{{\on{Ad}}}
\nc{\aff}{{\on{aff}}}
\nc{\Aff}{{\mathbf{Aff}}}
\newcommand{\Aut}{{\on{Aut}}}
\nc{\Bun}{{\on{Bun}}}
\newcommand{\cha}{\on{char}}
\nc{\Coh}{{\on{Coh}}}
\newcommand{\even}{\mathrm{even}}
\newcommand{\odd}{\mathrm{odd}}

\nc{\diag}{{\on{diag}}}
\nc{\dR}{{\on{dR}}}
\nc{\dist}{{\on{Dist}}}
\newcommand{\End}{{\on{End}}}
\nc{\Fl}{{\calF\ell}}

\newcommand{\Gr}{{\on{Gr}}}
\newcommand{\Hom}{{\on{Hom}}}
\nc{\Id}{{\on{Id}}}
\newcommand{\ind}{{\on{ind}}}
\nc{\Ind}{{\on{Ind}}}
\nc{\inv}{{\on{Inv}}}
\nc{\Iso}{{\on{Isom}}}
\newcommand{\Lie}{{\on{Lie}}}
\nc{\Nm}{{\on{Nm}}}

\nc{\pf}{{\on{pf}}}
\newcommand{\pr}{{\on{pr}}}
\nc{\rec}{{\on{rec}}}
\nc{\Iw}{\on{Iw}}
\newcommand{\Res}{{\on{Res}}}
\newcommand{\s}{{\on{sc}}}

\newcommand{\Spec}{\on{Spec}}

\nc{\tr}{{\on{tr}}}

\newcommand{\Mod}{{\mathrm{-Mod}}}
\newcommand{\ord}{\mathrm{ord}}

\newcommand{\gr}{\mathrm{gr}}
\newcommand{\Min}{\mathsf{Min}}

\newcommand{\GL}{{\on{GL}}}

\nc{\Spin}{\on{Spin}} \nc{\Pin}{\on{Pin}}
\nc{\SU}{{\on{SU}}} \nc{\SO}{{\on{SO}}}

\nc{\Ql}{{\overline{\bQ}_\ell}}
\newcommand{\reg}{\mathrm{reg}}

\nc{\fg}{\frakg}
\nc{\fp}{\frakp}

\nc{\rat}{\overline{\bQ}}
\nc{\triv}{{\bf{1}}}

\newcommand{\longto}{\longrightarrow}

\newcommand{\std}{\mathrm{std}}

\nc{\dm}{/\!\!/}

\def\xcoch{\mathbb{X}_\bullet}
\def\xch{\mathbb{X}^\bullet}

\nc{\wt}{\mathrm{wt}}
\nc{\Sh}{\on{Sh}}
\nc{\Sht}{\on{Sht}}
\nc{\wSh}{\widetilde{\Sht}}
\nc{\Sph}{\on{Sph}}
\nc{\Fr}{\on{Frob}}
\nc{\Fp}{{^\sigma\mE}}
\nc{\und}{\underline}
\nc{\mmu}{{\mu_\bullet}}
\nc{\nnu}{{\nu_\bullet}}
\nc{\Hk}{{\on{Hk}}}
\nc{\lhk}{\Hk^{\on{loc}}}

\nc{\bp}{{\mathbf{p}}}

\nc{\MV}{{\bM\bV}}

\newcommand{\xdashrightarrow}[2][]{\ext@arrow 3359 \rightarrowfill@@{#1}{#2}}
\def\rightarrowfill@@{\arrowfill@@\relax\relbar\shortrightarrow}
\def\arrowfill@@#1#2#3#4{%
  $\m@th\thickmuskip0mu\medmuskip\thickmuskip\thinmuskip\thickmuskip
   \relax#4#1
   \xleaders\hbox{$#4#2$}\hfill
   #3$%
}

\begin{document}
\title{On vector-valued twisted conjugation invariant functions on a group}
\author{Liang Xiao} 
\address{Liang Xiao, Department of Mathematics, University of Connecticut, Storrs, 341 Mansfield Road, Unit 1009, Storrs, CT 06269.}
\email{liang.xiao@uconn.edu}\date{}
\author{Xinwen Zhu}
\address{Xinwen Zhu, Department of Mathematics, Caltech, 1200 East California Boulevard, Pasadena, CA 91125.}
\email{xzhu@caltech.edu}\date{}
\dedicatory{To Joseph Bernstein with deep admiration}
\begin{abstract}
We study the space of vector-valued  (twisted) conjugation invariant functions on a connected reductive group.
\end{abstract}
\thanks{
L.X. was  partially supported by Simons Collaboration grant \#278433 and NSF grant DMS--1502147
. X.Z. was partially supported by NSF grant
DMS--1602092 and the Sloan Fellowship.}
\subjclass[2010]{}
\maketitle
\setcounter{tocdepth}{1}
\tableofcontents

\section{Introduction}

We first describe the problems we are studying in this article. Let us fix an algebraically closed field $k$ throughout this article. Let $G$ be a connected simply-connected semisimple group over $k$, and let $k[G]$ denote its ring of regular functions. Consider the conjugation action of $G$ on itself, and 
let 
$\bfJ=k[G]^{c(G)}$ denote the space of conjugation invariant regular functions on $G$. More generally, for an algebraic representation $V$ of $G$, let
\[
\bfJ(V):=(k[G]\otimes V)^{c(G)}=\{f: G\to V\mid f(gxg^{-1}) = g\cdot f(x),\quad x,g\in G\}
\]
denote the $\bfJ$-module of $V$-valued conjugation invariant (algebraic) functions on $G$, where $G$ acts on $k[G]\otimes V$ diagonally. There are also similarly defined
rings of invariants $\bfJ_0$ and $\bfJ_+$ and the corresponding modules $\bfJ_0(V)$ and $\bfJ_+(V)$, when $G$ is replaced by its asymptotic cone $\on{As}_G$ and the corresponding Vinberg monoid $V_G$. Our first result is as follows.

\begin{thm}
\label{T:main theorem}
Assume that $V$ admits a good filtration. Then the module $\bfJ(V)$ (resp. $\bfJ_+(V)$, resp. $\bfJ_0(V)$) is finite free over $\bfJ$ (resp. $\bfJ_+$, resp. $\bfJ_0$) of rank $\dim V(0)$, where $V(0)$ denotes the zero weight space of $V$ (with respect to a maximal torus of $G$). In addition, if $\cha k=0$, there are natural bases of $\bfJ(V)$, $\bfJ_+(V)$, and $\bfJ_0(V)$ determined by certain bases of $V$.
\end{thm} 
We refer to \S~\ref{SS: basis} for detailed discussions regarding to the last statement. This theorem resembles the classical result of Kostant (\cite{Ko}): Assume $\cha k=0$, and let $\frakg$ be the Lie algebra of $G$. Then $(k[\frakg]\otimes V)^{G}$ is a graded free $k[\frakg]^{G}$-module with a basis given by harmonic polynomials. Recall that the study of $(k[\frakg]\otimes V)^{G}$ leads to a natural filtration on the weight spaces (the Kostant--Brylinski filtration), which can be studied either via the geometry of the flag variety (as in \cite{Bry}) or via the geometric Satake correspondence (as in \cite{Gi}). Likewise, our proof of the above theorem leads us to define a multi-filtration on the weight spaces (see \S~\ref{S:the filtration}), which can also be studied either via the geometry of the flag variety (\S~\ref{SS:fil via BW}) or via the geometric Satake correspondence (\S~\ref{SS: fil via geomSat}). In particular, we obtain a multivariable analogue of the weight multiplicity $P_{\mu,\nu}(q)$ of a representation (see Remark~\ref{R:mult weight}). Note that Kostant's theorem has a very simple proof (\cite{BL96}). It would be interesting to see whether there is a similar argument in the group case (at least when $\cha k=0$).

Note that finite freeness of $\bfJ(V)$ over $\bfJ$ was previously known to Richardson (see \cite{Richardson}) when $\on{char} k =0$ and Donkin (see \cite{Do3} and the appendix) in positive characteristic, under the same assumption on $V$, but by a different method\footnote{For example, the method in \emph{loc. cit.} requires the flatness of the Chevalley map $G\to G/\!\!/c(G)$ as an input, whereas our approach could deduce the flatness as a corollary (see Corollary~\ref{L:cl flat}).}.

\medskip
Now let $V^*$ be the dual representation of $V$, which is also assumed to admit a good filtration. Then there is a natural $\bfJ$-bilinear pairing
\begin{equation}
\label{E:pairing of J}
\bfJ(V)\otimes\bfJ(V^*)\to \bfJ
\end{equation} 
induced by pairing $V$-valued functions with $V^*$-valued functions. One can show that this pairing is a certain deformation of the natural pairing between $V(0)$ and $V^*(0)$ (see Remark~\ref{R: pairing} (2)). Our main result (Theorem~\ref{T:determinant of intersection}) calculates the determinant of this pairing as an element in $\bfJ$ up to a unit (or more precisely as a divisor on $\Spec \bfJ$.) We choose a maximal torus $T$ of $G$. Let $\Phi(G,T)$ denote the root system of $G$.
For a root $\al$ of $G$, let $e^\al$ denote the corresponding character function on $T$. For a weight $\la$ of $T$, let $V(\la)$ denote the corresponding weight space.
Using the classical Chevalley isomorphism, we identify $\bfJ=k[G]^{c(G)}$ with the Weyl group invariant functions on $T$. 
\begin{thm}
Assume that $\cha k>2$.
Then the determinant of the pairing \eqref{E:pairing of J} is 
\[c\prod_{\al\in\Phi(G,T)}(e^{\al}-1)^{\zeta_\al},\]
where $c$ is a non-zero constant and $\zeta_\al=\sum_{n\geq 1} \dim V(n\al)$.
\end{thm}

As mentioned above, when $\cha k=0$, it is possible to construct bases of $\bfJ(V)$ and $\bfJ(V^*)$ so the pairing \eqref{E:pairing of J} can be represented by a matrix. We refer to \S~\ref{SS:example of matrix} for some examples where the matrices are  calculated explicitly. 
Our main interest in these matrices (and their determinants) lies in the fact that under the Satake isomorphism they exactly correspond to the intersection matrices for certain cycles on the mod $p$ fibers of some Shimura varieties; we refer to
\cite{XZ} for more details.

\medskip
In fact, in this article, we will consider a more general situation.
Assume that $\tau$ is an automorphism of the algebraic group $G$. We consider the $\tau$-twisted conjugation action of $G$ on itself
\begin{equation}
\label{E:Ginvfun}
c_\tau(h) ( g) =  h g\tau( h)^{-1} , \quad \textrm{for }  g,h \in G.
\end{equation}
This is equivalent to considering the usual conjugation action of $ G$ on the coset $ G\tau$ inside the semidirect product $ G \rtimes \langle\tau \rangle$. We can similarly define
$\bfJ(V):=(k[G]\otimes V)^{c_\tau(G)}$ as the space of $V$-valued functions $f$ on $G$ satisfying
$f(gx\tau(g)^{-1}) = g\cdot f(x),\quad x,g\in G$, 
and in particular $\bfJ:=k[G]^{c_\tau(G)}$.
We prove the above mentioned results for general $\tau$ (see Theorem~\ref{T:Kostant} and Theorem~\ref{T:determinant of intersection} for the general statements). 
This generality is needed in \cite{XZ} for applications, but will cause some additional complications for this article. Readers who are happy with $\tau=\id$ can skip \S~\ref{S:automorphism} and parts of \S~\ref{S:vector-valued twisted invariant functions} and  \S~\ref{S:determinant}.
  
We expect that results in this article can be generalized to the case of symmetric pairs. In addition, by reformulating the argument purely algebraically, it is likely that they also extend to quantum groups. 

\medskip
We briefly describe the remaining parts of the note. In \S~\ref{S: fil vect and Rees}, we discuss multi-filtered vector spaces and the corresponding Rees construction  in some generality. Unlike the classical situation, the dimension of the associated graded vector space might be larger than the dimension of the original vector space. Whether the two dimensions are equal is a subtle question and is closely related to the flatness of certain Rees modules. We give some sufficient conditions for the equality, which might be of independent interest. 

In \S~\ref{SS:connon}, we define the canonical filtration on a $G$-representations and study the associated Rees modules. We apply these discussions to the Vinberg monoid in \S~\ref{SS: Vin}.

In \S~\ref{S:automorphism}, we discuss twisted conjugacy classes of $G$ for the action \eqref{E:Ginvfun}. In particular, we discuss the notion of $\tau$-regular conjugacy class, and study the (twisted) Chevalley map and the Grothendieck--Springer resolution, generalizing some well-known results for $\tau=\id$. 

In \S~\ref{SS:endoring}, we study $\bfJ(V\otimes V^*)$. It is naturally the space of endomorphisms of a vector bundle $\widetilde V$ on the quotient stack $[G/c_\tau(G)]$ for the $\tau$-twisted conjugation action of $G$ on itself, and therefore is a $\bfJ$-algebra. This is non-commutative in general. But we will study a commutative subalgebra generated by a ``tautological" element in $\ga_{\on{taut}}\in \bfJ(V\otimes V^*)$.

\subsection*{Acknowledgments}
Some of the work was finished when both authors visited Beijing International Center of Mathematical Research, we thank the staff members for the hospitality. We thank Xuhua He, George Lusztig, and Jun Yu for useful discussions, Stephen Donkin for useful comments and the referee for carefully reading the early version of this article.

\subsection*{Notations and conventions}
\label{SS: NC}
Throughout the note, let $k$ denote an algebraically closed field. By a variety over $k$, we mean a separated, integral scheme of finite type over $k$.
If $X$ is an affine algebraic variety over $k$, let $k[X]$ denote the ring of regular functions on $X$. If $\frakX$ is an algebraic stack of finite presentation over $k$, let $\Coh(\frakX)$ denote the category of coherent sheaves on $\frakX$. 

For an algebraic group $H$, let $\bB H=[\Spec k/H]$ denote the classifying stack. We identify $\Coh(\bB H)$ with the category $\Rep^f(H)$ of finite dimensional representations of $H$, and use them interchangeably. In particular, the trivial representation, sometimes denoted by $\mathbf{1}$ corresponds to the structure sheaf of $\bB H$. If $X$ is an $H$-space and $V$ is an $H$-representation, let 
$V_{[X/H]}:=X\times^HV$ 
be the locally free sheaf on $[X/H]$ via the usual associated construction. Alternatively, it is the pullback of $V$ (as coherent sheaf on $\bB H$) along the natural projection $[X/H]\to \bB H$. 

For a reductive group $H$ acting on an affine variety $X$ over $k$, we write $X/\!\!/H : = \Spec k[X]^H$ for the GIT quotient. There is a natural morphism of stacks $[X / H] \to X/\!\!/ H$. 

The ind-completion $\Rep(H)$ of $\Rep^f(H)$ is the category of algebraic representations of $H$. If $H_1\subset H$ is a closed subgroup, and $V$ an $H_1$-representation, let
\[\on{ind}_{H_1}^HV:= (k[H]\otimes V)^{H_1}=\Gamma(H/H_1,V_{H/H_1})\]
be the induced $H$-representation. It is the \emph{right} adjoint of the restriction functor $\Res_{H_1}^H$ from $H$-representations to $H_1$-representations. If no confusion will likely arise, for an $H$-representation $V$, we write $V$ for $\Res_{H_1}^HV$ for simplicity.

For an algebraic group $H$, let $\on{Dist}(H)$ denote the Hopf algebra of invariant distributions on $H$. Sometimes, $\on{Dist}(H)$ is also called the hyperalgebra of $H$.  

Let $X$ be a set (or scheme) with an automorphism $\tau$. We denote by $X^\tau$ the subset (subscheme) of fixed points. If $X$ is an abelian group, let $X_\tau$ denote its group of coinvariants.

Throughout this note, $\NN$ stands for the set of \emph{nonnegative} integers.

\medskip

\section{Filtered vector spaces and Rees modules}
\label{S: fil vect and Rees}
\subsection{Filtered vector spaces}
\label{S:filtered vector space}
We need to first discuss (multi)-filtered vector spaces and (multi)-graded modules. 
It seems best to start with the following setup. Let $(S,\leq)$ be a partially ordered set (or sometimes called a poset for simplicity). 
As usual, we write $s_1<s_2$ if $s_1\leq s_2$ but $s_1\neq s_2$.  
Recall that $S$ is called \emph{directed} if for $s,s'\in S$, there exists $s''\in S$ such that $s\leq s''$ and $s'\leq s''$.
We say a partially ordered set \emph{poly-directed} if $S=\sqcup_i S_i$ is a disjoint union of directed poset $S_i$, and if $s\in S_i$ and $s'\in S_j$ are incomparable if $i\neq j$.

\begin{dfn}
Assume that $S$ is poly-directed.
We define an \emph{$S$-filtered vector space} (or a vector space with an \emph{$S$-filtration}) to be a vector space $M$, equipped with a collection of subspace $\{\on{fil}_s M, s\in S\}$ such that $\on{fil}_{s} M\subset \on{fil}_{s'}M$ if $s\leq s'$, and that
\begin{equation}
\label{E:weak filtration}
M = \bigoplus_{S_i}\bigcup_{s \in S_i} \fil_s M.
\end{equation}
Let $\on{Vect}^{S-\on{fil}}$ denote the category of $S$-filtered vector spaces,  with morphisms given by $k$-linear maps preserving the filtrations. Similarly,  we define an \emph{$S$-graded vector space} to be a vector space $M$ equipped with a direct sum decomposition $M=\oplus_{s\in S}M_s$ indexed by $S$. Let $\on{Vect}^{S-\on{gr}}$ denote the corresponding category. There is a natural functor
\begin{equation*}
\label{E:graded quotient}
\gr: \on{Vect}^{S-\on{fil}}\to \on{Vect}^{S-\on{gr}},\quad M\mapsto \gr_S M=\bigoplus_{s\in S}\gr_s M,\quad \gr_sM=\on{fil}_sM\big/\sum_{s'< s}\on{fil}_{s'}M.
\end{equation*}
In the above definition, if $s$ is \emph{minimal} in $S$, i.e. the set $\{s'\in S\mid s'<s\}$ is empty, then we set $\sum_{s'< s}\on{fil}_{s'}M=0$ so $\gr_sM=\fil_sM$. 
If $S$ is clear from the context, we write $\gr M$ for $\gr_SM$ for simplicity. 

We say a nonzero element $m\in \gr M$ is of \emph{homogeneous} of degree $s$ if $m\in \gr_s M$. For a homogeneous element $m$ of degree $s$, a \emph{lifting} of $m$ is an element $\widetilde m\in\on{fil}_sM$ whose projection to $\gr_s M$ is $m$.
\end{dfn}

\begin{ex}
\label{E: classical and mult fil}
(1) If $S=\bN$ equipped with the natural order, an $S$-filtered vector space is a vector space $M$ equipped with an increasing filtration $\on{fil}_0M\subset\on{fil}_1M\subset\cdots$ in the usual sense such that $M=\cup_s \on{fil}_sM$. In this case, for every non-zero element $m$, there is a unique $s$ such that $m\in \on{fil}_sM-\on{fil}_{s-1}M$ and its projection to $\gr_sM$ is called the \emph{symbol} of $m$, denoted by $\overline{m}$.

(2) Assume $S=\bN^r$ with the partial order $(s_1,\ldots,s_r)\leq (s'_1,\ldots,s'_r)$ if $s_j\leq s'_j$ for every $j$.
Let $M$ be a vector space equipped with $r$ independent increasing filtrations 
$\on{fil}^j_0M\subset \on{fil}^j_1M\subset \cdots$ such that $\bigcup_s \fil_s^jM =M$, for each $j=1,\ldots, r$. Then one can define an $S$-filtration of $M$  as follows: for $s=(s_1,\dots, s_r)$, let $\on{fil}_sM=\cap_j \on{fil}^j_{s_j}M$.
\end{ex}

The partially ordered sets we encounter in this article will satisfy the following descending chain condition
\begin{center}
(DCC)  \quad\quad\quad\quad\quad  Every descending chain $s_0>s_1>\cdots$ is finite. 
\end{center}
It is clear that the condition (DCC) passes to subsets of a partially order subset.
\medskip

The following lemma is usually referred to as the \emph{graded Nakayama lemma}.
\begin{lem}
\label{L: graded Nakayama}
Let $S$ be a poly-directed poset satisfying \emph{(DCC)}. Let $M$ be an $S$-filtered vector space. 
\begin{enumerate}
\item If $\gr M= 0$, then $M= 0$.
\item Let $\{m_i\}$ be a set of homogeneous elements that span $\gr M$, and let $\{\widetilde{m_i}\}$ be a set of liftings. Then $\{\widetilde{m_i}\}$ span $M$. In particular, $\dim M\leq \dim \gr M$.
\end{enumerate}
\end{lem}
\begin{proof}
(1) Assume $M\neq 0$. Then there should exist some $s_1\in S$ such that $\on{fil}_{s_1}M\neq 0$. Since $\gr M=0$, there should be another $s_2<s_1$ such $\on{fil}_{s_2}M\neq 0$. In this way, one could produce a descending chain $s_1>s_2>\cdots$ of infinite length. Contradiction.

(2) We may assume that $S$ is directed. Let $\widetilde M\subset M$ be the subspace spanned by $\{\widetilde{m_i}\}$, and let $M'=M/\widetilde M$, equipped with an $S$-filtration defined as $\on{fil}_sM'=\on{Im}(\on{fil}_sM\to M')$. By construction, $\gr_SM'=0$. Therefore $M'=0$ by Part (1).
\end{proof}

\begin{rmk}
\label{E: Rees dimension jump}
In the classical situation, i.e. Example~\ref{E: classical and mult fil} (1), it is always true that $\dim \gr M=\dim M$. One can also show that this is the case if $S$ is as in Example~\ref{E: classical and mult fil} (2) with $r=2$. 
However, the inequality $\dim \gr M\geq \dim M$ could be strict in general. For example, take $S$ as in Example~\ref{E: classical and mult fil} (2) with $r=3$ and let $M$ be a $2$-dimensional vector space equipped with three filtrations $0=\on{fil}^j_0M\subset \on{fil}^j_1M\subset \on{fil}^j_2M=M$, $j=1,2, 3$ such that $\on{fil}^1_1M,\on{fil}^2_1M, \on{fil}^3_1M$ are three different lines in $M$.
Then if we define the $S$-filtration on $M$ as in Example~\ref{E: classical and mult fil} (2), the graded spaces $\gr_{(1, 2, 2)}M,\gr_{(2, 1, 2)}M, \gr_{(2,2, 1)}M$ are all nontrivial. It then follows that $\dim \gr M=3> 2 = \dim M$.
\end{rmk}

As we shall see soon, whether $\dim M=\dim \gr M$ is closely related to the flatness of certain Rees module. So it would be desirable to have some criteria for the equality. Here is one which is a direct consequence of the above lemma. 
\begin{cor}
\label{C: dim equal}
Let $(S, M)$ be as in Lemma~\ref{L: graded Nakayama}. If 
\begin{enumerate}
\item[(*)] there is a basis $\{m_i\}$ of $\gr M$ consisting of homogeneous elements (such a basis is called a \emph{homogeneous basis}), and for each $i$ a lifting $\widetilde{m_i}$ of $m_i$, such that $\{\widetilde{m_i}\}$ form a basis of $M$, 
\end{enumerate}
then $\dim M=\dim \gr M$. 

Conversely, if $\dim M=\dim \gr M<\infty$, then for every homogeneous basis of $\gr M$, any set of liftings of elements in this basis to elements in $M$ form a basis of $M$.
\end{cor}

A particular situation where the above criterion is applicable is as follows.

\begin{lem}
\label{L: bases criterion}
Consider an $S$-filtered module $M$ as in Example~\ref{E: classical and mult fil} (2).
Assume that there is a basis $\bB=\{m_i\}$ of $M$ such that for each $j$, the corresponding symbols $\{\overline{m_i}^j\}$ for the filtration $\{\on{fil}^j\}$ form a basis of $\gr^jM=\oplus_{d}\on{fil}^j_{d}M/\on{fil}^j_{d-1}M$. Then $\dim M=\dim \gr M$.
\end{lem}
\begin{proof}
For $m_i \in \bB$, define its multidegree to be $\underline d(m_i):=(d_1(m_i),\ldots,d_r(m_i))\in\bN^r$ such that with respect to the $j$th filtration $\on{fil}^j_\bullet M$, $m_i\in \on{fil}^j_{d_j(m_i)}M-\on{fil}^j_{d_j(m_i)-1}M$. Given $\underline{d}=(d_1,\ldots,d_r)$, let
\[\bB_{\underline d}=\{m_i\in\bB\mid d_j(m_i)=d_j\},\quad \bB_{\leq \underline d}=\sqcup_{\underline d'\leq \underline d}\bB_{\underline d'}.\]
Then it is easy to show that $\bB_{\leq \underline d}$ form a basis of $\on{fil}_{\underline d}M$. Indeed, let $m\in\on{fil}_{\underline d}M$, and write $m=\sum a_im_i$ in terms of the basis $\bB$. If not every element in $\{\underline d(m_i)\mid a_i\neq 0\}$ is $\leq \underline d$, we can find some $1\leq j\leq r$ such that $d'_j:=\max\{d_j(m_i)\mid a_i\neq 0\}> d_j$. Then projecting to $\gr^j_{d'_j}M$ gives $0=\sum a_i\overline{m_i}^j$ for the sum over those $m_i$ with $d_j(m_i)=d'_j$. This contradicts our assumption.

It follows that the projection of elements in $\bB_{\underline d}$ in $\gr_{\underline d}M$ form a basis of $\gr_{\underline d}M$. The lemma follows from Corollary~\ref{C: dim equal}.
\end{proof}

We will also  use of the following lemma.
\begin{lem}
\label{L: sub-fil}
Let $(S,M,\{\on{fil}_sM\}_{s\in S})$ be an $S$-filtered vector as above (in particular $S$ is poly-directed),  and assume that $S$ satisfies \emph{(DCC)}. Let $S'\subset S$ be a subset.
Assume for every $s\in S$, the set $\{s'\in S'\mid s\leq s'\}$ is nonempty and has a unique least element, denoted by $s^h$. Then $S'$ with the induced partial order is also poly-directed, and also satisfies (DCC).
In addition, $\dim \gr_{S'}M\leq \dim \gr_S M$, where $\gr_{S'}M$ denotes the associated graded of $M$ with respect to the sub-filtration $\{\on{fil}_sM\}_{s\in S'}$. In particular, if $\dim M=\dim \gr_SM$, then $\dim M=\dim \gr_{S'}M$.
\end{lem}
\begin{proof}
The statement for $S'$ is clear. We prove the inequality. Then last equality follows from it by Lemma~\ref{L: graded Nakayama}.

If $s\in S'$, the $s$-graded piece in $\gr_S M$ is denote by $\gr_sM$ while the $s$-graded piece in $\gr_{S'}M$ is denoted by $\gr'_s M$.
For $t\in S'$, let $S_t=\{s\in S\mid s^h=t\}\subset S$. Note that this is a partially ordered set with the greatest element $t$ and $S=\sqcup_{t\in S'} S_t$.
We can define a filtration on $\gr'_tM$ labeled by $S_t$ as $\on{fil}_s \gr'_tM=\on{Im}(\on{fil}_s M\to \gr'_tM)$. Note that $\on{fil}_s\gr'_tM\subset \on{fil}_{s'}\gr'_tM$ if $s\leq s'$. Let 
\[\gr_s\gr'_tM=\on{fil}_s\gr'_tM\Big/\!\!\sum_{s'<s,\, s'\in S_t}\on{fil}_{s'}\gr'_tM.\]
Note that $S_t$ also satisfies (DCC), so by Lemma~\ref{L: graded Nakayama}, $\sum_{s\in S_t} \dim\gr_s\gr'_tM\geq  \dim \gr'_tM$.
On the other hand, for a fixed $s \in S_t$,  if $s'<s$ and ${s'}^h\neq t$, then ${s'}^h<t$ and therefore the image of $\on{fil}_{s'}M\to \gr'_tM$ is zero. It follows that $\gr_s\gr'_tM$ is a quotient of $\gr_sM=\on{fil}_sM/\sum_{s'<s}\on{fil}_{s'}M$. Therefore, it follows that for every $t\in S'$,
\[\sum _{s\in S_t}\dim\gr_sM\geq \sum _{s\in S_t}\dim\gr_s\gr'_tM\geq  \dim \gr'_tM,\]
giving $\dim \gr_S M\geq \dim \gr_{S'} M$.
\end{proof}

\subsubsection{Monoid and partially ordered set}
\label{SS: monoid and partial order}
The partially ordered sets we encounter in this paper will mostly arise in the following way.
Let $\Gamma$ be a commutative monoid, and let $S$ be a non-empty set with a $\Gamma$-action (where multiplication and the action map will be written additively).
Consider the following condition
\begin{itemize}
\item[(Can)] The only invertible element in $\Gamma$ is the unit element (such a monoid is called \emph{sharp}), and the action of $\Ga$ on $S$ is free, i.e.  for $s\in S$ and $ \ga,\ga'\in\Ga$, the identity $s+\ga=s+\ga'$ implies $\ga=\ga'$.
\end{itemize}
This condition in particular implies that $\Ga$ is \emph{integral}, i.e. the map from $\Ga$ to its group completion $\Ga^{\on{gp}}$ is injective. In addition, there is a well-defined partial order on $S$ defined by $s\leq s'$ if $s'=s+\ga$ for some $\ga\in \Ga$. Note that since $s_1\leq s_2$ and $s_2\leq s_1$ will imply $s_1=s_2$, this is indeed a partial order. 

\begin{lem}
The set $S$ equipped with the above partial order is poly-directed.
\end{lem}
\begin{proof}The only property of this partial order we need is the following: if $s\leq s_1$ and $s\leq s_2$, then there is $s'$ such that $s_1\leq s'$ and $s_2\leq s'$. Indeed, by definition $s_1=s+\ga_1$ and $s=s+\ga_2$. Then we can choose $s'=s+\ga_1+\ga_2$. We show that any partially ordered set $S$ satisfying this property is poly-directed.

Consider the equivalence relation on $S$ generated by the partial order. It is enough to show that every equivalence class equipped with the induced partial order is directed. But if $s,s'$ are in one equivalence class, then there exists a chain of elements $s=s_0,s_1,\ldots,s_r=s'\in S$ such that either $s_i\leq s_{i+1}\geq s_{i+2}$ or $s_{i}\geq s_{i+1}\leq s_{i+2}$. By replacing the second relation by $s_i\leq s'_{i+1}\geq s_{i+2}$ repeatedly, one find $s''$ such that $s\leq s''\geq s'$. 
\end{proof}

\subsection{Rees construction}
\label{SS:filtered ring}
Let $\Ga$ be a commutative monoid acting on a set $S$. Let $R=k[\Gamma]$ be the monoid algebra. For $\ga\in\Gamma$, the corresponding element in $R$ is denoted by $e^\ga$. Let $I_1\subset R$ be the ideal generated by $e^\ga-1$ for all $\ga\in \Ga$, and let $I_0\subset R$ be the ideal spanned by $e^\ga$ for all $0\neq \ga\in\Ga$. 

We define an \emph{$S$-graded $R$-module} to be a $k$-vector space $N$ with a direct sum decomposition $N=\oplus_{s\in S}N_s$, such that $e^\ga N_s\subset N_{s+\ga}$. They naturally form a category, with morphisms given by $R$-module homomorphisms preserving the grading, denoted by $R\Mod^{S-\gr}$. There is a natural functor 
$$R\Mod^{S-\on{gr}}\to \on{Vect}^{S-\on{gr}}, \quad N\mapsto N/I_0N,$$
where $(N/I_0N)_s=N_s/\sum_{\ga'+s'=s} e^{\ga'}N_{s'}$. Given a homogeneous element $n\in N/I_0N$ of degree $s$, \emph{a homogeneous lifting} of $n$ to $N$ is an element $\widetilde{n}\in N_s$ which projects to $n$.

If $(\Ga, S)$ satisfies (Can), there is also a natural functor 
$$R\Mod^{S-\on{gr}}\to \on{Vect}^{S-\on{fil}}, \quad N\mapsto N/I_1N,$$
where we define the $S$-filtration by $\on{fil}_s(N/I_1N)=\on{Im}(N_s\to N\to N/I_1N)$.
The following lemma is another version of the graded Nakayama lemma. 
\begin{lem}
\label{L: graded Nakayama2}
Let $(\Ga,S)$ be as above and assume that \emph{(Can)} and \emph{(DCC)} hold for $(\Ga,S)$. Let $R=k[\Gamma]$. Let $N$ be an $S$-graded $R$-modules.
\begin{enumerate}
\item If $N/I_0N=0$, then $N=0$.
\item Let $\{n_i\}$ be a set of homogeneous elements that span $N/I_0N$, and let $\{\widetilde{n_i}\}$ be a set of homogeneous liftings. Then $\{\widetilde{n_i}\}$ generates $N$ as an $R$-module.
\item Assume that in addition, the images of $\{\widetilde{n_i}\}$ in $N/I_1N$ form a basis of $N/I_1N$. Then $N\cong \oplus_i R\widetilde{n_i}$. In particular, $N$ is a free $R$-module.
\end{enumerate}
\end{lem}
\begin{proof}The proof of Part (1) and (2) is the same as the proof of Lemma~\ref{L: graded Nakayama}. For Part (3), let $\sum r_i\widetilde{n_i}=0$ be a homogeneous linear relation in $N$, and write $r_i=a_ie^{\ga_i}$. Then in $N/I_1N$, there is a linear relation $\sum a_i\widetilde{n_i}=0$, which implies all $a_i=0$.
\end{proof}

We continue to assume that $(\Ga, S)$ satisfies (Can). Then the functor $N\mapsto N/I_1N$ as above admits a right adjoint
$$\on{Vect}^{S-\on{fil}}\to R\Mod^{S-\on{gr}},\quad M\mapsto R_SM=\bigoplus_{s\in S} \on{fil}_s M,$$
where we define the multiplication $e^{\ga}: \on{fil}_{s} M\to \on{fil}_{s+\ga} M$ (i.e. the $R$-module structure) to be the natural inclusion. The functor $M\to R_SM$ sometimes is called the \emph{Rees construction} and $R_SM$ is called the \emph{Rees module}. It is convenient to introduce the formal symbols $e^s, \ s\in S$ and write elements in the degree $s$ summand in $R_SM$ as $me^s$ for $m\in\on{fil}_sM$. Then the $R$-module structure on $R_SM$ is given by $e^\ga\cdot me^s=me^{s+\ga}$.

We need some criteria of flatness of $R_SM$ over $R$. Note that if it is the case, then by \eqref{E: Rees fiber dim} $\dim \gr M$ should be equal to $\dim M$ (if both are finite). It turns out that under some mild assumptions on $M$, this condition is also sufficient. 
First, note that
\begin{equation}
\label{E: Rees fiber dim}
R_SM/I_0R_SM\cong \gr M, \quad R_SM/I_1R_SM \cong M,
\end{equation}
where the second isomorphism makes use of \eqref{E:weak filtration}.

\begin{lem}
\label{L: criterion free Rees}
Let $(\Ga,S,M,\{\on{fil}_sM\}_{s\in S})$ be as above and assume that Conditions \emph{(Can)} and \emph{(DCC)} hold for $(\Ga,S)$. Let $R_SM$ be the corresponding Rees module. Assume that Condition \emph{(*)} in Corollary~\ref{C: dim equal} holds.
Then $R_SM$ is free over $R$ with basis $\{\widetilde{m_i}e^{s_i}\}$. In particular, if $\dim \gr M=\dim M<\infty$, then $R_SM$ is free over $R$.
\end{lem}
\begin{proof} This follows from Lemma~\ref{L: graded Nakayama2} and \eqref{E: Rees fiber dim}.\end{proof}

\subsubsection{Rees algebra} 
\label{SSS: rees alg}
Let $(\Ga, S, M, \{\on{fil}_sM\}_{s\in S})$ be as before, where $(\Ga,S)$ satisfy condition (Can).
If $S$ itself is a commutative monoid such that the action of $\Ga$ on $S$ is induced by a monoid homomorphism $f:\Ga\to S$ (the action of $S$ on itself is the natural translation), and if $M$ is a (not necessarily commutative) $k$-algebra, then it makes sense to assume that the filtration $\{\on{fil}_sM\}_{s\in S}$ satisfies the additional condition
\begin{itemize}
\item[(Alg)] $k\subset \on{fil}_0M$ and $\on{fil}_sM\cdot\on{fil}_{s'}M\subset \on{fil}_{s+s'}M$.
\end{itemize}
In this case,  $R_SM$ is naturally a (not necessarily commutative) algebra over $R$, with the multiplication given by $me^s\cdot m'e^{s'}=mm'e^{s+s'}$, and the map $R\to R_SM$ is given by $e^\ga\mapsto 1\cdot e^{f(\ga)}$. We call it the \emph{Rees algebra} of $M$.

If in addition, $M$ has a co-algebra structure such that each $\on{fil}_sM$ is a sub-coalgebra, then $R_SM$ is also a coalgebra.

\section{Filtration on representations}
\label{S: filtration on weight space}
In this section, let $G$ be a connected reductive group over $k$. We discuss the canonical filtration on $G$-representations and the associated Rees modules in \S~\ref{SS:connon},  and apply discussions to the Vinberg monoid in \S~\ref{SS: Vin}.
In \S~\ref{S:the filtration}--\ref{SS: fil via geomSat}, we define and study a multi-filtration on the weight spaces of a representation of $G$. Just as the Kostant--Brylinski filtration plays an important role in the study of vector-valued invariant functions on $\frakg$, this filtration is important for the study of vector-valued invariant functions on the group. We will make use of the following notations and conventions throughout this section.

We fix a maximal torus $T$ of $G$, contained in a Borel subgroup $B$. Let $U$ denote the unipotent radical of $B$.
Let $\Phi=\Phi(G, T)$ denote the root system, and $\Delta\subset \Phi$ the set of simple roots with respect to $B$. For every $\al\in \Phi$, we fix an isomorphism $x_\al: \bG_a \simeq U_\al $, where $U_\al$ is the root subgroup corresponding to $\al$. 
The tuple $(G,B,T,\{x_\al\}_{\al\in\Delta})$ forms a pinning of $G$.

The Lie algebra $\fraku_\al$ of $U_\al$ is spanned over $k$ by $E_\alpha: = \frac{d}{dy_\alpha}$ (here we regard $y_\al: = x_\al^{-1}$ as a coordinate function on $U_\al$). More generally, the algebra of invariant distributions $\on{Dist}(U_\al)$ of $U_\al$ is spanned over $k$ by $\{E_\al^{(n)},\ n\geq 0\}$, where 
$E_\al^{(n)}(y_\al^j)=\begin{pmatrix}j\\ n\end{pmatrix} y_\al^{j-n}$. In particular, $E_\al^{(n)}(y_\al^j)=0$ if $n>j$.
Informally, we may think $E_\al^{(n)}=E_\al^n/n!$.  

Let $B^-$ be the opposite Borel with respect to $(T,B)$ and $U^-$ its unipotent radical.  Let $N=N_G(T)$ be the normalizer of $T$ in $G$, and let $W=N/T$ denote the (absolute) Weyl group. Let $w_0\in W$ be the longest Weyl group element. Let $\mu\mapsto \mu^*:=-w_0(\mu)$ be the involution on the character lattice $\xch(T)$, which preserves the set of dominant weights $\xch(T)^+$ (with respect to $B$). 
When we regard a weight $\nu \in \XX^\bullet(T)$ as a regular function on $T$,  we write it as $e^\nu$.

Let $Z_G$ denote the scheme-theoretic center of $G$.
Let $T_\ad$ be the adjoint torus of $G$, i.e. the quotient of $T$ by $Z_G$. Its character lattice $\xch(T_\ad)$ is the subgroup of $\xch(T)$ generated by roots.
We write $\xch(T_{\on{ad}})_{\on{pos}}$ for the monoid of nonnegative integer linear combinations of simple roots in $\Delta$. We consider the partial order $\preceq$ on $\xch(T)$ induced by the action of $\xch(T_\ad)_{\on{pos}}$ in the sense of \S~\ref{SS: monoid and partial order}, i.e. $\la_1\preceq\la_2$ if and only if $\la_2-\la_1$ is a nonnegative integral linear combination of simple roots of $G$. 

For a root $\al$, let $G_\al$ be the rank one subgroup of $G$ generated by $T, U_\al, U_{-\al}$. Let $B_\al=TU_\al$ and $ B^-_\al=TU_{-\al}$ be the pair of opposite Borel subgroups of $G_\al$. We similarly have the partial order $\preceq_\al$ on $\xch(T)$ induced by the action of $\bZ_{\geq 0}\al$.   We say $\la$ is \emph{$\al$-dominant} if $\langle\la,\al^\vee\rangle\geq 0$, where $\al^\vee$ is the coroot corresponding to $\al$. Note that if $0\preceq_{\al}\la$, then $\la$ is $\al$-dominant.

For a weight $\nu \in \XX^\bullet(T)$, let $k_\nu$ denote the corresponding one-dimensional $T$-module. For a representation $V$ of $T$ and $\nu \in \xch(T)$, we write $V(\nu)$ for the $\nu$-weight space, so $V(\nu)\cong \Hom_T(k_\nu,V)\otimes k_\nu$.

\subsection{The canonical filtration on $G$-modules}\label{SS:connon}
We first review Weyl and Schur modules.
 Via inflation, the $T$-module $k_\nu$ can be regarded as a representation of $B$ or $B^-$. Let
$$\mathtt S_\nu:=\on{ind}_{B^-}^G k_\nu \cong  \on{ind}_{B}^Gk_{w_0(\nu)}$$ 
be the \emph{Schur module} of highest weight $\nu$, and let
$$\mathtt W_\nu:=\mathtt S_{\nu^*}^*$$ 
denote the \emph{Weyl module} of highest weight $\nu$. More geometrically, 
we write 
$$\mO_{G/B}(\nu)=G\times^{B}k_\nu,\quad \mO_{G/B^-}(\nu)=G\times^{B^-}k_\nu$$ to denote the line bundle on the flag variety. Then
 $$\mathtt S_\nu=\Gamma(G/B^-, \mO(\nu))=\Gamma(G/B,\mO(w_0(\nu))).$$
It is known that $\ttS_\nu = \ttW_\nu = 0$ unless $\nu$ is dominant.

We call a dominant weight $\omega \in \XX^\bullet(T)^+$ \emph{minuscule} if all weights in $\ttS_\omega$ form a single orbit under the action of the Weyl group $W$. 
Note that in this case, the multiplicity of each weight space is one-dimension and $\ttS_\omega \cong \ttW_\omega$.  The set of minuscule weights is denoted by $\mathtt{Min}\subset \xch(T)^+$. Note that in our convention, the zero weight is minuscule.
\begin{lemma}
Let $\xch(T)^+_{\on{pos}}\subset \xch(T)$ be the submonoid generated by $\xch(T_\ad)_{\on{pos}}$ and $\xch(T)^+$. Then under the natural action of the monoid $\xch(T_\ad)_{\on{pos}}$, 
\begin{equation*}
\label{E:setS}
\xch(T)^+_{\on{pos}}=\bigsqcup_{\omega\in\mathtt{Min}} \big(\omega+\xch(T_\ad)_{\on{pos}}\big).
\end{equation*}
In particular, the pair $(\Ga,S)=(\xch(T_\ad)_{\on{pos}},\xch(T)^+_{\on{pos}})$ satisfies Conditions \emph{(DCC)} and \emph{(Can)} in the previous section. 
\end{lemma}
\begin{proof}
 It is known that the set $\mathtt{Min}$ gives a collection of coset representatives of the quotient $\XX^\bullet(T) / \XX^\bullet(T_{\on{ad}})$ so that every $\la\in\xch(T)$ can be uniquely written as $\la=\ga_\la+\omega_\la$ with $\ga_\la\in \xch(T_\ad)$ and $\omega_\la\in\mathtt{Min}$. In addition, if $\la\in \xch(T)^+$, then $\gamma_\la\in\xch(T_\ad)_{\on{pos}}$. The lemma then clearly follows.
\end{proof}

For a representation $V$ of $G$, we define an $\xch(T)^+_{\on{pos}}$-filtration on $V$, called \emph{the canonical filtration} of $V$,\footnote{This is closely related, but not the same as the notion of canonical filtration in \cite[\S 3]{Ma}. In particular, our definition is independent of any choice.} as follows. 
For $\la\in \xch(T)$, we denote by $V_{\preceq \la}$ the \emph{maximal} subrepresentation of $G$ such that $V_{\preceq \la}(\nu)\neq 0$ implies $\nu\preceq \la$. 
Clearly, $V_{\preceq\la}\neq 0$ only if $\la\in\xch(T)^+_{\on{pos}}$. Moreover, the functor $ V \mapsto V_{\preceq \lambda}$ is left exact. Therefore, we obtained the quadruple
\[(\Ga,S, M, \{\fil_sM\}_{s\in S})=\big(\xch(T_\ad)_{\on{pos}}, \xch(T)^+_{\on{pos}}, V, \{V_{\preceq\la}\}_{\la\in\xch(T)^+_{\on{pos}}}\big).\]
Let $R=k[\xch(T_\ad)_{\on{pos}}]$, and let $R_{\xch(T)^+_{\on{pos}}}V$ be the associated Rees module, which is an $\xch(T)^+_{\on{pos}}$-graded $R$-module. 
Note that the functor
\[G\Mod\to R\Mod^{\xch(T)^+_{\on{pos}}-\gr},\quad V\mapsto R_{\xch(T)^+_{\on{pos}}}V\]
is left exact. 

There is an important class of $G$-modules, whose Rees module associated to the canonical filtration is $R$-flat. 
Recall that a \emph{good filtration} of a representation $V$ of $G$ is a filtration of $V$ by $G$-submodules (in the classical sense as in Example~\ref{E: classical and mult fil} (1)) whose associated graded are Schur modules. We recall some important properties of this class of representations. 
 \begin{thm}
 \label{T:prop of good fil}
\begin{enumerate}
\item If $V$ admits a good filtration, then its restriction to every Levi subgroup $M\subset G$ also admits a good filtration (as an $M$-module).
\item The tensor product of two $G$-modules that admit a good filtration also admits a good filtration.
\item The following are equivalent.
\begin{itemize}
\item $V$ admits a good filtration.
\item  $\on{Ext}^i(\ttW_{\nu^*},V)=\on{H}^i(G,\ttS_{\nu}\otimes V)=0$ for every dominant weight $\nu$ and every $i>0$.  
\item $\on{Ext}^1(\ttW_{\nu^*},V)=\on{H}^1(G,\ttS_{\nu}\otimes V)=0$ for every dominant weight $\nu$. 
\end{itemize}
\item Regard $k[G]$ as a $G\times G$-bimodule via left and right translations. Then $k[G]$ admits a good filtration. For every dominant weight $\nu$, $\ttS_\nu\otimes\ttS_{\nu^*}$ appears in the composition factors of any good filtration of $k[G]$ exactly once.
\end{enumerate}
\end{thm}
Part (1) and (2) are due to Mathieu \cite{Ma} (and were already obtained earlier by Donkin \cite{Do1} in most cases) and (3) is due to Donkin \cite{Do}.  Part (4) is due to Donkin \cite{Do2} and independently Koppinen \cite{Kop}.
It follows easily from Part (3) that the number of  factors in the successive quotients of a good filtration of $V$  that are isomorphic to $\ttS_\nu$ is equal to $\dim \Hom_G(\ttW_{\nu},V)$, and therefore is independent of the choice of the good filtration. In particular, if $V$ is finite dimensional
\begin{equation}
\label{E: dim of V}
\dim V=\sum_{\nu} \dim \Hom_G(\ttW_{\nu},V) \cdot \dim \ttS_{\nu}.
\end{equation}

The main result of this subsection is as follows.
\begin{prop}
\label{P: good implies flat}
Let $V$ be a (finite dimensional) $G$-module that admit a good filtration. Then the Rees module $R_{\xch(T)^+_{\on{pos}}}V$ associated to the canonical filtration is a (finite) flat $R$-module. 
\end{prop}
\begin{proof}
As every $G$-module $V$ that admits a good filtration is the union of finite dimensional $G$-submodules $V_i$ that admit a good filtration, and as $R_{\xch(T)^+_{\on{pos}}}V$ is the union of $R_{\xch(T)^+_{\on{pos}}}V_i$, we may assume that $V$ is finite dimensional. Now by Lemma~\ref{L: criterion free Rees}, it suffices to show that $\dim V = \dim \gr V$. But this follows from \eqref{E: dim of V} and the following lemma.
\end{proof}
\begin{lem}
\label{L:good sub} 
Let $V$ be a representation of $G$ that admits a good filtration. Let $\la$ be a weight.
\begin{enumerate}
\item 
Both $V_{\preceq \la}\subset V$ and $V/V_{\preceq \la}$ admit good filtrations.
\item
Let $V_{\prec \la}=\sum_{\la'\preceq \la} V_{\preceq \la'}\subset V_{\preceq \la}$. Then $V_{\prec \la}$ admits a good filtration, and $V_{\preceq \la}/V_{\prec \la}$ is isomorphic to $\Hom_G(\ttW_{\la},V)\otimes \ttS_{\la}$. 
\end{enumerate}
\end{lem}
\begin{proof}
This is a special case of \cite[12.1.6]{Do1}. We include a proof for completeness.
(1)
By Theorem~\ref{T:prop of good fil} (3), it is enough to show that for every dominant $\nu$, the map
\begin{equation}
\label{E:weyl to sub}
\Hom_G(\ttW_\nu, V)\to \Hom_G(\ttW_\nu, V/V_{\preceq \la})
\end{equation}
is surjective. If $\nu\preceq \la$, then $\Hom_G(\ttW_\nu, V/V_{\preceq \la})=\Hom_G(\ttW_\nu, (V/ V_{\preceq \la})_{\preceq\la})=0$, since $(V/ V_{\preceq \la})_{\preceq\la}$ is evidently zero. Assume $\nu\not\preceq \la$. 
By Frobenius reciprocity, the map \eqref{E:weyl to sub} may be identified with the map $\Hom_B(k_\nu, V) \to \Hom_B(k_\nu, V / V_{\preceq \lambda})$. 
Let $L \subset V(\nu)$ be a line, such that $B$ acts on its image in $V/V_{\preceq \la}$ via the character $B\to T\xrightarrow{\nu}\bG_m$. Then $B$ acts on $L$ by the same way (as any weights $\succeq \nu$ does not appear in $V_{\preceq \lambda}$). Therefore \eqref{E:weyl to sub} is an isomorphism in this case.

(2) The same argument as in the proof of Part (1) shows that 
\begin{itemize}
\item both $V_{\prec \lambda}$ and $V_{\preceq \lambda} / V_{\prec \lambda}$ admit good filtrations; and that
\item the natural map $\Hom_G(\ttW_\nu, V_{\preceq \lambda}) \to \Hom_G(\ttW_\nu, V_{\preceq \lambda}/V_{\prec \lambda})$ is zero unless $\nu = \lambda$, in which case, the map is an isomorphism. 
\end{itemize}
Now the isomorphism $V_{\preceq \la}/V_{\prec \la}\cong\Hom_G(\ttW_{\la},V)\otimes \ttS_{\la}$ follows by combining the isomorphisms
\begin{itemize}
\item $V_{\preceq \lambda} / V_{\prec \lambda}\cong (V_{\preceq \lambda} / V_{\prec \lambda})(\la)\otimes \ttS_{\la}$;
\item $(V_{\preceq \lambda} / V_{\prec \lambda})(\la)\cong \Hom_G(\ttW_\la, V_{\preceq \lambda} / V_{\prec \lambda})\cong \Hom_G(\ttW_\la, V_{\preceq \lambda})\cong \Hom_G(\ttW_\la, V)$.\qedhere
\end{itemize}
\end{proof}

We end this subsection with the following results.
\begin{lem}
Let $V$ be a $G$-module that admits a good filtration, and let $\la\in\xch(T)^+_{\on{pos}}$.
Then the following sequence is  exact
\begin{equation}
\label{L:mugradeKoszul}
0\to V_{\preceq \la-\al_1-\cdots-\al_r} \to \cdots\to\bigoplus_{i<j}V_{\preceq \la-\al_i-\al_j}\to  \bigoplus_i V_{\preceq \la-\al_i}\to V_{\preceq \la} \to V_{\preceq \la}/V_{\prec \la}\to 0.
\end{equation}
\end{lem}
\begin{proof}
Note that $(e^{\al_1},\ldots,e^{\al_r})$ form a regular sequence in $R$ generating the ideal $I_0\subset R$, and since $R_{\xch(T)^+_{\on{pos}}}V$ is $R$-flat, they form a regular sequence in $R_{\xch(T)^+_{\on{pos}}}V$.
The lemma follows from taking the $\la$-graded piece of the corresponding Koszul complex.
\end{proof}

\begin{cor}
\label{C:good filtration via Koszul} In the exact sequence \eqref{L:mugradeKoszul}, both the kernel and the image of $\bigoplus_i V_{\preceq \la-\al_i}\to V_{\preceq \la}$ admit good filtrations.
\end{cor}
\begin{proof}
By Theorem~\ref{T:prop of good fil} (4) and Lemma~\ref{L:good sub}, each term in \eqref{L:mugradeKoszul} admits a good filtration. The corollary follows from the criteria for existence of good filtrations in Theorem~\ref{T:prop of good fil} (3).
\end{proof}

\subsection{Vinberg monoid via the canonical filtration}
\label{SS: Vin}
We apply the previous discussion to the $G\times G$-modules $k[G]$, with the $G\times G$-module structure given by left and right translation of $G$ on itself.
We regard a pair of weights $(\nu_1,\nu_2)$ as a weight of $T\times T$. Then for $\nu\in\xch(T)$, by
regarding $k[G]$ as a $G\times G$-representation, we can define
\begin{equation}
\label{E: fil on k[G]}
\fil_\nu k[G]:= k[G]_{\preceq (\nu^*,\nu)}.
\end{equation}
\begin{lem}
\label{C: fil on kG}
The above $\xch(T)^+_{\on{pos}}$ filtration $\{\on{fil}_\mu k[G]\}_{\mu\in \xch(T)^+_{\on{pos}}}$ of $k[G]$ satisfies the following properties.
\begin{enumerate}
\item $\on{fil}_{\nu} k[G]\subset \on{fil}_{\nu+\la} k[G]$ if $\la\in\xch(T_\ad)_{\on{pos}}$. 
\item $\on{fil}_{\nu}k[G]\cdot\on{fil}_{\nu'}k[G]\subset\on{fil}_{\nu+\nu'}k[G]$;
\item Each $\on{fil}_\nu k[G]$ is a sub-coalgebra of $k[G]$.
\item $\on{gr}k[G]=\oplus_{\nu\in\xch(T)^+}\ttS_{\nu^*}\otimes \ttS_{\nu}$;
\end{enumerate}
\end{lem}
\begin{proof}
Properties (1)--(3)  are clear. 
Property (4) follows from Theorem~\ref{T:prop of good fil} (4) and Proposition~\ref{P: good implies flat}. 
\end{proof}
Write $T^+_\ad=\Spec R$. This is a natural monoid (in fact, the affine space with coordinate function indexed by simple roots, and equipped with  coordinate multiplication), containing the adjoint torus $T_\ad$ of $G$ as the open subset of the group of invertible elements. In particular, $R$ has a coalgebra structure. To avoid possible confusion of notations in later discussions, we write $\bar{e}^{\al}$ (instead of $e^\al$) for the coordinate function on $T^+_\ad$ corresponding the simple root $\al$. The comultiplication sends $\bar{e}^{\al}$ to $\bar{e}^{\al}\otimes \bar{e}^{\al}$.

According to the discussion of Rees algebra in \S~\ref{SSS: rees alg}, $R_{\xch(T)^+_{\on{pos}}}k[G]$ is an $R$-algebra, and the map $R\to R_{\xch(T)^+_{\on{pos}}}k[G]$ is also a coalgebra homomorphism. Then 
$$V_G:=\Spec R_{\xch(T)^+_{\on{pos}}}k[G]$$ is a monoid, which is usually called the  \emph{Vinberg monoid} (at least when $G$ is semisimple and simply-connected, see Remark~\ref{R: comp def} below). In addition, it is 
equipped with a monoid homomorphism 
\begin{equation}
\label{E:abel map}
\frakd: V_G\to T^+_\ad,
\end{equation}
usually called the \emph{abelianization map}. We pick two distinguished representatives for the open and closed  $T_\mathrm{ad}$-orbit on $T_\mathrm{ad}^+$: ${\bf 1} = (1, \dots, 1), {\bf 0} = (0, \dots, 0) \in \AAA^r \cong T_\mathrm{ad}^+$. Then by \eqref{E: Rees fiber dim},
\[\frakd^{-1}({\bf 1})\cong \Spec (k[G])=G, \quad \frakd^{-1}({\bf 0})\cong \Spec (\on{gr} k[G])=:\on{As}_G.\]
The affine scheme $\on{As}_G$ is usually called the \emph{asymptotic cone} of $G$. We will make use of the following basic facts.
\begin{prop}
\label{T:Vinberg}
\begin{enumerate}
\item Let $G\times^{Z_G}T$ be the quotient of $G\times T$ by the action of $Z_G$ given by $z\cdot (g,t)=(zg,zt)$. 
 The affine monoid $V_G$ contains the open affine scheme $G\times^{Z_G}T$ as the group of invertible elements, such that the abelianization map \eqref{E:abel map} extends the natural group homomorphism $G\times^{Z_G}T\to T_\ad, \ (g,t)\mapsto (t\mod Z_G)$.
In particular, there is a natural $G\times G\times T$-action on $V_G$, where $G\times G$ acts on $V_G$ by left and right translation  and $T$ acts on $V_G$ by multiplication. 
\item The map $\frakd$ is faithfully flat.
\end{enumerate}
\end{prop}
\begin{proof}
Part (1) is clear. Part (2) follows from Proposition~\ref{P: good implies flat}.
\end{proof}
\begin{rmk}
\label{R: comp def}
The original construction of the Vinberg monoid as  in \cite{Vi} (when $\cha k=0$) and in \cite{Ri3} (in general) is different. However, it is easy to see that $V_G$ is uniquely characterized by the properties in Proposition~\ref{T:Vinberg} and the fact that $\frakd^{-1}({\bf 0})=\on{As}_G$. Therefore $V_G$ is indeed what people usually call the Vinberg monoid.
\end{rmk}

Let $V_T$ be the closure of $T\times^{Z_G}T\subset G\times^{Z_G}T$  in $V_G$. The image of the $\xch(T)^+_{\on{pos}}$-filtration on $k[G]$ under the map $k[G]\to k[T]$ defines an $\xch(T)^+_{\on{pos}}$-filtration on $k[T]$ given by
\begin{equation}
\label{E:fil on k[T]}
\fil_\nu k[T]=\bigoplus_{\la_{\on{dom}}\preceq \nu}k\cdot e^{\la},
\end{equation}
where for a weight $\la$, $\la_{\on{dom}}$ denotes the unique dominant element in the Weyl group orbit $W\la$ of $\la$. 
Then the embedding $T\times^{Z_G}T\to V_T$ is given by
\begin{equation*}
\label{E:k[VT]}k[V_T]=\bigoplus_{(\la,\nu),\ \nu-\la_{\on{dom}}\in \xch(T_\ad)_{\on{pos}}} k(e_1^\la\otimes e_2^\nu)\subset\bigoplus_{(\la,\nu),\ \nu-\la\in \xch(T_\ad)} k(e_1^\la\otimes e_2^\nu)=k[T\times^{Z_G}T],
\end{equation*}
where $e_i^\mu$ are the corresponding character functions on the $i$th factor of $T\times^{Z_G}T$,
and the map $V_T\to T_\ad^+$ is given by $\bar e^\la\mapsto 1\otimes e_2^\la$ for $\la\in\xch(T_\ad)_{\on{pos}}$.

\subsection{The filtration on the weight spaces}
\label{S:the filtration}
To prepare our study of vector-valued conjugation invariant functions $\bfJ_G(V)$, we need to introduce a different filtration on each weight space of a representation $V$ of $G$. This is not directly related to the filtration we discussed in the previous subsections. 
Fix a simple root $\alpha$ of $G$.
Let $V$ be a representation of $G$ and let $\nu$ be a weight of $T$. Define a filtration on $V(\nu)$ as follows
\begin{equation}
\label{E: equiv def 1}
\on{fil}^\al_iV(\nu):=V(\nu)\cap (\Res_{G_\al}^GV)_{\preceq_\al\nu+i\al}.
\end{equation}
There are two equivalent descriptions of the filtration.
First, we claim that
\begin{equation}
\label{E:filtration on V(nu) 1}
\on{fil}^{\al}_{i}V(\nu)=\ker\Big(\bigoplus_{j\geq 1} E_{\al}^{(i+j)}: V(\nu)\to \bigoplus_{j\geq 1}V(\nu+(i+j)\al)\Big),
\end{equation}
Indeed, let $v$ be a vector in the right hand side of \eqref{E:filtration on V(nu) 1}. Then $\on{Dist}(G_\al)v$ is a $G_\al$-module whose weights $\preceq_\al \nu+i\al$. Conversely, if $v\in V(\nu)\cap V_{\preceq_\al\nu+i\al}$, then clearly $E_\al^{(i+j)}v=0$ for every $j\geq 1$. The claim follows.

In addition, recall that if $V$ is a finite dimensional representation of $G_\al$ whose weights $\preceq_{\al}\la$, then all of its weights $\succeq_\al s_\al(\la)$. It follows that we can also define the filtration as
\begin{equation}
\label{E: equiv def 2}
\on{fil}^\al_iV(\nu)=\ker\Big(\bigoplus_{j\geq 1}F^{(\langle\nu,\al^\vee\rangle+i+j)}_\al: V(\nu)\to \bigoplus_{j\geq 1} V\big(\nu-(\langle\nu,\al^\vee\rangle+i+j)\al\big)\Big).
\end{equation}

Next we define the multi-filtration we need. Let $T_\s$ denote the preimage of $T$ in the simply-connected cover $G_\s$ of $G$, and $\Gamma=S=\xch(T_\s)^+$ be the monoid of dominant weights, which acts on itself by translations. We identify 
\begin{equation}
\label{E: semigroup of dom weight}
\xch(T_\s)^+\cong \bN^\Delta, \ \mu\mapsto (\langle\mu,\al^\vee\rangle)_{\al\in\Delta}.
\end{equation}  Under this identification, the partial order on $\xch(T_\s)^+$ induced by the translation action (as in \S~\ref{SS: monoid and partial order}) is just the standard partial order on $\bN^\Delta$, as in Example~\ref{E: classical and mult fil} (2). (Note that this is different from the restriction to $\xch(T_\s)^+$ of partial order $\preceq$ on $\xch(T_\s)$  induced by the action of $\xch(T_\ad)_{\on{pos}}$.)
We define an $\xch(T_\s)^+$-filtration on $V(\nu)$ as in Example~\ref{E: classical and mult fil} (2), i.e.
\begin{equation*}
\label{E:filtration on V(nu) 2}
\on{fil}_\la V(\nu)=\bigcap_{\al\in\Delta} \on{fil}^\al_{\langle\la,\al^\vee\rangle}V(\nu).
\end{equation*}
We obtain a quadruple $(\Ga, S, M, \{\on{fil}_sM\}_{s\in S})=\big(\xch(T_\s)^+,\xch(T_\s)^+, V(\nu), \{\on{fil}_\la V(\nu)\}_{\la\in\xch(T_\s)^+}\big)$. 
The main result of this section is
\begin{theorem}
\label{P: equi dim}
Assume that $V$ admits a good filtration, then $\dim \gr V(\nu)=\dim V(\nu)$. 
\end{theorem}
We prove this theorem here, assuming two ingredients that will be established in \S~\ref{SS:fil via BW}--\S~\ref{SS: fil via geomSat}.
\begin{proof}
We first reduce to the case when $V$ is a Schur module. Indeed, suppose we have a short exact sequence of $G$-modules
\[0\to V'\to V\to V''\to 0.\]
We deduce easily from \eqref{E:filtration on V(nu) 1} or \eqref{E: equiv def 2} an exact sequence $0\to \fil_\la V'(\nu)\to \fil_\la V(\nu)\to \fil_{\la} V''(\nu)$. In addition, if $V'$ admits a good filtration, by Theorem~\ref{T:prop of good fil} and Proposition~\ref{L: fil via inv} below, this sequence is also exact on the right.

It follows that it is enough to prove the theorem for a Schur module $V$. In this case, by Lemma~\ref{L: bases criterion}, it is a consequence of the following proposition.
\end{proof}
\begin{prop}
\label{P: canonical basis}
Let $V$ be a Schur module.
Then there exists a basis $\{v_j\}$ of $V(\nu)$ such that for every simple root $\al$, the corresponding symbols $\{\overline{v_j}\}$ for the filtration $\on{fil}^\al$ form a basis of the associated graded $\gr^\al V(\nu)=\oplus_i \on{fil}^\al_iV(\nu)/\on{fil}^\al_{i-1}V(\nu)$.
\end{prop}
There are several ways to construct such a basis. For example, the canonical basis constructed by Lusztig and Kashiwara, or the semi-canonical basis constructed by Lusztig satisfies the required properties in the proposition (see \cite[Theorem 3.1 and Corollary 3.9]{Lu2}).
At the end of \S~\ref{SS: fil via geomSat}, we will give an alternative construction of a basis with the needed properties using the MV basis from the geometric Satake correspondence.\footnote{This is the set of basis we use in \cite{XZ}. On the other hand, it is expected that the MV basis coincide with the semi-canonical basis. We thank Lusztig for pointing this out.}

\begin{rmk}
\label{R:mult weight}
For $\mu,\nu\in\xch(T)$ with $\mu$ dominant, we define
\[P_{\mu,\nu}(q):=\sum_{\la\in\xch(T_\s)^+} \dim \gr_{\la}\mathtt S_\mu(\nu)q^\la\in \bZ[\xch(T_\s)^+]\]
as an element in the monoid algebra for $\xch(T_\s)^+$. (Here we use $q^\la$ instead of $e^\la$ to denote the element in $\bZ[\xch(T_\s)^+]$ corresponding to $\la$.)
If we identify $\xch(T_\s)^+$ with $\bN^l$ as above, it becomes to a polynomial of $l$-variables
\[P_{\mu,\nu}(q_1,\ldots,q_l)=\sum \dim \gr_{(s_1,\ldots,s_l)}\mathtt S_\mu(\nu)q_1^{s_1}\cdots q_l^{s_l}.\]
Then Theorem~\ref{P: equi dim} implies that 
$P_{\mu,\nu}(1)=\dim \mathtt S_\mu(\nu)$. This may be viewed as a multivariable analogue of the Lusztig--Kato polynomials (which give the $q$-analogue of the weight multiplicities \cite{Lu}).
\end{rmk}

\subsection{The filtration via Borel--Weil}
\label{SS:fil via BW}
We will explain how to obtain the filtration defined in \S~\ref{S:the filtration} using the geometry of flag varieties.

Recall that by Frobenius reciprocity, there is a canonical morphism of $B^-$-modules $\mathtt S_\mu\to k_\mu$, and a canonical morphism of $B$-modules $\mathtt S_\mu\to k_{w_0(\mu)}$.

Let $\mu,\nu$ be two dominant weights. Let $V$ be a representation of $G$.
The map $\mathtt S_{\mu^*}\otimes \mathtt S_\nu\otimes V\to k_{-\mu}\otimes k_\nu\otimes V$ induces a natural map 
\begin{equation*}
\label{E: leading term}
\ell_{\mu,\nu}: (\mathtt S_{\mu^*}\otimes \mathtt S_\nu\otimes V)^G\to (k_{-\mu}\otimes k_{\nu}\otimes V)^T\cong V(\mu-\nu).
\end{equation*}
Note that the collection of the maps $\{\ell_{\mu,\nu}\}$ has the following property. 
For a dominant weight $\eta$, there is a canonical $G$-equivariant map in $\Hom_G(\mathtt W_\eta, \mathtt S_\eta)$, which gives a $G$-invariant element $\sigma_\eta\in \mathtt S_{\eta^*}\otimes \mathtt S_\eta$.
Multiplying by $\sigma_\eta$ induces
\[ (\mathtt S_{\mu^*}\otimes \mathtt S_\nu\otimes V)^G\to (\mathtt S_{\mu^*+\eta^*}\otimes \mathtt S_{\nu+\eta}\otimes V)^G.\]
We denote this map still by $\sigma_\eta$. Then it is clear that
\begin{equation}
\label{E: comp mult}
\ell_{\mu+\eta, \nu+\eta}\circ \sigma_\eta=\ell_{\mu,\nu}.
\end{equation}

This subsection is devoted to proving the following.
\begin{prop}
\label{L: fil via inv}
The map $\ell_{\mu, \nu}$ above induces an isomorphism $$(\mathtt S_{\mu^*}\otimes \mathtt S_\nu\otimes V)^G\cong \on{fil}_{\nu}V(\mu-\nu)\subset V(\mu-\nu).$$
Here by abuse of notations, the image of $\nu$ under the map $\xch(T)\to\xch(T_\s)$ is still denoted by $\nu$.
\end{prop}

We will later in \S~\ref{SS:proof of Borel Weil filtration} deduce this proposition from the following natural exact sequence of $B$-modules
\begin{equation}
\label{E: DBGG}
0\to \mathtt S_{\nu}\to \on{ind}_{T}^{B} k_{\nu}\to \bigoplus_{\al} \on{ind}_{B_\al}^B \ttM^\al_{s_\al(\nu)-\al}.
\end{equation}
Here, for a weight $\la\in \xch(T)$, 
$$\ttM^\al_\la:=\dist(G_\al)\otimes_{\dist(B_\al)}k_\la$$ is the \emph{restricted Verma module} of $G_\al$ of highest weight $\la$. 
The sequence \eqref{E: DBGG} is in fact the first two terms of the (restricted) dual BGG complex. Since this sequence in the above form (and in characteristic $p>0$) might be not familiar to some readers, we give a self-contained construction.

\subsubsection{Case of $\SL_2$}
\label{SS: SL2}
First, we review some facts about representations of $G=\SL_2$. Let $B$ (resp. $B^-$) be the subgroup of upper (resp. lower) triangular matrices in $\SL_2$, and let $T=B\cap B^-\cong \bG_m$ be the group of diagonal matrices. We identify $G/B^-=\bP^1$ in the way such that 
\begin{itemize}
\item $j:\bA^1\to \bP^1$ corresponds to the open $B$-orbit and $i:\{\infty\}\to\bP^1$ corresponds to the closed $B$-orbit, 
\item $0\in \bA^1$ is fixed by $B^-$. 
\end{itemize}
For $n \in \NN$, consider the exact sequence of sheaves
\begin{equation}
\label{E: exact on P1}
0\to \mO(n)\to j_*j^*\mO(n)\to j_*j^*\mO(n)/\mO(n)\to 0 
\end{equation}
on $\bP^1$. Since the sequence is $B$-equivariant, we may also regard it as an exact sequence on $[B\backslash G/B^-]$ via descent.

We fix two nonzero sections $t_0$ and $t_\infty$ of $\mO(1)$ that vanish at $0$ and $\infty$ respectively and view $x=t_0/t_\infty$ as a coordinate function on $\bA^1$. Then
\[\Gamma(\bA^1, \mO(n))= k[x]t_\infty^n\cong k[x]=:\ttM^{\vee}_n\]
is a $(\on{Dist}(G),B)$-module with highest weight $n$, on which $e$ acts as $\frac{d}{dx}$, $h$ acts as $n-2x\frac{d}{dx}$ and $f$ acts as $x(n-x\frac{d}{dx})$. 
Note that as a $\on{Dist}(G)$-module, it is isomorphic to the restricted dual Verma module, and as a  $B$-module, it is isomorphic to the induced representation $\on{ind}_T^B k_n$. The subspace 
$$\Gamma(\bP^1,\mO(n))\cong \{f(x) \in k[x]\mid \deg f\leq n\}:=\mathtt S_n$$ is the Schur module for $\SL_2$ of highest weight $n$. The section of the quotient sheaf is 
$$\Gamma(\bP^1,j_*j^*\mO(n)/\mO(n))\cong x^{n+1}k[x]=:\ttM_{-n-2}, $$ 
which as a $\on{Dist}(G)$-module is isomorphic to the restricted Verma module for $\SL_2$ of highest weight $-n-2$. 

For a $B$-module $V$, write $V=\oplus V(j)$ for the weight decomposition with respect to $T$, then
\[(V\otimes \ttM^\vee_n)^B=\Big\{ \sum_i (-1)^ie^{(i)}v\otimes x^i\mid v\in V(-n)\Big\}\cong V(-n). \]
It follows that the following diagram is commutative with horizontal sequences exact 
\begin{equation}
\label{E: express kernel}
\xymatrix@C=30pt{
0 \ar[r] & (V\otimes \mathtt S_n)^B \ar@{=}[d] \ar[r] & (V\otimes \ttM^\vee_n)^B \ar[rr] \ar[d]^\cong && (V\otimes \ttM_{-n-2})^B \ar[d] \\
0 \ar[r] &  (V\otimes \mathtt S_n)^B \ar[r]& V(-n)  \ar[rr]^-{\oplus_{i\geq 1}(-1)^{n+i} e^{(n+i)}} && \bigoplus_{i\geq 1}V(n+2i),
}
\end{equation}
where the right vertical map is the inclusion 
$(V\otimes \ttM_{-n-2})^B \hookrightarrow (V\otimes \ttM_{-n-2})^T \cong \bigoplus_{i\geq 1}V(n+2i)$.

Note that the above discussions may equally apply to $G_\al$, with $n$ replaced by a weight $\nu$ of $T$ such that $\langle\nu,\al^\vee\rangle\geq 0$. The corresponding $(\on{Dist}(G_\al), B_\al)$-modules $\ttM_n^\vee$, $\ttS_n$, and $\ttM_{-n-2}$ will be denoted by $\ttM^{\al,\vee}_\nu$, $\ttS^\al_\nu$, and $ \ttM^\al_{s_\al(\nu)-\al}$, respectively.

\subsubsection{General case}Recall that the $B$-orbits on $G/B^-$ are parameterized by the Weyl group $W$. For $w\in W$, let $C_w$ denote the corresponding $B$-orbit through $\dot{w}\in G/B^-$, where $\dot{w}$ is any lifting of $w$ to $N_G(T)$. In particular, $C_e$ is open and isomorphic to $B/T$, and $C_{s_\al}$'s are of codimension one, where $s_{\al}$ is the simple reflection  corresponding to the simple root $\al$. In addition,
the natural map $B\times^{B_\al} (G_{\al}/B_{\al}^-)\to G/B^-$ is an open embedding, with the image $C_{\leq s_\al}=C_e\sqcup C_{s_\al}$.
Let $C_\Delta=\cup_{\al\in\Delta} C_{\leq s_\al}$ be the open subset of $G/B^-$ complement to the union of $B$-orbits of codimension at least two. The inclusion $j: C_e\to C_\Delta$ is open and the inclusion $i: \sqcup C_{\al}\to C_\Delta$ is closed. For simplicity, the restriction of $\mO_{G/B^-}(\nu)$ to $C_\Delta$ is denoted by $\mO(\nu)$. Consider the following exact sequence of $B$-equivariant quasi-coherent sheaves on $C_\Delta$
\begin{equation}
\label{E: exact on Cdelta}
0\to \mO(\nu)\to j_*j^*\mO(\nu)\to  j_*j^*\mO(\nu)/\mO(\nu)\to 0.
\end{equation} Note that $C_e\cong B/T$, and $\mO(\nu)|_{C_e}\cong B\times^Tk_\nu$. In addition, the restriction of the map $j_*j^*\mO(\nu)\to  j_*j^*\mO(\nu)/\mO(\nu)$ to $C_{\leq s_\al}$ is the pullback of \eqref{E: exact on P1} (with $n$ replaced by $\nu$) along the natural projection $C_{\leq s_\al}\cong B\times^{B_\al}(G_\al/B_{\al}^-)\to[B_\al\backslash G_\al/B_\al^-]$. It follows from these observations and the previous discussions about $\SL_2$ that \eqref{E: DBGG} is obtained by taking the global sections of \eqref{E: exact on Cdelta}.

\subsubsection{Proof of Proposition~\ref{L: fil via inv}}
\label{SS:proof of Borel Weil filtration}
We may assume that $V$ is finite dimensional. 
Using Frobenius reciprocity, we have
\[
(\mathtt S_{\mu^*}\otimes \mathtt S_\nu\otimes V)^G = \Hom_G(\mathtt W_{\nu^*},\mathtt S_{\mu^*}\otimes V) =  \Hom_B(\mathtt W_{\nu^*}, k_{-\mu}\otimes V) = \Hom_B(k, \mathtt S_{\nu}\otimes k_{-\mu}\otimes V).
\]                                                      
Then by \eqref{E: DBGG}, there is an exact sequence
\[0\to \Hom_B(k, \mathtt S_{\nu}\otimes k_{-\mu}\otimes V)\to (k_\nu\otimes k_{-\mu}\otimes V)^T\to \bigoplus_\al(\ttM_{s_\al(\nu)-\al}^\al\otimes k_{-\mu}\otimes V)^{B_\al}.\]
It is easy to check that the resulting map $(\mathtt S_{\mu^*}\otimes \mathtt S_\nu\otimes V)^G= \Hom_B(k, \mathtt S_{\nu}\otimes k_{-\mu}\otimes V)\to (k_\nu\otimes k_{-\mu}\otimes V)^T$ is $\ell_{\mu,\nu}$. In addition, using \eqref{E: express kernel} and \eqref{E:filtration on V(nu) 1}, we see that the kernel of the map
$$V(\mu-\nu)=(k_\nu\otimes k_{-\mu}\otimes V)^T\to (\mathtt M_{s_\al(\nu)-\al}^\al\otimes k_{-\mu}\otimes V)^{B_\al}$$ is exactly $\on{fil}^\al_{\langle\nu,\al^\vee\rangle}V(\mu-\nu)$.
Proposition~\ref{L: fil via inv} now follows from the definition of the $\xch(T)^+$-filtration on $V(\mu-\nu)$.

\subsection{The filtration via geometric Satake correspondence}
\label{SS: fil via geomSat} 
We give a geometric construction of the above filtrations via the geometric Satake correspondence and in particular give a proof of Proposition~\ref{P: canonical basis}.

We refer to \cite{Gi, MV} for the geometric Satake correspondence (see also \cite{Z16, BR17} for an exposition).
Let $\hat G$ be the Langlands dual group of $G$ over $\bC$. Let $\Gr_{\hat G}=L\hat G/L^+\hat G$ denote its affine Grassmannian over $\bC$, equipped with the analytic topology. Here for an affine variety $Z$ over $\bC$, let $LZ$ (resp. $L^+Z$) denote its loop (resp. jet) space as usual, so $LZ(\bC) = Z(\bC((t)))$ and $L^+Z(\bC) = Z(\bC[[t]])$. Let $\on{P}_{L^+\hat G}(\Gr_{\hat G})$ be the category of $L^+\hat G$-equivariant perverse sheaves on $\Gr_{\hat G}$, with $k$-coefficients as in \cite{MV}. It is known from \emph{loc. cit.} that this is an abelian tensor category.
Recall that the geometric Satake correspondence is a natural equivalence of abelian tensor categories
\[\Sat: \Rep^f(G)\to \on{P}_{L^+\hat G}(\Gr_{\hat G}),\]
such that its composition with hypercohomology functor $\on{H}^*(\Gr_{\hat G},-)$ is isomorphic to the forgetful functor from $\Rep^f(G)$ to the category of finite dimensional $k$-vector spaces.
Let us also recall from {\it loc. cit.} the dictionary between some geometry and representation theory under this equivalence. For a weight $\nu\in\xch(T)=\xcoch(\hat T)$, let $t^\nu$ denote the $\hat T$-fixed point of $\Gr_{\hat G}$ corresponding to $\nu$ as usual, i.e. it is the image of $t\in \bC((t))^\times=L\bG_m(\bC)$ under the map $L\bG_m\xrightarrow{\nu} L\hat{T}\to L\hat{G}\to \Gr_{\hat{G}}$.

\begin{enumerate}
\item For each $\la\in\xch(T)^+$, let $\mathring{\Gr}_{\hat G,\la}$ denote the $L^+\hat G$-orbit through $t^\la$, and $\Gr_{\hat G,\la}$ its closure. They are $\langle2\check\rho,\la\rangle$-dimensional, where $2\check\rho$ is the sum of positive coroots of $G$.
Let $i_\la: \Gr_{\hat G,\la}\to \Gr_{\hat G}$ denote the closed embedding and $\mathring{i}_\la:\mathring{\Gr}_{\hat G,\la}\to \Gr_{\hat G}$ the locally closed embedding. Let $k[\langle2\check\rho,\la\rangle]$ denote the constant sheaf on $\mathring{\Gr}_{\hat G,\la}$, shifted to degree $-\langle2\check\rho,\la\rangle$. Then 
$$\Sat(\mathtt W_\la) \simeq {^p}\on{H}^0(\mathring{i}_{\la})_!k[\langle2\check\rho,\la\rangle],\quad \Sat(\mathtt S_\la)\simeq {^p}\on{H}^0(\mathring{i}_{\la})_*k[\langle2\check\rho,\la\rangle].$$ 
\item We fix a Borel subgroup $\hat B$ of $\hat G$, and a maximal torus $\hat T \subset \hat B$. Let $\hat U\subset \hat B$ denote the unipotent radical.  
Let $S_\nu$ be the $L\hat U$-orbit through $t^\nu$. Then the functor
\begin{equation}
\label{E:Fnu}F_\nu:= \on{H}^{\langle2\check\rho,\nu\rangle}_{S_\nu}(\Gr_{\hat G}, -): D_c^b(\Gr_{\hat{G}},k)\to \on{Vect}_k
\end{equation}
is exact when restricted to $\on{P}_{L^+\hat{G}}(\Gr_{\hat{G}})$, which corresponds to the weight functor $V\mapsto V(\nu)$ under the geometric Satake correspondence. 
\item For a simple root $\al$ of $G$, let $\hat P_\al$ be the standard parabolic subgroup whose Levi quotient $\hat G_\al$ is the Langlands dual group of $G_\al$.
There is the following diagram
\[\xymatrix{
&\Gr_{\hat P_\al}\ar_{q_\al}[dl]\ar^{i_\al}[dr]&\\
\Gr_{\hat G_\al}&&\Gr_{\hat G}.
}\]
The morphism $i_\alpha$ is a locally closed embedding and hence we also regard $S_\nu$ as a subscheme of $\Gr_{\hat P_\alpha}$.
The connected components of the affine Grassmannian $\Gr_{\hat G_\al}$ are parameterized by $\xch(T)/\bZ \al$. For $\theta\in \xch(T)/\bZ \al$, let $\Gr^\theta_{\hat G_\al}$ be the corresponding component, and $\Gr_{\hat P_\al}^\theta= q_\al^{-1}(\Gr_{\hat G_\al}^\theta)$.  The restrictions of $i_\al$ and $q_\al$ to $\Gr_{\hat P_\al}^\theta$ are denoted by $i^\theta_\al$ and $q_\al^\theta$ respectively. Recall that $\langle 2\check\rho,\al\rangle=\langle\check\al, \al\rangle$ and therefore $\langle 2\check\rho-\check\al, \theta\rangle$ makes sense, which we denote by $l_\theta$. Then there is a perverse exact functor 
$$\on{CT}_\al:=\bigoplus_\theta (q^\theta_{\al})_*(i_\al^{\theta})^![l_\theta]: \on{P}_{L^+\hat G}(\Gr_{\hat G})\to \on{P}_{L^+\hat G_\al}(\Gr_{\hat G_\al})$$ 
which corresponds under the geometric Satake correspondence to the restriction functor from $G$-representations to $G_\al$-representations. 
For a weight $\nu$, let $S_\nu^\al\subset\Gr_{\hat G_\al}$ be the $L\hat U_\al$-orbit through $t^\nu$, and $F^\al_\nu=\on{H}^{\langle\check\al,\nu\rangle}_{S^\alpha_\nu}(\Gr_{\hat G_\al}, -)$ the corresponding weight functor on $\on{P}_{L^+\hat G_\al}(\Gr_{\hat G_\al})$. Since $q_{\al}^{-1}(S_\nu^\al)=S_\nu$, the proper base change theorem implies that there is a canonical isomorphism $F_\nu\cong F_\nu^{\al}\circ\on{CT}_\al$, which corresponds to the natural identification $V(\nu)=(\Res^G_{G_\al}V)(\nu)$ under the geometric Satake correspondence.
\end{enumerate}

Now we give a geometric construction of the filtration defined in \S~\ref{S:the filtration}. First, we define a filtration on $S_\nu$ by open subsets 
\[S_\nu^{>\nu+i\al}:= q_\al^{-1}(S_\nu^\al\cap (\Gr_{\hat G_\al}-\Gr_{\hat G_\al, \nu+i\al})).\]

\begin{prop}
\label{P:fil alpha in terms of geometric Satake}
Let $V$ be a finite dimensional representation of $G$. Then we have an exact sequence
$$0\to \on{fil}^\al_iV(\nu)\to \on{H}^{\langle2\check\rho,\nu\rangle}_{S_\nu}(\Gr_{\hat G}, \Sat(V))\to \on{H}^{\langle2\check\rho,\nu\rangle}_{S_\nu^{>\nu+i\al}}(\Gr_{\hat G}, \Sat(V)).$$
\end{prop}
\begin{rmk}
It will follow from the arguments in \S~\ref{SSS:constr of basis} that if $V$ is a Schur module (or more generally $V$ admits a good filtration), the above sequence is also surjective at the right. 
\end{rmk}
\begin{proof}
For an $\alpha$-dominant weight $\nu$ of $G$,
we denote the closed embedding $\Gr_{\hat G_\al,\nu}\to \Gr_{\hat G_\al}$ by $i^\al_\nu$, and the complementary open embedding by $j^\al_\nu$. Then for $\mA \in \on{P}_{L^+\hat G}(\Gr_{\hat G})$, there is the following distinguished triangle of sheaves on $\Gr_{\hat G_\al}$
\begin{equation*}
\label{E:dis tri}
 (i^\al_{\nu+i\al})_*(i^\al_{\nu+i\al})^!\on{CT}_\al(\mA)\to \on{CT}_\al(\mA)\to (j^\al_{\nu+i\al})_*(j^\al_{\nu+i\al})^*\on{CT}_\al(\mA)\to. 
 \end{equation*}
Applying the functor $F^\al_\nu$ as in \eqref{E:Fnu}, and noticing by the proper base change,
$$F^\al_\nu((j^\al_{\nu+i\al})_*(j^\al_{\nu+i\al})^*\on{CT}_\al(\mA))=\on{H}^{\langle2\check\rho,\nu\rangle}_{S_\nu^{>\nu+i\al}}(\Gr_{\hat G}, \mA)),$$ 
we obtain
\begin{equation*}
\label{E:exact right}
0\to F^\al_\nu(i^\al_{\nu+i\al})_*{^p}\on{H}^0(i^\al_{\nu+i\al})^!\on{CT}_\al(\mA)\to \on{H}^{\langle2\check\rho,\nu\rangle}_{S_\nu}(\Gr_{\hat G}, \mA)\to \on{H}^{\langle2\check\rho,\nu\rangle}_{S_\nu^{>\nu+i\al}}(\Gr_{\hat G}, \mA).
\end{equation*}
The injectivity at the left follows from the fact that $(j^\al_{\nu+i\al})_*(j^\al_{\nu+i\al})^*\on{CT}_\al(\mA)$ lives in perverse cohomological degree $\geq 0$ and the exactness of $F_\nu^\al$. Now let $\mA=\Sat(V)$.
To conclude the proof, we apply the following lemma (which follows from the description of the Weyl modules under the geometric Satake correspondence as in (1) above) and the definition of the filtration \eqref{E: equiv def 1} to replace $F^\al_\nu(i^\al_{\nu+i\al})_*{^p}\on{H}^0(i^\al_{\nu+i\al})^!\on{CT}_\al(\Sat(V))$ by $\fil^\al_i V(\nu)$.
\end{proof}
\begin{lem}
Let $\la$ be a dominant weight of $G$.
Under the geometric Satake correspondence, the left exact functor $V\mapsto V_{\preceq \la}$ corresponds to $(i_\la)_*{^p}\on{H}^0(i_\la)^!: \on{P}_{L^+\hat G}(\Gr_{\hat G})\to \on{P}_{L^+\hat G}(\Gr_{\hat G})$.
\end{lem}

\subsubsection{Proof of Proposition~\ref{P: canonical basis}}
\label{SSS:constr of basis}
We construct the needed basis explicitly via MV cycles.
Let $V$ be a Schur module so $\Sat(V)={^p}\on{H}^0(\mathring{i}_\mu)_*k[\langle2\check\rho,\mu\rangle]$ for some dominant weight $\mu$ of $G$. In this case, the natural map ${^p}\on{H}^0(\mathring{i}_\mu)_*k[\langle2\check\rho,\mu\rangle]\to(\mathring{i}_\mu)_*k[\langle2\check\rho,\mu\rangle]$ induces the following commutative diagram
\begin{equation}
\label{E:diagram of restricting to Gr Galpha}
\begin{CD}
\on{H}^{\langle2\check\rho,\nu\rangle}_{S_\nu}(\Gr_{\hat G},{^p}\on{H}^0(\mathring{i}_\mu)_*k[\langle2\check\rho,\mu\rangle])@>>>\on{H}^{\langle2\check\rho,\nu\rangle}_{S_\nu^{>\nu+i\al}}(\Gr_{\hat G},{^p}\on{H}^0(\mathring{i}_\mu)_*k[\langle2\check\rho,\mu\rangle])\\
@VVV@VVV\\
\on{H}^{\langle2\check\rho,\nu+\mu\rangle}_{S_\nu\cap \mathring{\Gr}_{\hat G,\mu}}(\mathring{\Gr}_{\hat G,\mu},k)@>>>\on{H}^{\langle2\check\rho,\nu+\mu\rangle}_{S^{>\nu+i\al}_\nu\cap \mathring{\Gr}_{\hat G,\mu}}(\mathring{\Gr}_{\hat G,\mu},k).
\end{CD}
\end{equation}
The two vertical arrows are isomorphisms by (the proof of) \cite[Proposition 3.10]{MV}, and therefore we identify their sources and targets.
The groups in the bottom row are canonically isomorphic to the top Borel--Moore homology of $S_\nu\cap \mathring{\Gr}_{\hat G,\mu}$ and of $S^{>\nu+i\al}_\nu\cap \mathring{\Gr}_{\hat G,\mu}$ respectively, and therefore have a basis given by the fundamental classes of their irreducible components.
Such a fundamental class in the lower left corner maps to zero under the horizontal arrow if and only if the corresponding irreducible component does not intersect with $S^{>\nu+i\al}_\nu$. 

Now we take the basis of $V(\nu)=F_\nu(\Sat(V))$ given by the aforementioned fundamental classes $\{v_j\}$ in $\on{H}^{\langle2\check\rho,\nu+\mu\rangle}_{S_\nu\cap \mathring{\Gr}_{\hat G,\mu}}(\mathring{\Gr}_{\hat G,\mu},k)$, i.e. the MV basis. The discussion above says that, for every $\alpha \in \Delta$ and every $i\in \ZZ_{\geq 0}$,  those $v_j$'s which lie in the kernel of the top horizontal arrow of \eqref{E:diagram of restricting to Gr Galpha}, or equivalently in  $\fil_i^\alpha V(\nu)$ by Proposition~\ref{P:fil alpha in terms of geometric Satake}, in fact span $\fil_i^\alpha V(\nu)$. Therefore, their symbols $\overline{v_j}$ for the filtration $\fil^\alpha$ form a basis of the associated graded $\gr^\alpha V(\nu)$. This concludes the proof of Proposition~\ref{P: canonical basis}.

\section{Vector-valued twisted conjugation invariant functions}
\label{S:vector-valued twisted invariant functions}

In this section, we start to study the space of vector-valued (twisted) conjugation invariant functions on a group.

We will continue making use of the conventions and notations in \S~\ref{S: filtration on weight space}. In particular, we have a pinned reductive group $(G,B,T,\{x_\al\}_{\al\in\Delta})$ over $k$. Let $\sigma$ be an automorphism of $G$ preserving $(B,T)$, and hence the root system $\Phi(G, T)$. We define the $\sigma$-action on the character group $\XX^\bullet(T)$ by
\begin{equation*}
\label{E:sigma acts on X(T)}
\sigma(\alpha) (t) = \alpha(\sigma^{-1}(t)), \quad \textrm{for } \alpha \in \XX^\bullet(T),\ t \in T.
\end{equation*}
so that $\sigma(U_\al) = U_{\sigma(\al)}$.
For a $G$-representation $V$, let $\sigma V$ denote the representation 
$$G \xrightarrow{\sigma^{-1}}G \to \GL(V).$$ We identify elements in $\sigma V$ with $\{ \sigma v\,|\, v \in V\}$ and thus $g(\sigma v) = \sigma(\sigma^{-1}(g)v)$ for $g \in G$.  This way, we have $\sigma \ttS_\mu = \ttS_{\sigma(\mu)}$ and $\sigma \ttW_\mu = \ttW_{\sigma(\mu)}$ for $\mu \in \XX^\bullet(T)^+$.

\subsection{Vector-valued twisted conjugation invariant functions on a group}
\label{SS: vvif on cone}

Let $H$ be a linear algebraic group over $k$ equipped with an automorphism $\tau$. We still denote the $\tau$-twisted conjugation action of $H$ on itself (as defined in \eqref{E:Ginvfun}) by $c_\tau$, i.e 
$$c_\tau(h)(g)=hg\tau(h)^{-1}, \quad h,g\in H.$$ 
Let  $\bfJ_H=k[H]^{c_\tau(H)}$ denote the space of $\tau$-twisted conjugation invariant functions on $H$, i.e. those $f\in k[H]$ satisfying $f(hg\tau(h)^{-1})=f(g)$ for $h,g\in H$.
More generally, for a representation $V$ of $H$, we will denote by
\[\bfJ_H(V)=(k[H]\otimes V)^{c_\tau(H)}\] 
the space of vector-valued $\tau$-twisted conjugation invariant functions on $H$. Equivalently, let $[H/c_\tau(H)]$ (or sometimes $[H\tau/H]$) denote the quotient stack of $H$ by the $\tau$-twisted conjugation action of $H$ on itself, and let $\widetilde V:=V_{[H/c_\tau(H)]}$ denote the corresponding vector bundle on $[H/c_\tau(H)]$ (see \S~\ref{SS: NC} for the notations and conventions). 
Then
$$\bfJ_H=\Gamma([H\tau/H], \mO),\quad \bfJ_H(V)=\Gamma([H\tau/H], \widetilde V).$$
Moreover, $\bfJ_H(V)$ is naturally a $\bfJ_H$-module.
\begin{rmk}
\label{R: special tau}
Note that the spaces $\bfJ_H(V)$ depend only on the image of $\tau$ in the outer automorphism group $\on{Out}(H)$ of $H$. Indeed, if $\tau_2=c(h)\circ \tau_1$, where $c(h):H\to H, \ \ g\mapsto hgh^{-1}$ is the inner automorphism of $H$ induced by some $h\in H$, then the map $H\to H,\ x\mapsto xh$ is an isomorphism intertwining the action $c_{\tau_1}$ and $c_{\tau_2}$, and therefore induces isomorphisms between the spaces of vector-valued twisted conjugation invariant functions.
\end{rmk}

We discuss a few basic properties of these spaces. First, if $H=H_1\times H_2$ and $V=V_1\boxtimes V_2$ is the exterior tensor product of two representations of $H_1$ and $H_2$ respectively, then clearly $\bfJ_H(V)$ is compatible with tensor product
\begin{equation}
\label{E:JV tensor}
\bfJ_H(V)=\bfJ_{H_1}(V_1)\otimes_k\bfJ_{H_2}(V_2).
\end{equation}
Next, if $K$ is another linear algebraic group equipped with an automorphism $\tau$ and there is a $\tau$-equivariant homomorphism $K\to H$, we have a natural homomorphism 
\begin{equation*}
\label{E:ChRes1}
\Res^\tau_V: \bfJ_H(V) \to \bfJ_K(V)
\end{equation*} 
by restricting (a.k.a. pulling back) of the $V$-valued functions on $H$ to $K$; this is compatible with the $\bfJ_H$-module and $\bfJ_K$-module structure through the natural restriction map 
$\Res^\tau_{\bf 1}: \bfJ_H \to \bfJ_K$, and
in particular induces the map of $\bfJ_K$-modules
\begin{equation*}
\label{E:ChRes2}
\Res^\tau_V\otimes 1: \bfJ_H(V)\otimes_{\bfJ_H}\bfJ_K\to \bfJ_K(V).
\end{equation*}

\medskip

Next we discuss the compatibility of $\bfJ_H(V)$ with tensor induction.
Let $H_0$ be a linear algebraic group equipped with an automorphism $\tau_0$. Suppose we have an embedding $\langle\tau_0\rangle\subset\langle\tau\rangle$ of cyclic groups of finite index so that $\tau_0 = \tau^d$. Then we can form the tensor induction $H=\Ind_{\langle \tau_0\rangle }^{\langle\tau\rangle}H_0$, which is the group of $\tau_0$-equivariant maps from $\langle\tau\rangle$ to $H$, where $\tau_0$ acts on $\langle\tau\rangle$ by translation.
In addition, $\tau$ acts on $H$ since it acts on $\langle\tau\rangle$ by translation. 
Since $\{\id, \tau, \dots, \tau^{d-1}\}$ forms a set of representatives of $\langle \tau\rangle / \langle \tau_0 \rangle$, we may explicitly identify $H$ with the product $\prod_{i=0}^{d-1} H_0$, on which $\tau$ acts by $(h_0, \dots, h_{d-1}) \mapsto (\tau_0(h_{d-1}), h_0, \dots, h_{d-2})$.

We have two natural maps from $H_0$ to $H$, given by embedding $i_0$ into the $0$th factor of $H$ and by the diagonal embedding $\Delta$, i.e.
$$
i_0(h) = (h, 1, \dots, 1) , \quad \Delta(h) = (h, \dots, h).
$$
It is straightforward to check that these two embeddings satisfy the following relation: for $g, h \in H_0$,
$$
c_\tau(\Delta(h))(i_0(g)) = c_\tau(h, \dots, h)(g, 1, \dots, 1) = (hg\tau_0(h)^{-1}, h h^{-1}, \dots, hh^{-1}) = 
i_0 \big(c_{\tau_0}(h)(g)  \big),
$$
i.e. 
$i_0$ intertwines the $\tau_0$-twisted conjugation action of $H_0$ on the source, and the $\tau$-twisted conjugation action of $\Delta(H_0)$ on the target:
\[
\xymatrix@R=15pt@C=40pt{
H_0 \ar@^{(->}[r]^-{i_0} 
 \ar@(dl,dr)_-(.2){c_{\tau_0}}
& H\ar@(dl, dr)_-(.8){c_{\tau}}
\\
H_0 \ar@^{(->}[r]^-{\Delta} & H.
}
\]
From this, we naturally deduce a morphism of stacks 
\begin{equation}
\label{E:G_0/G_0 isom G/G}
(i_0/\Delta): [H_0 / c_{\tau_0}(H_0)] \to [H/ c_\tau(H)].
\end{equation}
If $K_0\subset H_0$ is a $\tau_0$-stable subgroup, then $K=\Ind_{\langle \tau_0\rangle }^{\langle\tau\rangle}K_0$ is a $\tau$-stable subgroup of $H$, and we obtain a commutative diagram
\[
\begin{CD}
[K_0 / c_{\tau_0}(K_0)] @>>> [K/ c_\tau(K)]\\
@VVV@VVV\\
[H_0 / c_{\tau_0}(H_0)] @>>> [H/ c_\tau(H)].
\end{CD}
\]

\begin{lem}
\label{L:reduce from G to G0}
\begin{enumerate}
\item 
The morphism $(i_0/\Delta)$ in \eqref{E:G_0/G_0 isom G/G} is an isomorphism of stacks.

\item
If $V$ is an $H$-module, regarded as a representation of $H_0$ via the \emph{diagonal} embedding $\Delta$, then we may form $\bfJ_{H_0}(V)$ and $\bfJ_{K_0}(V)$ with respect to the $\tau_0$-twisted conjugation action, and $\bfJ_{H}(V)$ and $\bfJ_{K}(V)$ with respect to the $\tau$-twisted conjugation action. Then we have a natural commutative diagram, 
$$
\xymatrix@C=40pt{
\bfJ_H(V) \otimes_{\bfJ_H} \bfJ_K \ar[d]_{(i_0 / \Delta)^*} \ar[r]^-{\Res_V^\tau \otimes 1} & \bfJ_K(V)
\ar[d]^{(i_0 / \Delta)^*}
\\
\bfJ_{H_0}(V) \otimes_{\bfJ_{H_0}} \bfJ_{K_0} \ar[r]^-{\Res_V^{\tau_0} \otimes 1} & \bfJ_{K_0}(V)
}
$$
where the two vertical morphisms are isomorphisms.
\end{enumerate}
\end{lem}
\begin{proof} 
(1) The group $\{1\} \times H_0^{d-1}$ is a coset representative of $H / \Delta(H_0)$. It is enough to show that the following map is an isomorphism
\begin{align*}
c_\tau(-, i_0(-)): (\{1\} \times H_0^{d-1}) \times H_0 \ \longto&\  H = H_0^d
\\
\big( (h_1, \dots, h_{d-1}) , g \big)   \ \longmapsto& \ c_\tau(1, h_1, \dots, h_{d-1})(i_0(g))
\\
= &\ \big(g\tau_0(h_{d-1}^{-1}), h_1, h_2h_1^{-1}, \dots, h_{d-1} h_{d-2}^{-1}\big).
\end{align*}
But this is clear, and the inverse map is given by
$$
(g_0, \dots, g_{d-1}) \mapsto \big(g_1, g_2g_1, \dots, g_{d-1}g_{d-2} \cdots g_1), g_0 \tau_0(g_{d-1}g_{d-2} \cdots g_1) \big).
$$
 So \eqref{E:G_0/G_0 isom G/G} is  an isomorphism.

(2) follows from (1) immediately. In an explicit form, the inverse of the pull back $(i_0/\Delta)^*$ is given by sending $\bar f \in \bfJ_{H_0}(V)$ to a $\tau$-twisted conjugation invariant function $f: H \to V$ defined by the following formula: for $g_0, \dots, g_{d-1} \in H_0$
\[
f(g_0, \dots, g_{d-1}) = (g_1^{-1}g_2^{-1}\cdots g_{d-1}^{-1}, \dots, g_{d-2}^{-1} g_{d-1}^{-1}, g_{d-1}^{-1}, 1) \cdot \bar f(g_{d-1}g_{d-2} \cdots g_0).\qedhere
\]
\end{proof}

\medskip

Finally, let us discuss the compatibility of $\bfJ_H(V)$ with central homomorphisms.
Let $F\subset H$ denote a $\tau$-stable central subgroup of $H$. We assume that $F$ is of multiplicative type, i.e. if $\Lambda=\Hom(F,\bG_m)$ denote its character group, then $F=\Spec k[\Lambda]$. The exact sequence 
$$1\to {\Lambda}^\tau\to \Lambda\xrightarrow{1-\tau}\Lambda\to \Lambda_\tau\to 1$$
induces 
$$1\to F^\tau\to F\xrightarrow{1-\tau} F\to F_\tau\to 1.$$ 
Let $H'=H/(1-\tau)F$ and $H''=H/F$. Then the kernel of the map $H'\to H''$ is $F_\tau$.
Left multiplication by $F$ induces an $F$-action on $k[H]$ via $(z\cdot f)(h)=f(zh)$.
Then we have the decomposition $k[H]=\oplus_{\psi\in\Lambda} k[H]_{\psi}$ according to the weights of $F$, where $k[H]_{\psi=\mathbf{1}}=k[H'']$ and each $k[H']_{\psi}$ is an invertible $k[H'']$-module. In addition, $\oplus_{\psi\in \Lambda^{\tau}} k[H]_{\psi}=k[H']$, and for given $\chi\in (1-\tau)\Lambda$, $\oplus_{\tau(\psi)-\psi=\chi} k[H]_{\tau(\psi)}$ is an invertible $k[H']$-module.

Let $V$ be an $H$-module, decomposed as $\oplus_{\chi}V_{\chi}$ according to the weights of $F$. Then
\begin{equation*}
\label{E:JHV and F}
\bfJ_H(V)=\bigoplus_{\psi\in\Lambda} \big(k[H]_{\tau(\psi)}\otimes V_{\tau(\psi)-\psi} \big)^{c_\tau(H'')}=\bigoplus_{\chi\in(\tau-1)\Lambda}  \Big(\bigoplus_{\tau(\psi)-\psi=\chi}k[H]_{\tau(\psi)}\otimes V_{\chi}\; \Big)^{c_\tau(H'')}.
\end{equation*}
Each direct summand $\big(\oplus_{\tau(\psi)-\psi=\chi}k[H]_{\tau(\psi)}\otimes V_{\chi} \big)^{c_\tau(H'')}$ is acted on by $F_{\tau}$. In particular,
\[\bfJ_H= k[H']^{c_{\tau}(H'')}\]
is acted by $F_{\tau}$, and the $\bfJ_H$-module structure on $\bfJ_H(V)$ is compatible with this action. Note that if $V$ is a representation of $H''$, then 
\begin{equation}
\label{E:invariants for JV}
\bfJ_H(V)^{F_\tau}=\bfJ_{H''}(V).
\end{equation}
In particular, $\bfJ_H^{F_\tau}=\bfJ_{H''}$. 
\begin{lem}
\label{L: fpqc descent for JV}
Let $V$ be a representation of $H''$.
Assume that the above $F_\tau$-action realizes $\Spec \bfJ_{H}$ as an $F_\tau$-torsor (in fpqc topology) over $\Spec \bfJ_{H''}$, then the natural map
\[\Res_V^\tau\otimes1: \bfJ_{H''}(V)\otimes_{\bfJ_{H''}}\bfJ_H\to \bfJ_H(V)\]
is an isomorphism.
\end{lem}
\begin{proof}
Note that the action of $F_\tau$ on $\bfJ_H(V)$ equips $\bfJ_H(V)$ with a descent datum for the map $\bfJ_{H''}\to \bfJ_{H}$. The lemma now follows from \eqref{E:invariants for JV} and faithfully flat descent.
\end{proof}

\subsection{Vector-valued twisted conjugation invariant functions on the Vinberg monoid}
Now let $H=G$ be the reductive group as fixed at the beginning of the section. Let $\tau$ be an automorphism of $G$. Let $V$ be a representation of $G$.  Sometimes we write $\bfJ(V)=\bfJ_G(V)$ and $\bfJ=\bfJ_G$ for simplicity.
We now define their their analogues $\bfJ_+(V)$ (resp. $\bfJ_0(V)$) of vector-valued twisted conjugation invariant functions on $V_G$ (resp. $\mathrm{As}_G$).

We consider the $G$-action on $V_G$ via the twisted diagonal embedding $G\to G\times G,\ g\mapsto (g,\tau(g))$ so that its restriction to $\frakd^{-1}({\bf 1})=G$ is the $\tau$-twisted conjugation action of $G$ on itself. For this reason, we also denote this action of $G$ on $V_G$ and on $\mathrm{As}_G$ by $c_\tau$.
Let $[V_G/c_\tau(G)]$ and $[\mathrm{As}_G/c_\tau(G)]$ denote the stack quotients.
The map $\frakd$ induces 
$$[\frakd]: [V_G/c_\tau(G)]\to T_\ad^+,\quad [\frakd]^{-1}({\bf 1})=[G/c_\tau(G)], \quad [\frakd]^{-1}({\bf 0})=[\mathrm{As}_G/c_\tau(G)].$$ 
Write $\widetilde{V}_+:=V_{[V_G/c_\tau(G)]}$ (resp. $\widetilde{V}_0=V_{[\on{As}_G/c_\tau(G)]}$) for the vector bundle on $[V_G/c_\tau(G)]$ (resp. $[\on{As}_G/c_\tau(G)]$) corresponding to $V$. Then clearly
$\widetilde V:=\widetilde V_+|_{[\frakd]^{-1}({\bf 1})}$ and $\widetilde V_0=\widetilde V_+|_{[\frakd]^{-1}({\bf 0})}$.
Now we can define 
$$\bfJ_+=\Gamma([V_G/c_\tau(G)],\mO), \quad \bfJ_+(V)=\Gamma([V_G/c_\tau(G)],\widetilde V_+),$$
and
$$ \bfJ_0=\Gamma([\mathrm{As}_G / c_\tau(G)],\mO),\quad \bfJ_0(V)=\Gamma([\mathrm{As}_G / c_\tau(G)],\widetilde V_0).$$
Then $\bfJ_+(V)$ (resp. $\bfJ_0(V)$) is a $\bfJ_+$ (resp. $\bfJ_0$)-module. 

Remark~\ref{R: special tau} also applies to the study of the spaces $\bfJ_+(V)$ and $\bfJ_0(V)$.
Recall that a choice of a pinning of $G$ defines a section of the projection $\Aut(G)\to \on{Out}(G)$. Therefore, to study  $\bfJ_*(V)$ for $*=+,0,\emptyset$, without loss of generality we may and will assume that $\tau=\sigma$ is an automorphism preserving the pinning $(G,B,T,\{x_\al\}_{\al\in\Delta})$ we fix at the beginning of this section.

We explain the relations between $\bfJ_+(V), \bfJ_0(V)$ and $\bfJ(V)$.
Let $V_G/\!\!/c_\sigma(G)=\Spec \bfJ_+$ be the GIT quotient. Then $[\frakd]$ factors as 
$$[V_G/c_\sigma(G)]\to V_G/\!\!/c_\sigma(G)\xrightarrow{\bar\frakd} T_\ad^+.$$
 \begin{lem}
 \label{L: gr specialization vs specialization gr}
 \begin{enumerate}
\item For any representation $V$, $\bfJ_+(V)\otimes_{\bfJ_+}k[\bar\frakd^{-1}({\bf 1})]\cong \bfJ(V)$. In particular, $k[\bar\frakd^{-1}({\bf 1})]\cong \bfJ$.
\item If $V$ admits a good filtration, then  $\bfJ_+(V)\otimes_{\bfJ_+}k[\bar\frakd^{-1}({\bf 0})]\cong \bfJ_0(V)$. In particular, $k[\bar\frakd^{-1}({\bf 0})]\cong \bfJ_0$.
\end{enumerate}
\end{lem}
\begin{proof}
Recall that there is an $\xch(T)^+_{\on{pos}}$-filtration on $k[G]$ by \eqref{E: fil on k[G]}. Then
\[\bfJ_+(V)\otimes_{\bfJ_+}k[\bar\frakd^{-1}({\bf 1})]=
\bigoplus_{\omega \in \Min}\varinjlim_{\mu\in \omega +  \xch(T)^+_{\on{pos}}}\big (\fil_\mu k[G]\otimes V\big)^{c_\sigma(G)}.\]
Since taking $G$-invariants commutes with taking direct limits,  Part (1) follows.

To prove Part (2), first notice that for $\mu\in \xch(T)^+_{\on{pos}}$, there is the following commutative diagram with all rows exact.
\begin{equation}
\label{E:lifting grmu to fil mu}
\xymatrix@C=20pt{
&\bigoplus_{i} \big(\fil_{\mu-\al_i}k[G]\otimes V\big)^{c_\sigma(G)} \ar[r] \ar[d] & (\fil_\mu k[G]\otimes V)^{c_\sigma(G)} \ar[r] \ar@{=}[d] & \big(\bfJ_+(V)\otimes_{\bfJ_+}k[\bar\frakd^{-1}({\bf 0})]\big)_\mu \ar[r] \ar[d] & 0
\\
0 \ar[r] & \big(\sum_{i} \fil_{\mu-\al_i}k[G]\otimes V\big)^{c_\sigma(G)} \ar[r] & (\fil_\mu k[G]\otimes V)^{c_\sigma(G)} \ar[r] & (\gr_\mu k[G]\otimes V)^{c_\sigma(G)}\ar[r] & 0 .
} 
\end{equation}
Here $ (\bfJ_+(V)\otimes_{\bfJ_+}k[\bar\frakd^{-1}({\bf 0})])_\mu$ denotes the $\mu$-graded piece of $\bfJ_+(V)\otimes_{\bfJ_+}k[\bar\frakd^{-1}({\bf 0})]$, and $\sum_i\fil_{\mu-\al_i}k[G]$ denotes the image of $\bigoplus_i \fil_{\mu-\al_i}k[G]\to \fil_\mu k[G]$. The second row is clearly left exact, and the right exactness follows from the fact that $\sum_{i} \fil_{\mu-\al_i}k[G]$ has a good filtration as a $G\times G$-module, by Theorem~\ref{T:prop of good fil}, Corollary~\ref{C:good filtration via Koszul}. By the same reasoning,
the left vertical map is also surjective. It follows that the right vertical map is an isomorphism. Part (2) follows.
 \end{proof}

\begin{rmk}
\label{R: inv function split}
If $V$ is the trivial representation, one can directly show the surjectivity of the map
\begin{equation*}
\label{E:grmu}
(\fil_{\mu}k[G])^{c_\sigma(G)}\to (\gr_\mu k[G])^{c_\sigma(G)}
\end{equation*} 
by constructing a splitting. Namely, note that
$$\dim (\gr_\mu k[G])^{c_\sigma(G)}=\left\{\begin{array}{cc} 0 & \sigma(\mu)\neq \mu\\ 1 & \sigma(\mu)=\mu. \end{array}\right.$$
In the latter case, the composition 
$$(\mathtt W_{\mu}^*\otimes \mathtt W_\mu)^{c_\sigma(G)}\to \fil_\mu k[G]^{c_\sigma(G)}\to (\ttS_{\mu^*}\otimes \ttS_\mu)^{c_\sigma(G)}$$ is an isomorphism, giving the splitting.

In addition, by applying Lemma~\ref{L: criterion free Rees} to the case $M=k[V_G]^{c_\sigma(G)}$, we conclude that $\bar\frakd:V_G/\!\!/c_\sigma(G)\to T_\ad^+$ is flat.
\end{rmk}

We finish this subsection with the following \emph{twisted Chevalley isomorphism}. 
Recall that we denote by $W=N(T)/T$ the Weyl group. Let  $W_0=W^\sigma$ and let $N_0$ be the preimage of $W_0$ in $N$. It acts on $T$ via the twisted conjugation $c_\sigma$ and it also acts on $V_T$ via $c_\sigma$.
\begin{prop}
\label{AE:classical HC morphism}
The restriction of a function on $V_G$ to $V_T$ induces an isomorphism $$\Res_{+,\boldsymbol{1}}^\sigma: k[V_G]^{c_\sigma(G)}\xrightarrow\cong k[V_T]^{c_\sigma(N_0)}.$$
In particular, restricting a $\sigma$-conjugation invariant function on $ G$ to $T$ induces the twisted Chevalley isomorphism 
$$
\on{Res}_{\boldsymbol 1}^\sigma: \bfJ=k[G]^{c_\sigma( G)} \xrightarrow{\cong} k[T]^{c_\sigma(N_0)}.
$$
\end{prop}
The last statement was also proved in \cite[Theorem~1]{springer}, essentially by the same argument given below.
\begin{proof}
Recall the $\xch(T)^+_{\on{pos}}$-filtration on $k[T]$ defined by \eqref{E:fil on k[T]}, which is the image of the $\xch(T)^+_{\on{pos}}$-filtration on $k[G]$ defined by \eqref{E: fil on k[G]}. Then   
$$(\fil_\nu k[T])^{c_\sigma(N_0)}: = \bigoplus_{\substack{\la \in \XX^\bullet(T)^{+,\sigma}\\ \mu - \lambda \in \XX^\bullet(T_\ad)_\mathrm{pos}}} \Big(k \cdot \sum_{\nu \in W_0 \la} e^{\nu} \Big),
$$
By Remark~\ref{R: inv function split}, it is straightforward to see that $\gr k[G]^{c_\sigma(G)}\cong \gr k[T]^{c_\sigma(N_0)}$. Therefore, $\Res_{+, \boldsymbol{1}}^\sigma$ is an isomorphism.
\end{proof}

\subsection{Freeness}

Here are the main results of this section. We define a number
\[
r_V:=\dim V|_{T^\sigma}(0).
\]

\begin{assumption}
\label{A:G simply connected}
In this subsection, let $G$ be a simply-connected semisimple group over $k$ and $V$ a $G$-representation that admits a good filtration.
\end{assumption}

\begin{theorem}
\label{T:Kostant}
Keep Assumption~\ref{A:G simply connected}. Then  $\bfJ_+(V)$ (resp. $\bfJ_0(V)$, resp. $\bfJ(V)$) is a free $\bfJ_+$-module (resp. $\bfJ_0$-module, resp. $\bfJ$-module) of rank
$r_V$. 
\end{theorem}
\begin{cor}
\label{L:cl flat}
Keep Assumption~\ref{A:G simply connected}.
The morphisms 
$$\chi_+: V_G\to V_G/\!\!/c_\sigma(G), \quad \chi_0: \on{As}_G\to \on{As}_G/\!\!/c_\sigma(G), \quad \chi: G\to G/\!\!/c_\sigma(G)$$ 
are faithfully flat.
\end{cor}
We call these morphisms the \emph{(twisted) Chevalley maps}. When $\sigma=\id$, this corollary is also proved by \cite{Bou} by a different method.
\begin{proof}
Let $\on{Reg}$ denote the regular representation of $G$, i.e. the representation of $G$ on $k[G]$ induced by left multiplication, which as mentioned before admits a good filtration. Since for any affine variety $X=\Spec R$ with a $G$-action, $(R\otimes\on{Reg})^G=R$, the statement follows from Theorem~\ref{T:Kostant}.
\end{proof}

\begin{rmk}
\label{R: flatness in general}
The Chevalley map $\chi: G\to G/\!\!/c_\sigma(G)$ is not flat in general, and therefore Theorem~\ref{T:Kostant} cannot hold for arbitrary reductive group. We refer to \cite[Proposition 4.1]{Richardson} for a discussion of this point when $\sigma=\id$.

Nevertheless, there is a sufficient condition for the flatness of $\bfJ(V)$ over $\bfJ$ when $G$ is a general connected reductive group. Let $G_\s$ be the simply-connected cover of the derived subgroup of $G$, and let $F$ be the kernel of the central isogeny $1\to F\to G':=G_\s\times Z_G\to G\to 1$. 
It is known that $\sigma$ lifts to a unique automorphism of $G_\s$ (e.g. see \cite[\S 9.16]{St}). 
We decompose $V=\oplus_{\psi} V_\psi\otimes k_\psi$ according to the central character for the action of $Z_{G}$ on $V$ so that each $V_\psi$ is a $G_{\s}$-module. Then by \eqref{E:JV tensor}
\[\bfJ_{G'}(V)=\bigoplus_{\psi} \bfJ_{G_{\s}}(V_\psi)\otimes \bfJ_{Z_{G}}(k_\psi)=\bigoplus_{\psi|_{Z^\sigma_{G}}=\mathbf{1}} \bfJ_{G_{\s}}(V_\psi)\otimes \bfJ_{Z_{G}}(k_\psi),\]
which is free over $G'/\!\!/c_\sigma(G')\cong G_s/\!\!/c_\sigma(G_s)\times (Z_{G})_\sigma$ of rank $r_V$ by Theorem~\ref{T:Kostant}. It follows from Lemma~\ref{L: fpqc descent for JV} that if the action of $F_\sigma$ on $G'/\!\!/c_\sigma(G')$ is free, then $\bfJ_G(V)$ is finite projective of rank $r_V$ over $G/\!\!/c_\sigma(G)$. In particular, if the map $F_\sigma\to (Z_G)_\sigma$ is injective, then $J_G(V)$ is finite projective over $J_G$. For example, this is the case if $G_\s=G_\der$ and $\sigma=\id$.
\end{rmk}

The rest of this subsection is devoted to the proof of Theorem~\ref{T:Kostant}.
We will first prove the statement for $\bfJ_0(V)$, and then deduce from it the statement for $\bfJ_+(V)$ and $\bfJ(V)$.
Note that by Lemma~\ref{C: fil on kG} (4),
\begin{equation*}
\bfJ_0(V)=\Gamma([ \on{As}_G / c_\sigma(G)],\widetilde V_0)= \bigoplus_{\nu\in \xch(T)^+}(\ttS_{\sigma(\nu^*)}\otimes \ttS_{\nu}\otimes V)^G,
\end{equation*}
Let us fix $\xi\in\xch(T)$. Clearly, the part of Theorem~\ref{T:Kostant} for $\bfJ_0(V)$ will follow from the following refinement.
\begin{prop}
\label{P: Rees summand}
Assume that $V$ admits a good filtration. Then
\[
\bfJ_0(V)_\xi: = \bigoplus_{\sigma(\nu)-\nu=\xi}(\ttS_{\sigma(\nu^*)}\otimes \ttS_{\nu}\otimes V)^G\]
is a finite free $\bfJ_0$-module of rank $=\dim V(\xi)$. Moreover, we may choose a basis $\{e_i\}$ of $\bfJ_0(V)_\xi$ as a $\bfJ_0$-module such that each $e_i\in (\gr_{\nu_i}k[G]\otimes V)^{c_\sigma(G)}$ for some $\nu_i\in\xch(T)^+$ satisfying $\sigma(\nu_i)-\nu_i=\xi$.
\end{prop}
\begin{proof}
Set $\xch(T)^{+,\sigma}=\{\nu\in\xch(T)^+\mid \sigma(\nu)=\nu\}$ and $\xch(T)^+_\xi=\{\nu\in\xch(T)^+\mid \sigma(\nu)-\nu=\xi\}$. Then $\xch(T)^{+,\sigma}$ naturally acts on $\xch(T)^{+}_\xi$.
We shall apply the discussions from \S~\ref{S: fil vect and Rees} to the following tuple
$$(\Ga,S,M,\{\on{fil}_sM\}_{s\in S})=\big(\xch(T)^{+,\sigma}, \xch(T)^+_\xi, V(\xi), \{\on{fil}_\nu V(\xi)\}_{\nu\in\xch(T)^+_\xi}\big).$$
The following lemma implies that  $(\xch(T)^{+,\sigma}, \xch(T)^+_\xi)$ satisfies (Can) and (DCC), and Lemma~\ref{L: sub-fil} is applicable (with $S'\subset S$ being $\XX^\bullet(T)^+_\xi\subset \XX^\bullet(T)^+ = \XX^\bullet(T_\s)^+$).

\begin{lem}
\label{L: minimalnusimplyconnected}
We equip $\xch(T)^+$ with a partial order by identifying $\xch(T)^+\cong \bN^\Delta$ as in \eqref{E: semigroup of dom weight}, i.e. $\la_1\geq \la_2$ if and only if $\langle\la_1,\check\al\rangle\geq \langle\la_2,\check\al\rangle$ for every simple coroot $\check \alpha$. Then for every $\nu_0\in \xch(T)^+$, the set
\[\{\nu\in\xch(T)^+\mid \sigma(\nu)-\nu=\xi, \nu\geq  \nu_0\}\]
has a unique minimal element, denoted by $\nu_0^h$. In addition, in each $\sigma$-orbit of simple coroots, there is at least one $\check \al$ such that $\langle \nu_0^h,\check\al\rangle=\langle \nu_0,\check \al\rangle$.

Specializing this discussion to $\nu_0 =0$, we deduce that $\XX^\bullet(T)^+_\xi = \nu_0^h + \XX^\bullet(T)^{+,\sigma}$.
\end{lem}
\begin{proof}
The last sentence of the lemma is a direct corollary of the existence and the properties of $\nu_0^h$. We focus on constructing the needed element $\nu_0^h$.
Since $G$ is a simply-connected semisimple group, we may take the set of fundamental weights $\{\omega_\al\}_{\al\in\Delta}$ such that $\langle\omega_\al,\check \beta\rangle=\delta_{\al\beta}$ for any pair $\al,\beta\in\Delta$. 
Then every weight can be written as $\xi=\sum\langle \xi,\check \al\rangle \omega_\al$.

If we write $\nu = \sum_\alpha \nu_\alpha \omega_\alpha$ with $\nu_\alpha \in \ZZ$,  the equality $\sigma(\nu)-\nu= \xi$ is equivalent to the system of equations
\[\nu_\al-\nu_{\sigma(\al)}=\langle\xi, \sigma(\check \al)\rangle, \quad \alpha \in \Delta,\]
with variables $\{\nu_\al\}\in \bZ^{\Delta}$, and the condition $\nu \geq \nu_0$ is equivalent to that \begin{equation}
\label{E:nu alpha bigger than nu0alpha}\nu_\alpha \geq \langle\nu_0, \check \al\rangle \quad \textrm{for any} \quad \alpha \in \Delta.
\end{equation}
For every $\sigma$-orbit $\mO\subset \Delta$, let $\omega_\mO=\sum_{\al\in\mO}\omega_\al$. Then after modifying a solution $\nu$ by a multiple of $\omega_\mO$ if necessary, we can always find $\nu$ satisfying the additional inequalities \eqref{E:nu alpha bigger than nu0alpha} and such that in each $\sigma$-orbit $\mO\subset\Delta$, there is at least one $\al\in\mO$ such that the equality in \eqref{E:nu alpha bigger than nu0alpha} holds. Then this is the desired $\nu_0^h$. 
\end{proof}

We now return to the proof of Proposition~\ref{P: Rees summand}.
By Proposition~\ref{L: fil via inv} and \eqref{E: comp mult}, we have an isomorphism
\[
\bfJ_0(V)_\xi=  \bigoplus_{\nu\in \xch(T)^+_\xi}(\ttS_{\sigma(\nu^*)}\otimes \ttS_{\nu}\otimes V)^G \stackrel{\oplus_\nu \ell_{\sigma(\nu),\nu}}{\cong}  \bigoplus_{\sigma(\nu)-\nu=\xi} \on{fil}_\nu V(\xi)=R_{\xch(T)^+_\xi}V(\xi),
\]
as modules over $\bfJ_0\cong k[\xch(T)^{+,\sigma}]$.

Therefore, by Lemma~\ref{L: criterion free Rees}, the lemma will follow if we can show that 
\begin{equation}
\label{E:dim of Vxi versus graded}\dim \gr_{\xch(T)^+_\xi}V(\xi)=\dim V(\xi).
\end{equation}
But by Lemma~\ref{L: minimalnusimplyconnected}, we can apply Lemma~\ref{L: sub-fil} to deduce this equality from Theorem~\ref{P: equi dim}. 
\end{proof}

By Proposition~\ref{P: Rees summand},we may choose a basis $\{e_i\}$ of $\bfJ_0(V)_\xi$ as a $\bfJ_0$-module such that each $e_i\in (\gr_{\nu_i}k[G]\otimes V)^{c_\sigma(G)}$ for some $\nu_i\in\xch(T)^+$. By the exactness of the bottom row of \eqref{E:lifting grmu to fil mu}, we can lift each element in  $\{e_i\}$ to $ \bfJ_+(V)$ so that $\widetilde{e_i}\in  (\on{fil}_{\nu_i}k[G]\otimes V)^{c_\sigma (G)}$.
By Lemma~\ref{L: graded Nakayama2} (2), the natural map
\begin{equation*}
\label{frame map}
\bigoplus_i\bfJ_+\widetilde{e_i}\to \bfJ_+(V)
\end{equation*}
is surjective. In particular, $\bfJ_+(V)$ is finitely generated over $\bfJ_+$. By Lemma~\ref{L: gr specialization vs specialization gr} (1), $\bfJ(V)$ is also finitely generated over $\bfJ$. 

Now Theorem~\ref{T:Kostant} is reduced to show that for every point $x\in V_G/\!\!/c_\sigma(G)$, the fiber of the module $\bfJ_+(V)$ over $x$ has dimension $\geq r_V$. By the semi-continuity, it is a consequence of the following. 
\begin{lem}
\label{L:generic dimension}
Over the generic point of $V_G/\!\!/c_\sigma(G)$, the rank of $\bfJ_+(V)$ is $r_V$.
\end{lem}
\begin{proof}
Since the map $V_G\to T_\ad^+$ is $T$-equivariant, and this $T$-action commutes with the $G\times G$-action, it is enough to prove a similar statement for $\bfJ_+(V)|_{\frakd^{-1}(\mathbf{1})}=\bfJ(V)$.

If $V$ is a representation of $G$, the restriction from $G$ to $T$ gives a natural map
\begin{equation}
\label{AE:generalized HC morphism}
\on{Res}_{V}^\sigma: \bfJ(V) \to (k[T]\otimes V)^{c_\sigma(N_0)},
\end{equation}
compatible with the isomorphism $\bfJ\cong k[T]^{c_\sigma(N_0)}$ from Proposition~\ref{AE:classical HC morphism}.  We call $\Res_V^\sigma$ the \emph{twisted Chevalley restriction homomorphism}. 
But unlike the case when $V$ is the trivial representation, the map \eqref{AE:generalized HC morphism} in general fails to be an isomorphism and the failure is studied in details in \cite{balagovic,khoroshkin-nazarov-vinberg} (when $\sigma=\id$). In the next section, we will also study in details (a variant of) this map.
Currently, we just notice the following. By Remark~\ref{R:strong rs} and Remark~\ref{R:srs locus of G} below, there is an affine open $N_0$-stable subset 
$$T^{\sigma\textrm{-reg}}=\big\{t\in T\,\big|\, I_t:=\{g\in G\mid gt\sigma(g)^{-1}=t\}=T^\sigma\big\}\subset T,$$ an open subset $G^{\sigma\textrm{-rs}}$ of $G$, and an isomorphism
$G \times^{N_0} T^{\sigma\textrm{-reg}}\xrightarrow \cong G^{\sigma\textrm{-rs}}$. It follows that
\begin{lem}
\label{L:generic Chevalley}
The map \eqref{AE:generalized HC morphism}, regarded as a map of coherent sheaves on $G/\!\!/c_\sigma(G)\cong T/\!\!/c_\sigma(N_0)$, restricts to an isomorphism over a (non-empty) open subset of $T^{\sigma\textrm{-}\mathrm{reg}}/\!\!/c_\sigma(N_0)$.
\end{lem}

Note that the fiber of $(k[T]\otimes V)^{c_\sigma(N_0)}$ over any point of $T^{\sigma\textrm{-reg}}/\!\!/c_\sigma(N_0)$ is isomorphic to $V^{T^{\sigma}}$, whose dimension is $\dim V^{T^\sigma}=r_V$. Therefore, the generic fiber of $\bfJ(V)$ over $\on{Spec} \bfJ$ is of dimension $r_V$. Again, since $\frakd: V_G\to T_\ad^+$ is $T$-equivariant, Lemma~\ref{L:generic dimension} follows.
 \end{proof}

\subsection{Construction of the basis}
\label{SS: basis}
In this subsection, we assume that $\cha k=0$, and $G$ is semisimple and simply-connected. Then every representation $V$ of $G$ admits a good filtration.
The proof of Theorem~\ref{T:Kostant} in fact gives a method to construct a basis of $\bfJ(V)$, from certain basis of $V$.

For every weight $\xi\in(\sigma-1)\xch(T)$, let $\bB(\xi)$ be a basis of $V(\xi)$ satisfying the conditions in Proposition~\ref{P: canonical basis}; for example, we may choose the MV basis as in \S~\ref{SS: fil via geomSat}.

For each $\bbb\in\bB(\xi)$ and each simple root $\al$, let $\varepsilon_\al(\bbb)\geq 0$ be the integer such that $\bbb\in\fil^{\al}_{\varepsilon_\al(\bbb)}V(\xi)-\fil^{\al}_{\varepsilon_\al(\bbb)-1}V(\xi)$, i.e.
\[E_\al^{\varepsilon_\al(\bbb)}(\bbb)\neq 0,\quad E_\al^{\varepsilon_\al(\bbb)+1}(\bbb)=0.\]
By Lemma~\ref{L: minimalnusimplyconnected}, there is a unique minimal $\nu_\bbb \in \XX^\bullet(T)^+$ such that $\sigma(\nu_\bbb)-\nu_\bbb=\xi$ and $\langle\nu_\bbb,\check \al\rangle\geq \varepsilon_\al(\bbb)$ for every $\al \in \Delta$.
Then by Proposition~\ref{L: fil via inv}, we have an isomorphism
\[
\ell_{\sigma(\nu_\bbb),\nu_\bbb}: (\ttS_{\sigma(\nu^*_\bbb)}\otimes \ttS_{\nu_\bbb}\otimes V)^G\to \on{fil}_{\nu_\bbb}V(\sigma(\nu_\bbb)-\nu_\bbb).\]
It follows that $\bbb$ comes from a unique element in $(\ttS_{\sigma(\nu^*_\bbb)}\otimes \ttS_{\nu_\bbb}\otimes V)^G$, denoted by $f_{\bbb,0}$. 

\begin{rmk}
\label{r:special point}
Recall that $\bfJ_0=k[\xch(T)^{+,\sigma}]$. For $\omega\in \xch(T)^{+,\sigma}$, let $q^{\omega}$ the corresponding element in $\bfJ_0$. Let $x_1\in\Spec \bfJ_0$ be the point defined by $q^{\omega}(x_1)=1$ for any $\omega$. Then by the proof of Proposition~\ref{P: Rees summand}, the fiber of $\bfJ_0(V)$ at $x_1$ is canonically isomorphic to $\oplus_{\xi\in (\sigma-1)\xch(T)}V(\xi)$, and the restriction of $f_{\bbb,0}$ to $x_1$ is just $\bbb$.
\end{rmk}

Since in characteristic zero, Schur and Weyl modules are isomorphic, there is a canonical $G$-equivariant map 
\begin{equation}
\label{E:matrix coef}
\ttS_{\nu_\bbb}\otimes \ttS_{\sigma(\nu^*_\bbb)}\cong \ttW_{\nu_\bbb}\otimes \ttS_{\sigma(\nu^*_\bbb)}\to  k[G\sigma]
\end{equation} 
given by taking (a twisted version of) matrix coefficients. Thus 
$f_{\bbb,0}$ defines an element  $f_\bbb\in (k[G\sigma]\otimes V)^G=\bfJ(V)$. 
Explicitly, we write $f_{\bbb,0}$ as $\sum_{i,j}\sigma e_i^* \otimes e_j \otimes v_{ij}$, where $\{e_i\}$ is a basis of $\ttS_{\nu_\bbb}$ and $\{e_i^*\}$ the dual basis, and $v_{ij} \in V$. Then
\begin{equation}
\label{E:Xi map}
f_{\bbb}(g\sigma) = \sum_{i,j} \langle \sigma e_i^*, g\sigma e_j\rangle \cdot  v_{ij},
\end{equation}
where $\langle \cdot, \cdot \rangle$ is the natural pairing between $\ttS_{\sigma(\nu_\bbb)} $ and $\ttS_{\sigma(\nu^*_\bbb)}$. In fact, \eqref{E:Xi map} is valid as long as $\ttS_{\nu_\bbb} \cong \ttW_{\nu_\bbb}$ (even in positive characteristic).

\begin{prop}
The collection $\{f_\bbb\mid \bbb\in \sqcup_{\xi\in (\sigma-1)\xch(T)} \bB(\xi) \}$ forms a basis of $\bfJ(V)$ as a $\bfJ$-module.
\end{prop}
\begin{proof}
This follows from the proof of the freeness of $\bfJ(V)$ over $\bfJ$. Indeed, since $\bfJ_+(V)=\sum_\nu (\fil_\nu k[G]\otimes V)^{c_\sigma(G)}$ and  the matrix coefficient map \eqref{E:matrix coef} lands in $ (\fil_{\nu_\bbb} k[G]\otimes V)^{c_\sigma(G)}$, we may regard $f_\bbb$ as a homogeneous element in $\bfJ_+(V)$, denoted as $f_{\bbb,+}$. By definition, Lemma~\ref{L: criterion free Rees} and \eqref{E:dim of Vxi versus graded} from the proof of Proposition~\ref{P: Rees summand}, the image of $\{f_{\bbb,+}\}$ under the restriction $\bfJ_+(V)\otimes_{\bfJ_+}\bfJ_0\cong \bfJ_0(V)$ forms a basis. The proof of Theorem~\ref{T:Kostant} shows that   $\{f_{\bbb,+}\}$ forms a basis of $\bfJ_+(V)$ over $\bfJ_+$. Therefore, $\{f_\bbb\}$ forms a basis of $\bfJ(V)=\bfJ_+(V)\otimes_{\bfJ_+}\bfJ$ over $\bfJ$.
\end{proof}

\section{Chevalley groups with an automorphism}
\label{S:automorphism}
In this section, we establish a few results about Chevalley groups equipped with a pinned automorphism $\sigma$. Our conventions and notations are as in \S~\ref{S:vector-valued twisted invariant functions}. In particular, we have a pinned reductive group $(G,B,T,\{x_\al\}_{\al\in\Delta})$ over $k$. We further assume that $\sigma$ is a finite order automorphism of $G$ preserving the pinning, i.e. $\sigma\circ x_\al=x_{\sigma(\al)}$ for $\al\in\Delta$.
Put $A=T/(\sigma-1)T$. Let  $W_0=W^\sigma$, which acts on $A$, and let $N_0$ be the preimage of $W_0$ in $N$, which acts on $T$ by twisted conjugation $c_\sigma$. Let $(\frakg,\frakb,\frakt,\fraku)$ denote the Lie algebra of $(G,B,T,U)$.  Write $G_\s$ for the simply-connected cover of $G$, and $T_\s$ the maximal torus of $G_\s$ which is the preimage of $T$.

\subsection{Root datum with an automorphism}
\label{SS:Root aut}
We start with some discussion of a version of folding of root systems. 
We also refer to \cite{springer} for some related discussions.

For each $\sigma$-orbit $\calO\subset \Phi(G, T)$, we write
$$
\alpha_\calO : = \sum _{\gamma \in \calO} \gamma,
$$
which belongs to $\XX^\bullet(A) \cong \XX^\bullet(T)^\sigma \subset \XX^\bullet(T)$. If we pick $\gamma\in \mO$, then $\alpha_\mO=\gamma + \sigma\gamma + \dots+ \sigma^{|\mO|-1}\gamma$, where $|\mO|$ denotes the cardinality of $\mO$.
Note that $\alpha_\calO$ may be different from the image of $\gamma$ under the usual norm map $\XX^\bullet(T) \to \XX^\bullet(A)$.

\begin{lemma}
\label{L:Phi'(G,A) is a root system}
The collection of $\alpha_\mO$ for all $\sigma$-orbits $\mO\subset \Phi( G, T)$, regarded as a subset of $\xch(A)$, has a structure of a root datum.
Let $G_\sigma$ denote the corresponding reductive group over $k$ containing $A$ as a maximal torus,\footnote{We choose the notation $G_\sigma$ for the group because it has a maximal torus $A=T_\sigma$. This does not suggest that it relates to the $\sigma$-coinvariants of $G$, whatever it means.} and let $\Phi(G_\sigma, A)$ denote the corresponding root system. Then the map $\calO \mapsto \alpha_\calO$ establishes a bijection between the set of $\sigma$-orbits in $\Delta$ and a set of  simple roots in $\Phi(G_\sigma, A)$. With this choice of simple roots of  $\Phi(G_\sigma, A)$, its subset of positive roots are  $\Phi(G_\sigma, A)^+=\{\al_\mO\mid\mO\subset\Phi(G,T)^+\}$. 
\end{lemma}
\begin{proof}
Since $\sigma$ lifts to a unique automorphism of the simply-connected cover of the derived group of $G$ (e.g. see \cite[\S 9.16]{St}), we may assume that $G$ is semisimple and simply-connected. Then $(G,T)=\prod (G_i, T_i)$, where $G_i$ is simple and simply-connected, and the action of $\sigma$ permutes the simple factors. To prove the lemma, clearly we may assume that there are $r$ simple factors, all of which are isomorphic to $(G_0,T_0)$ and are cyclically permuted by $\sigma$, and that $\sigma^r$ is an automorphism of $(\xch(T_0), \Phi(G_0,T_0))$ of order $d$.
Then  if $d = 1$, $G_\sigma=G_0$.
If $d>1$, then $\Phi(G,T)$ is of type $\mathsf A_n$, $\mathsf D_n$, or $\mathsf E_6$, and one can check case by case that the reductive group $G_\sigma$ is determined by the following table ($n \geq 1$).
	\begin{center}
		\begin{tabular}{|c|c|c|c|c|c|}
			\hline
			$G_0$ & $\SL_{2n+1}$ & $\SL_{2n+2}$ & $ \on{Spin}_{2n+2}, d=2$ & $\mathsf E_6$ & $\on{Spin}_8, d =3$
			\\
			\hline
			$G_\sigma$ & $\SO_{2n+1}$ & $\on{Spin}_{2n+3}$ & $ \on{Sp}_{2n}$ & $\mathsf F_4$ & $\mathsf G_2$
			\\
			\hline
		\end{tabular}
	\end{center}
The detailed calculations for the case $G_0=\SL_{2n+1}, \SL_{2n+2}$ and $\on{Spin}_{2n+2}, d=2$ can be found in Example~\ref{Ex:folding A2n}, Example~\ref{Ex:folding A2n+1} and Example~\ref{E: even Spin} below.	
\end{proof}
\begin{rmk}
Assume that $G$ is semisimple and simply-connected.
Note that the group of invariants $G^\sigma$ of $G$ under the $\sigma$-action is a connected semisimple group by \cite[Theorem 8.2]{St} containing $T^{\sigma}$ as a maximal torus (note however that the group of invariants is denoted by $G_\sigma$ in \emph{loc. cit.}). The root system $\Phi(G^\sigma,T^\sigma)$ is isogenous to $\Phi(G_\sigma, T_\sigma)$ but in general is not isomorphic to it. For example, if $G=\SL_{2n}$ ($n>1$) with $\sigma$ nontrivial, then $G^\sigma=\on{Sp}_{2n}$, whereas $G_\sigma=\on{Spin}_{2n+1}$.
\end{rmk}

\begin{remark}
Let $\hat G$ denote the (adjoint semisimple) Deligne--Lusztig dual group with dual torus $\hat T$ over a finite field $\kappa$ (so that the absolute root datum of $(\hat G, \hat T)$ is dual to the root datum $\Phi(G, T) \subset \XX^\bullet(T)$ and the Weil descent datum defining $(\hat G, \hat T)$ over $\kappa$ induces the $\sigma$-action on $\Phi(G,T)$). Then the maximal split torus $\hat A$ of $\hat G$ is dual to $A$. For example, if $G=\SL_n$ with the non-trivial involution $\sigma$ as in Example~\ref{Ex:folding A2n}, Example~\ref{Ex:folding A2n+1} below, then $\hat G$ is isomorphic to the projectivel unitary group $\on{PU}_n$ over $\kappa$.
Then the dual root datum of $\Phi(G_\sigma, A)$ is equal to the sub-root system of the relative root system $\Phi_\mathrm{rel}(\hat G, \hat A) \subset \XX^\bullet(\hat A)$ consisting of those $\al^\vee\in \Phi_\mathrm{rel}(\hat G, \hat A)$ such that $2\al^\vee\not\in  \Phi_\mathrm{rel}(\hat G, \hat A)$.
\end{remark}

It is easy to see (e.g. by a case-by-case inspection) that every root in $\Phi(G_\sigma,A)$ comes from one or two $\sigma$-orbits in $\Phi(G,T)$. In the latter case, the cardinality of one orbit is twice of the cardinality of the other.
\begin{dfn}
\label{D: type of roots}
A root in $\Phi(G_\sigma,A)$ is called of \emph{type $\mathsf A$} if there is a unique $\sigma$-orbit $\mO\subset\Phi(G,T)$ such that this root is $\al_\mO$.
In this case, the corresponding orbit $\mO$ is also called of \emph{type $\mathsf A$}.
Note that $\langle \al, \check \beta \rangle = 0$ for all pairs of distinct roots $\al,\beta\in\mO$. If $\al_\mO$ is a simple root in $\Phi(G_\sigma,A)$, then $\mO$ is a $\sigma$-orbit of simple roots of $\Phi(G,T)$ and all vertices in the sub Dynkin diagram corresponding to $\mO$ are isolated.

A root in $\Phi(G_\sigma,A)$ is called of \emph{type $\sfB\sfC$} if there are two $\sigma$-orbits $\mO^-$ and $\mO^+$ such that this root is $\al_{\mO^-}=\al_{\mO^+}$, and that $|\mO^-|=2|\mO^+|$. The orbit $\mO^-$ (resp. $\mO^+$) is called of \emph{type $\sfB\sfC^-$} (resp. \emph{type $\sfB\sfC^+$}).
In this case,  $\langle \sigma^{|\mO^+|}\alpha, \check \alpha\rangle =1$ for every $\alpha \in \mO^-$, and $\beta := \alpha + \sigma^{|\mO^+|}\alpha \in \calO^+$. If $\al_{\mO^-}$ is a simple root in $\Phi(G_\sigma,A)$, then $\mO^-$ is a $\sigma$-orbit of simple roots in $\Phi(G,T)$, and the sub Dynkin diagram corresponding to $\mO^-$ is a product of $|\mO^+|$ copies of the root system $\mathsf A_2$.
\end{dfn}

\begin{ex}
\label{Ex:folding A2n}
Let $G$ be $\SL_{2r+1}$ ($r\geq 1$) with row and column indices in $\{-r, \dots, r\}$, the pinning $(B,T, e)$ given by the group of standard upper triangular matrices, the subgroup of diagonal matrices, and $e=\sum_{i=-r}^{r-1} E_{i,i+1}$. Then the unique non-trivial pinned automorphism $\sigma$ is given by 
\begin{equation}
\label{Ex:sigma A2n}
\sigma(X)=J\,{^t\!X}^{-1}J \quad \textrm{for }X \in \SL_{2r+1},
\end{equation}
where $J$ is the anti-diagonal matrix, with entries  $J_{i,-i}=(-1)^{r+i}$. Then $$\Phi(G, T) = \{\,\varepsilon_i-\varepsilon_j\;|\; -r\leq i, j\leq r,\, i \neq j\},$$
where $\varepsilon_i$ is the character of $T$ given by evaluating at the $(i,i)$-entry.
Since $\sigma$ acts on $\Phi(G, T)$ by $\sigma(\varepsilon_i) = -\varepsilon_{-i}$, the $\sigma$-orbits on $\Phi(G, T)$ are
\begin{equation}
\label{E:sigma orbits in SL2r+1}
\calO_{i,j}= \{ \varepsilon_{-i} -  \varepsilon _{-j},\, \varepsilon_j -  \varepsilon _i\}, \quad 
\calO^-_i= \{\varepsilon_{-i} -  \varepsilon _0,\, \varepsilon_0 -  \varepsilon _i\},\quad  \calO_i^+= \{ \varepsilon_{-i} -  \varepsilon _i\}
\end{equation}
for $i\in \{ \pm 1, \dots, \pm r\}$ and $j \in \{ \pm1, \dots, \pm( |i|-1)\}$. They are of type $\sfA$, $\sfB\sfC^-$, and $\sfB\sfC^+$ respectively.

As $A =T / (\sigma-1)T$, its character group is
$$
\XX^\bullet(A) = \xch( T)^{\sigma}=\bigoplus_{i=1}^r\bZ(\varepsilon_{-i}-\varepsilon_i).
$$
The simple roots of $\Phi(G_\sigma, A)$ are $\varepsilon_{-i-1} - \varepsilon_{-i} + \varepsilon_{i} -\varepsilon_{i+1}$ for $i=1, \dots, r-1$ and $\varepsilon_{-1} - \varepsilon_1$. The former ones are of type $\sfA$ (being equal to $\alpha_{\mO_{i+1,i}}$) and the latter one is of type $\sfB\sfC$ (being equal to $\al_{\calO_1^-} = \al_{\calO_1^+}$).
\end{ex}

\begin{ex}
\label{Ex:folding A2n+1}
Let $G$ be $\SL_{2r}$ ($r \geq 2$) with row and column indices in $\{-r, \dots,-1, 1, \dots, r\}$, the pinning $(B,T, e)$ given by the group of standard upper triangular matrices, the subgroup of diagonal matrices, and $e=\sum_{i=-r}^{-2} (E_{i,i+1} + E_{-i-1, -i}) + E_{-1,1}$. Then the unique non-trivial pinned automorphism $\sigma$ is given by 
\begin{equation*}
\sigma(X)=J\,{^t\!X}^{-1}J \quad \textrm{for }X \in \SL_{2r},
\end{equation*}
where $J$ is the anti-diagonal matrix, with entries  $J_{i,-i} = - J_{-i,i} =(-1)^{r+i}$ for $i =1, \dots, r$. Then $$\Phi(G, T) = \{\,\varepsilon_i-\varepsilon_j\;|\; -r\leq i, j\leq r,\, i \neq j,\, ij \neq 0\},$$
where $\varepsilon_i$ is the character of $T$ given by evaluating at the $(i,i)$-entry.
Since $\sigma$ acts on $\Phi(G, T)$ by $\sigma(\varepsilon_i) = -\varepsilon_{-i}$, the $\sigma$-orbits on $\Phi(G, T)$ are
\begin{equation*}
\calO_{i,j}= \{ \varepsilon_{-i} -  \varepsilon _{-j},\, \varepsilon_j -  \varepsilon _i\}, \quad  \calO_i= \{ \varepsilon_{-i} -  \varepsilon _i\}
\end{equation*}
for $i\in \{ \pm 1, \dots, \pm r\}$ and $j \in \{ \pm1, \dots, \pm( |i|-1)\}$. They are all of type $\sfA$.

As $A =T / (\sigma-1)T$, its character group is
$$
\XX^\bullet(A) = \xch( T)^{\sigma} =\Big( \bigoplus_{i=1}^r\ZZ(\varepsilon_{-i}-\varepsilon_i) \Big) + \ZZ(\varepsilon_{-r} + \cdots+ \varepsilon_{-1}).
$$
The simple roots of $\Phi(G_\sigma, A)$ are $\varepsilon_{-i-1} - \varepsilon_{-i} + \varepsilon_{i} -\varepsilon_{i+1}$ for $i=1, \dots, r-1$ and $\varepsilon_{-1} - \varepsilon_1$. 
\end{ex}

\begin{lem}
\label{L:action Weyl group}
The action of the Weyl group $N_{G_\sigma}(A)/A$ on $A$ is identified with the natural action of $W_0$ on $A$.
\end{lem}
\begin{proof}For a root $\al\in\Phi(G,T)$, let $s_{\al}$ denote the corresponding reflection acting on $\xch(T)$.
Let $\al_\mO$ be a simple root of $\Phi(G_\sigma,A)$ and $s_{\al_\mO}$ the corresponding simple reflection. One checks easily that $s_{\al_\mO}=\prod_{\ga\in\mO} s_{\al}\in W_0$ if $\al_\mO$ is of type $\sfA$, and $s_{\al_\mO}=\prod_{\al\in \mO^+}s_{\al}\in W_0$ if $\al_\mO$ is of type $\sfB\sfC$. Therefore, $N_{G_\sigma}(A)/A\subset W_0$ as automorphisms of $A$. It remains to notice that they are isomorphic as abstract Coxeter groups, as can be checked easily from the table in Lemma~\ref{L:Phi'(G,A) is a root system}.
\end{proof}

We discuss subgroups of $G$ associated to $\sigma$-orbits of roots, which specialize to ``root $\SL_2$" when $\sigma=\id$. First, we have the following lemma.
\begin{lem}
\label{L:type BC+ root}
Let $\alpha \in \Phi(G, T)$. If the $\sigma$-orbit $\mO$ containing $\alpha$ is of type $\sfA$ or  $\sfB\sfC^-$, then 
$\sigma^{|\mO|}\circ x_\alpha = x_\alpha$. If the $\sigma$-orbit $\mO$ containing $\alpha$ is of type $\sfB\sfC^+$, then $\sigma^{|\mO|}\circ x_\alpha =- x_\alpha$.
\end{lem}
\begin{proof}
First,  if $\al$ is a simple root, then $\sigma^{|\mO|}\circ x_\alpha = x_\alpha$ since $\sigma$ acts by pinned automorphisms. 
In general, every $\sigma$-orbit $\calO$ of type $\sfA$ or of type $\sfB\sfC^-$ is conjugated to a $\sigma$-orbit of simple roots by an element $w\in W_0=W^\sigma$. 
We can choose a lifting of $w$ to a $\sigma$-invariant element $\dot{w}$ in $N$ (see \cite[(5), p. 55]{St}). For a root $\alpha \in \calO$, write $\Ad_{\dot{w}}(E_\al)=c E_{w(\al)}$ for some invertible constant $c$. Then $\ord(|\mO|)=\ord(|w(\mO)|)$ and
$$\Ad_{\dot{w}}(E_\al)=cE_{w(\al)}=\sigma^{|\mO|}(cE_{w(\al)})=\sigma^{|\mO|}(\Ad_{\dot{w}}(E_\al))=\Ad_{\dot{w}}(\sigma^{|\mO|}(E_\al)).$$
Therefore, $\sigma^{|\mO|}\circ x_\al=x_\al$. 

Next, assume that $\alpha$ belongs to a $\sigma$-orbit $\calO^+$ of type $\sfB \sfC^+$, then $\al=\beta + \sigma^{|\mO^+|}(\beta)$ for some root $\beta$ in a $\sigma$-orbit $\mO^-$ of type $\sfB\sfC^-$, and $E_\alpha=c[E_\beta, E_{\sigma^{|\mO^+|}(\beta)}]$ for some invertible constant $c$. Note that $\sigma^{|\mO^+|}$ will send $E_\beta$ to $c' E_{\sigma^{|\mO^+|}(\beta)}$ and $E_{\sigma^{|\mO^+|}(\beta)}$ to $c'^{-1}E_\beta$ for some invertible $c'$. Hence $\sigma^{|\mO^+|}(E_\alpha) = -E_\alpha$, and therefore $\sigma^{|\mO^+|} \circ x_\alpha = -x_\alpha$.
\end{proof}

Now, for a $\sigma$-orbit $\calO \subset \Phi(G, T)^+$ of positive roots, let $G_{\mO}$ be the subgroup of $G$ generated by $T$ and root subgroups $U_{\pm \alpha}$ for $\alpha\in \mO$. Clearly, this is a $\sigma$-stable reductive subgroup of $G$.  Let $U_{\mO}$ (resp. $U_{\mO, -}$) be the subgroup of $U$ generated by $U_\al$ for $\al\in \mO$ (resp. $-\al \in \mO$), and let $B_{\mO}=U_{\mO}T$. If $\mO$ is a $\sigma$-orbit of simple roots, we also let $U^{\mO}$ denote the subgroup of $U$ generated by $U_\beta$, for those positive roots $\beta$ that are not in the sub-root system spanned by $\mO$. Then $B=U^\mO B_\mO$ is a semi-direct product decomposition, and $P_\mO=U_\mO G_\mO$ is a $\sigma$-stable standard parabolic subgroup of $G$.

Note that $(G_{\mO},B_{\mO}, T, \{x_\al\}_{\al\in \mO})$ is a pinning of $G_{\mO}$.
When $G$ is semisimple and simply-connected, there are essentially two cases for its derived subgroup $G_{\mO,\on{der}}$ (which is always simply-connected). The following lemma also follows from Lemma~\ref{L:type BC+ root} (and in fact is equivalent to Lemma~\ref{L:type BC+ root}).
\begin{lem}
\label{L:rankonegroup} 
Assume that $G$ is semisimple and simply-connected.
\begin{enumerate}
\item If $\mO$ is of type $\sfA$, then $G_{\mO, \on{der}}\cong \prod_{i=1}^{|\mO|} \SL_2$ where the $\sigma$ acts by permuting factors and preserves the pinning (up to possibly rescaling the $x_\alpha$'s).
\item If $\calO$ is of type $\sfB\sfC^-$, then $G_{\mO, \on{der}}\cong\prod_{i=1}^{|\mO|/2} \SL_3$ where the $\sigma$ acts by permuting factors and $\sigma^{|\mO|/2}$ acts on each factor $\SL_3$ as in Example~\ref{Ex:folding A2n}, which also preserves the pinning (up to possibly rescaling the $x_\alpha$'s).

\item 
If $\mO$ is of type $\sfB\sfC^+$, then $G_{\mO, \on{der}}\cong \prod_{i=1}^{|\mO|} \SL_2$, where the $\sigma$ acts by
\begin{equation}
\label{E:BC+ type sigma action}
(g_1, \dots, g_{|\mO|}) 
\longmapsto \Big(\Ad_{\big(\begin{smallmatrix}
1& 0 \\0 & -1
\end{smallmatrix}\big)}(g_{|\mO|}), g_1, \dots, g_{|\mO|-1}\Big).
\end{equation}
\end{enumerate}
\end{lem}

Let $\mathbf T=B/U$ denote the abstract Cartan. Note that $\sigma$ acts on $\mathbf T$,  and $T\to B\to \mathbf T$ is a $\sigma$-equivariant isomorphism. Let $\mathbf A=\mathbf T/(1-\sigma)\mathbf T$, which is canonically isomorphic to $A$. By transport of structures, $W_0$ acts on $\mathbf A$, and for every root $\al_\mO\in\Phi(G_\sigma,A)$, $e^{\al_\mO}$ can be regarded as a regular function on $\mathbf A$.
 Let 
\begin{equation}
\label{E: btoaquot}
q_B: B \to \mathbf T\to \mathbf A
\end{equation} 
denote the quotient map.

\begin{dfn}
\label{D: divisor for roots}
Let $\mO\subset \Phi(G,A)$ be a $\sigma$-orbit. We define a divisor
\begin{equation}
 \label{E:disc divisors}
\nonumber
\mathbf A_{\mO} = \begin{cases}
\{t\in \mathbf A\mid e^{\al_\mO}(t)=1\}&  \mO \mbox{ is of type } \sfA \textrm{ or } \sfB\sfC^-,
\\
\{t\in \mathbf A\mid e^{\al_\mO}(t)=-1\} &  \mO \mbox{ is of type } \sfB\sfC^+.
\end{cases}
\end{equation}
Note that $\mathbf A_{\mO}=\mathbf A_{-\mO}$.
Let
\[\mathring{\mathbf A}=\mathbf A-\cup_{\mO} \mathbf A_{\mO},\]
where the union is taken over all $\sigma$-orbits $\mO\subset \Phi(G,T)$.
This is a $W_0$-invariant open subset of $\mathbf A$. For a $\sigma$-orbit $\mO\subset\Phi(G, T)$, let 
$$
\mathbf A^{[\mO]}=\mathbf A-\bigcup_{{\mO'}\neq \mO}\mathbf A_{{\mO'}}.
$$
In particular, it is an open neighborhood of the generic points of $\bfA_\calO$ in $\bfA$.

Let $\mathring{B}$ and $\mathring{T}$ be the preimage of $\mathring{\mathbf A}$ under the natural projections $q_B: B\to \mathbf A$ and $T\cong \mathbf T \to  \mathbf A$. SImilarly, let $B^{[\mO]}$ and $B_{\mO'}^{[\mO]}$ be the preimage of $\mathbf A^{[\mO]}$ under the projections $B\to\mathbf A$ and $B_{\mO'}\to \mathbf A$.
Finally, via the isomorphism $A\cong \mathbf A$, we have similarly defined spaces $A_\mO,\mathring{A}$ and $A^{[\mO]}$.
\end{dfn}

\begin{remark}
\label{R: irr and red of AO}
Note that $A_\mO$ may not be irreducible nor reduced in general. For example, if $G_\sigma=\on{Sp}_{2n}$, then every long root in $\Phi(G_\sigma,A)$ is twice of a weight of $A$ and therefore the corresponding divisor $A_{\mO}$ consists of two connected components if $\on{char} k>2$, and is non-reduced if $\on{char} k=2$. On the other hand, if $\on{char} k>2$, then $A_{\mO}$ is always reduced.
\end{remark}

\begin{lem}
\label{L:branched locus}
Assume that $G$ is semsimple and simply-connected. Then $\mathbf A/\!\!/W_0$ is isomorphic to an affine space, and the natural morphism $\mathbf A \to \mathbf A/\!\!/W_0$ is a finite and flat $W_0$-cover, with branch loci $\cup_{\mO} \mathbf A_{\mO}$. In particular, the restriction $\mathring{\mathbf A}\to\mathring{\mathbf A}/\!\!/W_0$ is a finite \'etale Galois cover.
\end{lem}
\begin{proof}
We may identify the action of $W_0$ on $\mathbf A$ with the action of $W_0\cong N_{G_\sigma}(A)/A$ on $A$ by Lemma~\ref{L:action Weyl group}.
Note that by table in the proof of Lemma~\ref{L:Phi'(G,A) is a root system}, $G_\sigma$ is isomorphic to a product of simply-connected groups and odd orthogonal groups. Therefore, $k[\mathbf A]^{W_0}$ is a polynomial algebra (see \cite{St3}), and $\mathbf A\to \mathbf A/\!\!/W_0$ is finite flat.
The branch locus of the covering $\mathbf A \to \mathbf A/\!\!/W_0$ is the union of divisors $\{t\in \mathbf A\mid \check{\al}_\mO(e^{\al_\mO}(t))=1, \al_\mO\in \Phi(G_\sigma,A)\}$. Here $ \check{\al}_\mO$ denotes the coroot of $\al_\mO$, regarded as a homomorphism $ \check{\al}_\mO:\bG_m\to A\cong \mathbf A$. Its construction is as follows. Let $\ga\in \mO$ if $\al_\mO$ is of type $\sfA$ or $\ga\in\mO^+$ if $\al_\mO$ is of type $\sfB\sfC$. Let $\check{\ga}$ be the corresponding coroot, regarded as a homomorphism $\bG_m\to T$. Then $\check{\al}_\mO$ is the composition of this homomorphism with the projection $T\to A\cong \mathbf A$.
Again, by checking the table in the proof of Lemma~\ref{L:Phi'(G,A) is a root system}, one sees that $\ker \check{\al}_\mO$ is trivial if $\al_\mO$ is of type $\sfA$ and is $\{\pm 1\}$ if $\al_\mO$ is of type $\sfB\sfC$. The lemma follows.
\end{proof}

\subsection{Twisted conjugacy classes}
In this subsection, we study $\sigma$-regular elements of $G$, generalizing some of the well-known results of Steinberg \cite{St2} for $\sigma=\id$. Some results with restriction of the characteristic of $k$ were also obtained by Mohrdieck \cite{M} before.

Let $G$ be a reductive group over $k$.
Recall the $\sigma$-twisted conjugation (or $\sigma$-conjugation for brevity)
\[c_\sigma: G\times G\sigma\to G\sigma,\quad (h,g\sigma)=hg\sigma(h)^{-1}\sigma=:c_\sigma(h)(g).\]
Since $\sigma$ preserves $(B,T)$, it acts on the set $\Delta\subset \Phi(G,T)$ of simple roots.

Let $I$ be the \emph{centralizer group scheme} for the action of $G$ on itself by $c_\sigma$, i.e. it is the group scheme over $G$ defined by the Cartesian diagram
\[
\begin{CD}
I@>>> G\times G\sigma\\
@VVV@VVc_\sigma\times\pr_2 V\\
G@>>>G\times G\sigma,
\end{CD}
\]
where $\pr_2$ denotes the projection to the second factor and $G\to G\times G\sigma$ denotes the ``diagonal" embedding $g\mapsto (g,g\sigma)$. Its fiber over $g\in G$ is denoted by $I_g$.
\begin{dfn}
Let $G^{\sigma\textrm{-reg}}$ denote the \emph{$\sigma$-regular locus} of $G$. Namely,  $$G^{\sigma\textrm{-reg}} = \{g\in G\mid \dim I_g= \dim T^\sigma=:r\} \subset G.$$
Let $B^{\sigma\textrm{-reg}}:=G^{\sigma\textrm{-reg}}\cap B$, $T^{\sigma\textrm{-reg}}=G^{\sigma\textrm{-reg}}\cap T$ and $U^{\sigma\textrm{-reg}}=G^{\sigma\textrm{-reg}}\cap U$.
\end{dfn}
\begin{rmk}
\label{R:general reg element}
(1) It will be clear from the following discussion that $\dim I_g\geq r$ for any $g\in G$. Therefore, $G^{\sigma\textrm{-reg}}$ is an open subset of $G$ by semi-continuity.

(2) When $\on{char} k = 0$, $\dim I_g = \on{Lie} I_g = \ker(\id - \Ad_g\sigma: \frakg \to \frakg)$. Complication arises in the positive characteristic case; for example, the center $\frakz$ of $\frakg$, i.e.
$$
\frakz: = \bigcap _{\alpha \in \Phi(G, T)} \Ker(d\alpha: \frakt \to  k),
$$
may not be trivial even if $G$ is semisimple. 

(3) For every automorphism $\tau$ of the algebraic group $G$, one can define the open subset of $\tau$-regular elements $G^{\tau\textrm{-reg}}\subset G$ as those $g\in G$ such that $I_g$ achieves the minimal dimension. But since every $\tau$ differs from some pinned automorphism $\sigma$ by an inner automorphism, the study of $G^{\tau\textrm{-reg}}$ reduces to the study of $G^{\sigma\textrm{-reg}}$, by virtual of Remark~\ref{R: special tau}. 
\end{rmk}

We first study $U^{\sigma\textrm{-reg}}$. Let  $V= [U,U]\cdot (\sigma-1)(\prod_{\al\in\Delta} U_{\al})\subset U$, and let
$$q_U:U\to W=U/V$$
denote the quotient map. 
The natural conjugation of $T^\sigma$ on $U$ induces an action of $T^\sigma$ on $W$, and there is a unique $T^\sigma$-open orbit $\mathring{W}\subset W$. Explicitly, if we choose an element $\al\in\mO$ for each $\sigma$-orbit of simple roots so that there is an isomorphism 
$$\prod_{\textrm{chosen }\al} x_\al: \bG_a^{|\Delta^\sigma|}\to U\to W$$
giving a basis of $W$ (as a vector space over $k$), then $v\in \mathring{W}$ if and only if all the coefficients of $v$ with respect to this basis are non-zero. Let $\mathring{U}=q_U^{-1}(\mathring{W})\subset U$.

We say that an element $w\in W$ is a \emph{$\sigma$-twisted Coxeter element} in $W$, if one can choose an element $\al\in\mO$ for each $\sigma$-orbit $\mO$ of simple roots and write $w$ as the product of simple reflections corresponding to these chosen $\al$'s (in some order).   Let $\dot{w}$ be a representative of $w$ in $N$.
\begin{lem}
\label{L:Steinberg lemma}
Assume that $G$ is simply-connected.
Let $w$ be a $\sigma$-twisted Coxeter element in $W$. Then for every $b\in B$,
$$\dim \ker(\id-\Ad_{b\dot{w}}\sigma: \frakg\to\frakg)\leq \dim \frakz^\sigma+r.$$
\end{lem}
\begin{proof}
Given the properties of $\sigma$-twisted Coxeter elements presented in \cite[Theorem 7.6]{Sp1}, 
the proof of \cite[Lemma 4.3]{St2} applies literally to our situation.
\end{proof}

Now we come to the first key result.
\begin{prop}
\label{L: unipotent regular}
For each $\sigma$-orbit of simple roots $\mO \subset \Phi(G, T)$, choose one $\al\in\mO$ and define $$x:=\prod_{\al} x^{-1}_{\al}(1) \in U,$$ where the product is taken over the chosen simple roots in a fixed order. Then
\begin{enumerate}
\item $(\id-\Ad_x\sigma)^{-1}(\fraku)=\frakt^\sigma+\fraku$.
\item Let $g\in G$. If $gx\sigma(g)^{-1}\in U$, then $g\in B$. In addition, $I_x\subset B$.
\item $x \in U^{\sigma\textrm{-}\mathrm{reg}}$; in particular, $U^{\sigma\textrm{-}\mathrm{reg}}$ is nonempty. 
\item $\mathring{U}=U^{\sigma\textrm{-}\mathrm{reg}}$ is a single orbit under the $\sigma$-conjugation action of $T^\sigma\cdot U$ on $U$. In particular, $I_u\subset B$ and $(\id-\Ad_u\sigma)^{-1}(\fraku)=\frakt^\sigma+\fraku$ for every $u\in  U^{\sigma\textrm{-}\mathrm{reg}}$.
\end{enumerate}
\end{prop}
\begin{proof}
Write $\frakg_x$ for the Lie algebras of $I_x$, and similarly $\frakb_x$ and $\fraku_x$ for the Lie algebra of $I_x\cap B$ and $I_x\cap U$. We have the following lemma. 
\begin{lem}
\label{L: aux for reg}
We have $\dim \fraku_x=|\Delta^\sigma|$ and $\frakg_x=\frakb_x\subset\frakt^\sigma+\fraku$. 
\end{lem}
\begin{proof}The inclusion $\frakb_x\subset\frakt^\sigma+\fraku$ is clear. We prove the two equalities in the statement.
We first assume that $G$ is semisimple and simply-connected. Then $r = \dim T^\sigma = |\Delta^\sigma|$.
Let $w_0$ be the longest element in the Weyl group and $\dot{w}_0\in N_0$ a representative of $w_0$.  Then $\dot{w}_0x\dot{w}_0\in B\dot{w}B$ for some $\sigma$-twisted Coxeter element $w$. 
Then as in \cite[Theorem 4.6]{St2}, 
\begin{equation*}
\label{E:bg inequality}
\dim \frakz^\sigma+\dim \frak u_x\leq \dim \frakb_x\leq \dim  \frakg_x\leq \dim  \frakz^\sigma+r.
\end{equation*}
Therefore
$\dim \fraku_x\leq r$ and $\dim (I_x\cap U)\leq r$. On the other hand, $U$ acts on fibers of $q_U$ via $\sigma$-conjugation. Therefore, $\dim(I_x\cap U)\geq \dim W=r$. 
Putting them together gives $\dim \fraku_x=\dim (I_x\cap U)=r$ and $\frakb_x=\frakg_x$.  Now for a general reductive group $G$, let $G_\s$ be the simply-connected cover of its derived group. Then $\sigma$ lifts to a unique automorphism of $G_\s$ (e.g. see \cite[\S 9.16]{St}). Since the central isogeny $G_\s\to G$ induces an isomorphism on unipotent subgroups, we have $\dim(I_x \cap U) = |\Delta^\sigma|$. In addition, since the kernels and the cokernels of the two maps $\frakb_{\s}\to \frakb$ and $\frakg_{\s}\to \frakg$ are equal, and since $(\frakb_{\s})_x = (\frakg_{\s})_x$, $\frakb_x = \frakg_x$ also holds for $G$. 
\end{proof}
Now we prove the proposition.
(1) First, note that $(\id-\Ad_x\sigma)(\fraku)\subset \frakv$, where $\frakv$ is the Lie algebra of $V$. But since $\dim \fraku_x=\dim \fraku-\dim \frakv$, we have $(\id-\Ad_x\sigma)(\fraku)=\frakv$. 
Next, if $G$ is of adjoint type, then the composition of maps 
$$\frakt^\sigma\xrightarrow{\id-\Ad_x\sigma} \fraku\to \fraku/\frakv$$ 
is an isomorphism. It follows that in this case $(\id-\Ad_x\sigma)(\frakt^\sigma+\fraku)=\fraku$. Since $\frakg_x=\frakb_x\subset \frakt^\sigma+\fraku$, we see that $(\id-\Ad_x\sigma)^{-1}(\fraku)=\frakt^\sigma+\fraku$. 
In general, let $G_\ad$ be the adjoint quotient of $G$. Since the central isogeny induces an isomorphism $\fraku\cong \fraku_\ad$, it is easy to see that $(\id-\Ad_x\sigma)^{-1}(\fraku)\subset\frakb$. In addition, if $Y\in\frakb$ such that $Y-\Ad_x\sigma(Y)\in \fraku$, then $(Y \mod \fraku)\in \frakt^\sigma$. Therefore, $(\id-\Ad_x\sigma)^{-1}(\fraku)=\frakt^\sigma+\fraku$ holds in general.

(2) Write $g=u_1nu_2$ under the Bruhat decomposition for some $n\in N, u_1,u_2\in U$. Then $nu_2x\sigma(u_2)^{-1}=u_1^{-1}u\sigma(u_1)\sigma(n)$ for some $u\in U$. It follows that $n=\sigma(n)$ by the uniqueness of the Bruhat decomposition. Then $nu_2x\sigma(u_2)^{-1}n^{-1}= u_1^{-1}u\sigma(u_1)\in U$.  Since $q_U(u_2x\sigma(u_2)^{-1})=q_U(x)\in \mathring{W}$,  we must have $n \in T$ and therefore $g\in B$. It follows that $I_x^\red\subset  B$. Together with $\frakg_x = \frakb_x$, we deduce that $I_x \subset B$.

(3) Let $b\in I_x\subset B$. We write $b=tu$ according to the decomposition $B=TU$. Then
$tux = x\sigma(t) \sigma (u)$. By projecting along $B\to \mathbf T$, we see that $\sigma(t)=t$, and thus $t^{-1}xtx^{-1}=ux\sigma(u)^{-1}x^{-1}\in V$. This shows that $t\in (\cap_{\al\in\Delta} \ker\al)^\sigma$, which has dimension $\dim T^\sigma - |\Delta^\sigma|$, and that $u\in I_x\cap U$. 
Therefore, 
$$\dim I_x= (\dim T^\sigma - |\Delta^\sigma|) + \dim( I_x \cap U )= \dim T^\sigma,
$$ i.e. $x$ is $\sigma$-regular. 

(4) Note that $q_U^{-1}(x)\subset U$ is a single orbit under the $\sigma$-conjugation action of $U$ on itself. Indeed, since $U$ is unipotent, the orbit through $x$ is a closed subset of $q_U^{-1}(x)$. On the other hand, since $\dim(I_x\cap U)=\dim W$, the dimension of this orbit is equal to the dimension of $q_U^{-1}(x)$. Therefore, $q_U^{-1}(x)$ is an orbit.
Now, let $u\in \mathring{U}$. After taking a conjugation by elements in $T^\sigma$, we may assume that $q_U(x)=q_U(u)$, and therefore $x$ and $u$ are $\sigma$-conjugate by an element in $U$. This shows that $\mathring{U}\subset U^{\sigma\textrm{-}\mathrm{reg}}$, and $\mathring{U}$ is a single $T^\sigma\cdot U$-orbit. 
On the other hand, let $u\in U^{\sigma\textrm{-}\mathrm{reg}}$. Since $\dim I_u= \dim T^\sigma$, the $T^\sigma\cdot U$-orbit through $u$ is $\dim U$-dimensional, and therefore must meet $\mathring{U}$. This implies that $\mathring{U}= U^{\sigma\textrm{-}\mathrm{reg}}$. The last statement follows from Part (1) and (2).
\end{proof}

\begin{lem}
\label{L: regular springer fiber}
An element $u\in U$ is $\sigma$-regular if and only if the set $\mB_u:=\{gB\in G/B\mid g^{-1}u\sigma(g)\in B\}$ is finite.
\end{lem}
\begin{proof} Proposition~\ref{L: unipotent regular} implies that if $u\in U^{\sigma\textrm{-}\mathrm{reg}}$, then $\mB_u$ consists of only one element. Now, let $u\in U- U^{\sigma\textrm{-}\mathrm{reg}}$. Then after a $\sigma$-conjugation by an element in $U$, we may assume that there is a $\sigma$-orbit $\mO$ of simple roots, such that $u\in U^\mO$, where $U^\mO$ is the subgroup of $U$ introduced before Lemma~\ref{L:rankonegroup}. Then $\mB_u$ contains a positive dimensional subvariety $(G_\mO/B_\mO)^\sigma$. The lemma is proved.
\end{proof}

Next, we study $T^{\sigma\textrm{-reg}}$, and then $B^{\sigma\textrm{-reg}}$. Recall the map $q_B$ as in \eqref{E: btoaquot} and Defintion~\ref{D: divisor for roots}.
For $t\in T(k)$, let $\Phi(G,T)_t\subset\Phi(G,T)$ be the smallest sub-root system containing those $\sigma$-orbits $\mO$ such that $q_B(t\sigma)\in \mathbf A_\mO(k)$. We allow $\Phi(G,T)_t=\emptyset$ if $t\in\mathring{T}$.
Let $G_t\subset G$ denote the corresponding reductive subgroup, and $B_t=G_t\cap B$ a Borel subgroup of $G_t$. Its unipotent radical $U_t$ is the subgroup generated by $\{U_\al\}$, for those positive roots $\al\in \Phi(G,T)_t$. We write $B_t=U_tT$.  
Let $\frakg_t$ denote the Lie algebra of $G_t$.

\begin{lem}
\label{L: root system Phit}
The sub-root system $\Phi(G,T)_t$ is exactly the union of $\sigma$-orbits of roots in $\Phi(G,T)$ of the following three types: 
\begin{enumerate}
\item $\mO$, if $\mO$ is of type $\sfA$ and $e^{\al_\mO}(t)=1$; 
\item  $\mO^-\cup \mO^+$, if $\mO^\pm$ is of type $\sfB\sfC^\pm$ and $e^{\al_\mO}(t)=1$; 
\item $\mO$, if $\mO$ is of type $\sfB\sfC^+$ and $e^{\al_\mO}(t)=-1$. 
\end{enumerate}
\end{lem}
\begin{proof}
Indeed, the smallest sub-root system containing those $\sigma$-orbits $\mO$ such that $q_B(t\sigma)\in \mathbf A_\mO(k)$ would necessarily contain the these roots. Therefore, it is enough to show that they indeed form a sub-root system of $\Phi(G,T)$, or equivalently, if $\mO_1,\mO_2$ are two $\sigma$-orbits from the above types, and $\al_i\in\mO_i$, then the $\sigma$-orbit $\mO$ containing $s_{\al_i}(\al_j)$ would also be one of the above types. To check this, write $\beta=s_{\al_i}(\al_j)$, and $\beta_\mO$ the sum over the $\sigma$-orbit of $\beta$. We may assume that $\beta\neq \pm\al_j$ so in particular $\al_i$ and $\al_j$ are in the same irreducible factor of the root system $\Phi(G,T)$. Assume that $\tau=\sigma^r$ fixes this irreducible factor $\Phi$. 
If $(\Phi,\tau)$ is as in Example~\ref{Ex:folding A2n}, one checks directly that $\beta_\mO$ also belongs to one of the above types. If $(\Phi,\tau)$ is not as in Example~\ref{Ex:folding A2n}, then it is readily to see that $\beta_\mO$ is an integral linear combination of $\al_{\mO_1}$ and $\al_{\mO_2}$ and all $\al_{\mO_1}$, $\al_{\mO_2}$ and $\beta_\mO$ are of type (1). The lemma is proven.
\end{proof}

\begin{lem}
\label{L: pinned auto of Gt}
Let $\Delta_t\subset \Phi(G,T)_t$ be the set of simple roots (with respect to $(B_t,T)$). Then
after possible rescaling $x_a$'s, the automorphism $G_t\to G_t, g\mapsto t\sigma(g)t^{-1}$ preserves the pinning $(G_t,B_t,T, \{x_\al\}_{\al\in\Delta_t})$.
\end{lem} 
\begin{proof}Note that if $\mO\subset \Delta_t$ is a $\sigma$-orbit, then $e^{\al_\mO}(t)=1$ if $\mO$ is of type $\sfA$ and $\sfB\sfC^-$, and $e^{\al_\mO}(t)=-1$ if $\mO$ is of type $\sfB\sfC^+$. Then the lemma follows from Lemma~\ref{L:rankonegroup}. 
\end{proof}

\begin{lem}
\label{L: Jordan decom}
Let $b=ut\in U_tT$. Then $\id-\Ad_b\sigma: \frakg/\frakg_t\to\frakg/\frakg_t$ is an isomorphism. In particular,
\begin{equation}
\label{E: ker in gt}
\ker(\id-\Ad_b\sigma:\frakg\to\frakg)=\ker(\id-\Ad_b\sigma: \frakg_t\to\frakg_t),
\end{equation}
and $\dim I_b= \dim (I_b\cap G_t)$.
\end{lem}
\begin{proof}
For $i \in\bZ$, let  $\Phi_{i}$ denote the set of positive roots $\alpha$ of height $i$, i.e. $\langle\alpha,\check\rho \rangle =i$, where $\check \rho$ is the half of the sum of positive coroots.
We choose the following basis of $\frakg$ (in the given order): first a basis in $\frakg_t$; then
the standard root basis $\{E_\al\}$ associated to the roots in $\dots, \Phi_i, \Phi_{i-1}, \dots$ but not in $\Phi(G,T)_t$, with $\sigma$-orbits grouped together; and finally the standard root basis $\{E_\al\}$ associated to the corresponding negative roots in $\dots, \Phi_{i-1}, \Phi_{i}, \dots$ but not in $\Phi(G,T)_t$, with $\sigma$-orbits grouped together.
With respect to this choice of basis, the linear operator $\id -\Ad_{b} \sigma$ is represented by a block upper triangular matrix $M$, where the first block corresponds to $\frakg_t$, and other blocks correspond to $\sigma$-orbits of roots not in $\Phi(G,T)_t$.

Note that a diagonal block that corresponds to a $\sigma$-orbit $\calO'$ not in $\Phi(G,T)_t$ is invertible due to: (i) its determinant is $\pm (e^{\alpha_\calO'}(t) - 1)$ if $\calO'$ is of type $\sfA$ or $\sfB\sfC^-$ and $\pm (e^{\alpha_\calO'}(t) + 1)$ if $\calO'$ is of type $\sfB\sfC^+$, which follows from Lemma~\ref{L:type BC+ root} and an easy computation; and (ii) $q_B(t) \notin \bfA_{\calO'}$.  The first claim of the lemma follows. Then clearly \eqref{E: ker in gt} holds, which in turn implies that $\Lie (I_b^\red)\subset \Lie I_b\subset \frakg_t$. Therefore the neutral connected component of $I_b^\red$ is a closed subgroup of $G_t$. The lemma is proved.
\end{proof}
A very similar argument yields the following ``$\sigma$-twisted" Jordan decomposition.
\begin{lem}
\label{L: Jordan decom2}
Let $b=ut\in UT$. Then $b$ is $\sigma$-conjugated by an element in $U$ to an element $u't$ with $u'\in U_t$.
\end{lem}
\begin{proof}
Let $u_1=u$.
By induction on $i$, one can show that we can $\sigma$-conjugate $b$ by an element in $U$ to $u_it$, where $u_i$ is in the subgroup of $U$ generated by $U_t$ and $U_{\al}$ for $\al\in \Phi_i\cup\Phi_{i+1}\cup\cdots$. This is because if $t\not\in\bfA_{\mO'}$, then $\id-\Ad_t\sigma$ is invertible on the space $\oplus_{\al\in\mO'} kE_\al$.
\end{proof}

\begin{lem}
\label{E: sigmaregularinT}
We have equalities $T^{\sigma\textrm{-}\mathrm{reg}}=\mathring{T}$. 
\end{lem}
\begin{proof}
Note that for $t\in \mathring{T}$, $G_t=T$, and $I_t\cap T=T^\sigma$. It follows from Lemma~\ref{L: Jordan decom} that $t\in T^{\sigma\textrm{-}\mathrm{reg}}$.

Next we show that $T^{\sigma\textrm{-reg}} \subset \mathring T$, i.e. if $q_B(t\sigma)\in\bfA_{\mO}(k)$ for some $\sigma$-orbit $\mO \subset \Phi(G, T)$, then $I_t$ contains a unipotent subgroup and therefore $\dim I_t>\dim T^\sigma$ (as clearly $T^\sigma \subset I_t$).  If $\mO$ is of type $\sfA$ or $\sfB\sfC^+$, pick $\alpha \in \mO$ and write $\alpha_i = \sigma^i(\alpha)$ for $i = 1, \dots, |\mO|$.  These $\alpha_i$'s  are pairwise orthogonal and we may assume that $\sigma \circ x_{\alpha_i} = x_{\alpha_{i+1}}$ for $i = 1, \dots, |\mO|-1$ and thus $\sigma \circ x_{|\mO|}$ is equal to $x_{\alpha_1}$ if $\calO$ is of type $\sfA$ and is equal to $-x_{\alpha_1}$ if $\mO$ is of type $\sfB\sfC^+$. One checks that, for $c \in k$, 
$$
\prod_{i=1}^{|\mO|} x_{\alpha_i} \big(e^{\alpha_2+ \cdots+ \alpha_{i}}(t)c\big)
$$
form a unipotent subgroup of $I_t$. 

It remains to consider the case $\mO$ is of type $\sfB\sfC^-$ and $\cha k\neq 2$. Let $d=|\mO|/2$. Pick $\al\in\mO$ and write $\al_i=\sigma^i(\al)$ for $i=1,\ldots,2d$ and $\beta_i=\al_i+\al_{i+d}$ for $i=1,\ldots,d$. We may assume that 
\begin{itemize}
\item $\sigma \circ x_{\al_{2d}} = x_1$, $\sigma \circ x_{\al_i} = x_{\alpha_{i+1}}$ for $i =1, \dots, 2d-1$, and
\item$x_{\al_i}(u)x_{\al_{d+i}}(v)=x_{\al_{d+i}}(v)x_{\al_i}(u)x_{\beta_i}(uv)$ for $i=1, \dots, d$. 
\end{itemize}
This in particular implies that $\sigma \circ x_{\beta_i} = x_{\beta_{i+1}}$ for $i = 1, \dots, d-1$, and $\sigma\circ x_{\beta_d} = - x_{\beta_1}$.  Now one can check that for every $c\in k$, $$\prod_{i = 1}^{d} \Big( x_{\al_i}\big(e^{\al_2+\cdots+\al_{i}}(t)c\big)\cdot  x_{\al_{d+i}}\big(e^{\al_2+\cdots+\al_{d+i}}(t)c\big) \cdot x_{\beta_i}\big(-\tfrac{1}{2}e^{\beta_2+\cdots+\beta_{i}+\al_2+\cdots+\al_{d+1}}(t)c^2\big)\Big)\in I_t,$$
giving the desired unipotent subgroup. 
\end{proof}
\begin{rmk}
\label{R:strong rs} 
We say an element $t\in T$ strongly $\sigma$-regular if $I_t=T^\sigma$. We claim that they form a non-empty affine open subset of $T^{\sigma\textrm{-reg}}$, sometimes denoted by  $T^{\textrm{s.}  \sigma\textrm{-reg}}$.
Indeed, let $t\in \mathring{T}$, and let $g\in I_t(k)$. Using the Bruhat decomposition of $g=u_1nu_2$ with $u_1,u_2\in U$ and $n\in N$, it is easy to see that $g=n\in N^\sigma$, and $g\in T^\sigma$ if and only if no non-trivial elements in $W_0$ fixes the image of $t$ under the projection $T\to A$.  But the last condition clearly defines a non-empty affine open subset of $T$, verifying the claim.  Note that if $G$ is semsimple and simply-connected, then by Lemma~\ref{L:branched locus}, $T^{\textrm{s.}  \sigma\textrm{-reg}}=T^{\sigma\textrm{-reg}}=\mathring{T}$.
\end{rmk}

To continue, we need the following lemma.
We identify the tangent space $T_gG$ of $G$ at $g$ with $\frakg$ via the right translation $R_g$ by $g$.
Let $(h,g\sigma)\in G\times G\sigma$ and $g'=c_\sigma(h)(g)$. 
A direct calculation shows the following. 
\begin{lem}
\label{E:diff of csigma}
The differential of $c_\sigma$ at $(h,g\sigma)$ is
\[
dc_\sigma: T_{h}G\oplus T_{g\sigma}G\sigma\to T_{g'}G\sigma, \quad (X,Y)\mapsto X-\Ad_{g'}(\sigma(X))+\Ad_{h}(Y).
\]
\end{lem}

Recall that we denote by  $N_0$ the preimage of $W_0=W^\sigma\subset W$ in $N$. It acts on $T$ via the twisted conjugation $c_\sigma$ preserving $T^{\sigma\textrm{-reg}}$. Consider the map
\begin{equation}
\label{E:conjugation}
G \times^{N_0} T^{\sigma\textrm{-reg}}\to G.
\end{equation} 
\begin{lem}
\label{L:sigmars}
The map \eqref{E:conjugation} is an open embedding.
\end{lem}
 Let $G^{\sigma\textrm{-}\mathrm{rs}}$ denote the image of this map, called the \emph{$\sigma$-regular semisimple locus} of $G$.  Note that $G^{\sigma\textrm{-rs}}\cap T=T^{\sigma\textrm{-reg}}$ by Lemma~\ref{E: sigmaregularinT}.
\begin{proof}It follows from Lemma~\ref{E:diff of csigma} that the map is \'etale. It remains to prove that it is also injective. Assume that $gt\sigma(g)^{-1}=t'$ with $t,t'\in T^{\sigma\textrm{-reg}}=\mathring{T}$. Using the Bruhat decomposition $g=u_1nu_2$ as in Remark~\ref{R:strong rs}, one deduces that $u_1=u_2=1$ and $t\sigma(n)=nt'$. The injectivity follows.
\end{proof}
\begin{rmk}
\label{R:srs locus of G}
Note that $N_0$ also preserves $T^{\textrm{s.}  \sigma\textrm{-reg}}$, and the map $G\times^{N_0}T^{\textrm{s.}  \sigma\textrm{-reg}}\to G$ is open. The image is denoted by $G^{\textrm{s.} \sigma\textrm{-}\mathrm{rs}}$, called the \emph{strongly $\sigma$-regular semisimple locus} of $G$. 
Then $I|_{G^{\textrm{s.}  \sigma\textrm{-}\mathrm{rs}}}$ is smooth and fiberwise conjugate to $T^\sigma$ in $G$. By Lemma~\ref{L:branched locus}, $G^{\textrm{s.}  \sigma\textrm{-}\mathrm{rs}}=G^{\sigma\textrm{-}\mathrm{rs}}$ if $G$ is semisimple and simply-connected.
\end{rmk}
\begin{rmk}
Note that $\sigma$ is of finite order (since it preserves a pinning).
If $\cha k$ does not divide the order of $\sigma$, then $G^{\sigma\textrm{-rs}}\sigma$ is contained in the set of semisimple elements of the non-connected algebraic group $G\rtimes\langle\sigma\rangle$.
\end{rmk}

\begin{lem}
\label{L: isomorphism locus}
\begin{enumerate}
\item The restriction of $c_\sigma$ to $U\times \mathring{T}\sigma\to \mathring{B}\sigma$ is an isomorphism. 
\item For a $\sigma$-orbit $\mO$ of simple roots, let $U^\mO$ be the subgroup of $U$ introduced before Lemma~\ref{L:rankonegroup}. Then the restriction of $c_\sigma$ to $U^{\mO}\times B_{\mO}^{[\mO]}\sigma\to B^{[\mO]}\sigma$ is an isomorphism. 
\end{enumerate}
\end{lem} 
\begin{proof}
We only prove (2) since the proof of (1) is similar (and simpler). Note that $U^{\mO}$ is a normal subgroup in $B$ and $B=U^{\mO}B_{\mO}$ as a semidirect product (of algebraic groups). 

First, we claim that for any $b\sigma\in B_{\mO}^{[\mO]}\sigma$, the group
\[\{u\in U^{\mO}\mid b\sigma(u)b^{-1}=u\}\]
is trivial. Indeed, let $\Phi_i$ be as in the proof of Lemma~\ref{L: Jordan decom}, and for $i\geq 1$, let $U_{i}$ denote the group generated by the root groups $U_\alpha$ with $\alpha \in \Phi_{i} \cup \Phi_{i+1} \cup \cdots$. Then we obtain a filtration 
$U=U_{1}\supset U_{2}\supset\cdots$ by normal subgroups. Let $U^{\mO}_i=U^{\mO}\cap U_i$. Then $U^{\mO}_i/U^{\mO}_{i+1}$ is abelian, isomorphic to its Lie algebra. As argued in Lemma~\ref{L: Jordan decom}, $\{u\in U^{\mO}_i/U^{\mO}_{i+1}\mid b\sigma(u)b^{-1}=u\}$ is trivial because $q_B(b)$ does not lie in $\bfA_{\calO'}$ for any $\sigma$-orbits $\calO'$ that appear in $U^\calO$. The claim follows by induction.

It follows that the map in Part (2) is a monomorphism, and that the map
$\Lie\ U^{\mO}\to \Lie\ U^{\mO}$ given by $ X\mapsto X-\Ad_{b}(\sigma(X))$ 
is an isomorphism. 
On the other hand, for any $b_1\in U^{\mO}$, and $Y\in \Lie B_{\mO}$,
$$Y= \Ad_{b_1}Y \mod \Lie\ U^{\mO}. $$
It follows from Lemma~\ref{E:diff of csigma} that $U^{\mO}\times B_{\mO}^{[\mO]}\sigma\to B^{[\mO]}\sigma$ is \'etale and therefore is an open embedding. Note that the following diagram is commutative 
\[
\xymatrix{
U^{\mO}\times B_{\mO}^{[\mO]}\sigma\ar^{c_\sigma}[rr]\ar_{\pr_2}[rd]& &B^{[\mO]}\sigma\ar[dl]\\
&B_{\mO}^{[\mO]}\sigma,&
}
\] 
where $B^{[\mO]}\sigma\to B^{[\mO]}_{\mO}\sigma$ is induced by the projection $B=U^\mO B_\mO\to B_\mO$.
In addition, for every point $b\sigma\in B_{\mO}^{[\mO]}\sigma$, the $U^{\mO}$-orbit through this point is closed in $B^{[\mO]}\sigma$ since $U^{\mO}$ is unipotent. It follows that the $B_{\mO}^{[\mO]}\sigma$-morphism $c_\sigma: U^{\mO}\times B_{\mO}^{[\mO]}\sigma\to B^{[\mO]}\sigma$ is fiberwise open and closed. Therefore $U^{\mO}\times B_{\mO}^{[\mO]}\sigma\to B^{[\mO]}\sigma$ is open and surjective. Part (2) follows.
\end{proof}

We also need the following companion result. If $\mO^+$ is a $\sigma$-orbit of type $\sfB\sfC^+$ such that $\al_{\mO^+}$ is a simple root of $\Phi(G_\sigma,A)$, let $\mO^-$ denote the corresponding $\sigma$-orbit of simple roots of type $\sfB\sfC^-$ and we fix an order of roots in $\mO^-$.

\begin{lem}
\label{L: isomorphism locus2}Keep the notations as above. Then the restriction
$$c_\sigma: \Big(U^{\mO^-}\times \prod_{\al\in\mO^-} U_{\al}\Big)\times B_{\mO^+}^{[\mO^+]}\sigma\to B^{[\mO^+]}\sigma$$ is an isomorphism, where the product $\prod_{\al\in\mO^-} U_{\al}$ is taken with respect to a chosen order.
\end{lem}
\begin{proof}
The same proof of  Lemma~\ref{L: isomorphism locus} (2) for the $\sigma$-orbit $\calO^-$ shows that the restriction of $c_\sigma$ to $U^{\mO^-}\times B^{[\mO^+]}_{\mO^-}\sigma\to B^{[\mO^+]}\sigma$ is also an isomorphism. Therefore, it reduces to prove that
\[c_\sigma:  \prod_{\al\in\mO^-} U_{\al}\times B_{\mO^+}^{[\mO^+]}\sigma\to B^{[\mO^+]}_{\mO^-}\sigma\]
is an isomorphism. This follows from a simple explicit computation. Namely, we set $d=|\mO^+|$, fix a root $\alpha \in \calO^-$, and write $\al_i=\sigma^{i}(\al)$ for $1\leq i\leq 2d$ and $\beta_i=\al_i+\sigma^{d}(\al_i)$. Then for $b=t\prod x_{\beta_i}(b_i)$ with $e^{\al_\mO}(t)\neq 1$, and $u=\prod x_{\al_i}(a_i)$,
\[ub\sigma(u)^{-1}= t\prod x_{\al_i}(e^{-\al_i}(t)a_i-a_{i-1})\prod x_{\beta_i}(b'_i),\]
where $b'_i-b_i$ is a polynomial in $(a_1,\ldots,a_{2d})$ of degree two (with coefficients involving $e^{\alpha_i}(t)$'s). Since $e^{\al_\mO}(t)-1$ is invertible, for every $(c_1,\ldots, c_{2d})$, there is a unique $(a_1,\ldots,a_{2d})$ such that $c_i=e^{\al_i}(t)a_i-a_{i-1}$, the map in the lemma is an isomorphism.
\end{proof}

\begin{prop}
\label{L: surj of reg}
\begin{enumerate}
\item The map $q_B: B^{\sigma\textrm{-}\mathrm{reg}}\to \bfA$ is surjective. 
\item Each fiber of $q_B: B^{\sigma\textrm{-}\mathrm{reg}}\to \bfA$ is a single $B$-orbit for the $\sigma$-conjugation action. 
\item For every $b\in B^{\sigma\textrm{-}\mathrm{reg}}$, $(\id-\Ad_b\sigma)^{-1}(\fraku)=\frakt^\sigma+\fraku$. 
\item Assume that $G$ is semisimple and simply-connected. Fix $b\in B^{\sigma\textrm{-}\mathrm{reg}}$. Let $g\in G$ such that $gb\sigma(g)^{-1}\in B$ and $q_B(gb\sigma(g)^{-1})=q_B(b)$. Then $g\in B$. In addition, $I_b\subset B$.
\end{enumerate}
\end{prop}
\begin{proof}
(1) For every $t\in T$, consider the automorphism of $G_t$ given in  Lemma~\ref{L: pinned auto of Gt}, denoted by $t\sigma$. Applying Proposition~\ref{L: unipotent regular} to $(G_t, t\sigma)$ gives an element $u \in U_t$ which is $t\sigma$-regular in $G_t$. Let $b=ut\in U_tT$. Then $I_b\cap G_t=\{g\in G_t\mid gut\sigma(g)^{-1}t^{-1}=g\}$ is the $t\sigma$-twisted centralizer of $u$ in $G_t$. Therefore, $b = ut$ is also $\sigma$-regular in $G_t$ by Lemma~\ref{L: Jordan decom}.

(2) If $b\in  B^{\sigma\textrm{-}\mathrm{reg}}$, then the dimension of the $B$-orbit through $b$ under the $\sigma$-conjugation action of $B$ on itself is equal to $\dim q_B^{-1}(q_B(b))$. This shows that any two $\sigma$-regular $B$-orbits in a fiber of $q_B$ must meet since fibers of $q_B$ are irreducible. Therefore, there is exact one $B$-orbit in each fiber of $q_B: B^{\sigma\textrm{-}\mathrm{reg}}\to \bfA$.

(3) We may assume that $b=ut$, with $u\in U_t$ being $t\sigma$-regular in $G_t$. Then the claims follows from Lemma~\ref{L: Jordan decom}, and Proposition~\ref{L: unipotent regular}.

(4) We may assume that $b=ut$, with $u\in U_t$ being $t\sigma$-regular in $G_t$. Write an element $g \in I_b(k)$ as $g = u_1n u_2$  with $u_1,u_2\in U$ and $n\in N$. Then $u_1n u_2 b = b\sigma(u_1) \sigma(n) \sigma (u_2)$. It follows that $\sigma (n) = n$, so
$$
n u_2 b \sigma(u_2)^{-1} n^{-1} = u_1^{-1}b\sigma(u_1).
$$
Taking projection to $\bfA$, we see that $w(q_B(t))=q_B(t)$, where $w=n \mod T$. Since $G$ is simply-connected, by
Lemma~\ref{L:branched locus} and Lemma~\ref{L: root system Phit}, $w$ must be in the Weyl group of $G_t$. Then by applying the same argument as in the proof of Proposition~\ref{L: unipotent regular} to $G_t$, we deduce that $w = 1$. In particular, $I_b^\red \subset B$. On the other hand,
$$
\ker(\id - \Ad_b\sigma: \frakg \to \frakg) = \ker(\id - \Ad_b\sigma: \frakg_t \to \frakg_t) = \ker(\id - \Ad_b\sigma:\frakb \cap \frakg_t\to \frakb\cap\frakg_t).
$$
where the first equality follows from Lemma~\ref{L: Jordan decom}, and the second follows from Lemma~\ref{L: aux for reg}. Putting these together, we see that $I_b\subset B$.

More generally let $g\in G$ such that $gb\sigma(g)^{-1}\in B$ and $q_B(gb\sigma(g)^{-1})=q_B(b)$, there is some $b'\in B$ such that $gb\sigma(g)^{-1}=b'b\sigma(b')^{-1}$, by Part (2). Therefore $g^{-1}b'\in I_b\subset B$ and $g\in B$.
\end{proof}

\begin{cor}
\label{C: unr tangent needed}
Let $K=\ker q_B=(\sigma-1)T\cdot U$, and let $\frakk$ be its Lie algebra. Then for every $b\in B^{\sigma\textrm{-}\mathrm{reg}}$, $(\id-\Ad_b\sigma)^{-1}(\frakk)=\frakb$.
\end{cor}
\begin{proof} Let $X\in \frakg$ such that $X-\Ad_b\sigma(X)=(1-\sigma)H+\fraku$, with $H\in \frakt$. Then $(\id-\Ad_b\sigma)(X-H)\in\fraku$. Therefore $X-H\in\frakt^\sigma+\fraku$ and $X\in\frakb$.
\end{proof}

\begin{cor}
\label{L: codim of reg}
The codimension of $B-B^{\sigma\textrm{-reg}}$ is at least two. 
\end{cor}
\begin{proof}
This follows from Lemma~\ref{L: isomorphism locus} (1) and Proposition~\ref{L: surj of reg}.
\end{proof}

\begin{cor}
\label{L: regular springer fiberII}
An element $b\in B$ is $\sigma$-regular if and only if the set $\mB_b=\{g\in G/B\mid g^{-1}b\sigma(b)\in B\}$ is finite.
\end{cor}
\begin{proof}
This follows from Lemma~\ref{L: isomorphism locus} by applying Lemma~\ref{L: regular springer fiber} to $u\in U_t$.
\end{proof}

\subsection{Twisted Grothendieck--Springer resolution}
\label{S:twisted Springer resolution}
In this subsection, we assume that $G$ is semisimple and simply-connected. 

Note that $k[T]^{c_\sigma(N_0)}=k[A]^{W_0}$. Then
the twisted Chevalley isomorphism (Proposition~\ref{AE:classical HC morphism}) implies that $G/\!\!/c_\sigma(G)\cong \mathbf A/\!\!/W_0:=\on{Spec} k[\mathbf A]^{W_0}$. So we write the (twisted) Chevalley map as
$$\chi: G\to G/\!\!/c_\sigma(G)=\bfA/\!\!/W_0,$$
which we recall is faithfully flat  by Corollary~\ref{L:cl flat}. 

As in the untwisted case, there is the following commutative diagram
\[\begin{CD}
\widetilde G@>q>> \mathbf A\\
@VVV@VVV\\
G @>\chi>> \mathbf A/\!\!/W_0,
\end{CD}\]
where the left vertical map
\[
\widetilde G: = \{(gB, x) \in G/B \times G\;|\; x \in gB\sigma(g)^{-1}\} \longto G \qquad (gB,x) \mapsto x,
\]
is what we call the \emph{($\sigma$-twisted) Grothendieck--Springer resolution}. 
The map $q$ is induced by
\[\widetilde G\cong G\times^B (B\sigma)\xrightarrow{q_B} G\times^B \mathbf A\to \mathbf A.\]
Together, these two maps induce a proper map
$\pi:\widetilde G\to G\times_{\mathbf A/\!\!/W_0}\mathbf A$.

Let $\widetilde G^{\sigma\textrm{-reg}}$ be the preimage of $G^{\sigma\textrm{-reg}}$. We have the following proposition.
\begin{prop}
\label{P: regular cartesian}
The induced map $\widetilde G^{\sigma\textrm{-}\mathrm{reg}}\to G^{\sigma\textrm{-}\mathrm{reg}}\times_{\mathbf A/\!\!/W_0}\mathbf A$ is an isomorphism.
\end{prop}
\begin{proof}
We start with the following special case.
\begin{lem}
\label{L: RS cartesian}
Let $\widetilde G^{\sigma\textrm{-}\mathrm{rs}}$ be the preimage of $G^{\sigma\textrm{-}\mathrm{rs}}$.
Then the induced map \begin{equation}
\label{E:pi restricted to Gsigmars}
\pi|_{\widetilde G^{\sigma\textrm{-}\mathrm{rs}}}: \widetilde G^{\sigma\textrm{-}\mathrm{rs}}\to G^{\sigma\textrm{-}\mathrm{rs}}\times_{\mathring{\mathbf A}/\!\!/W_0}\mathring{\mathbf A}\end{equation} is an isomorphism.
\end{lem}
\begin{proof}

By \eqref{E:conjugation} and Lemma~\ref{E: sigmaregularinT}, we have $G^{\sigma\textrm{-rs}} \cong G \times^{N_0} \mathring T$. So we can write
$$\widetilde G^{\sigma\textrm{-rs}} = \big\{(gB, g't\sigma(g')^{-1}) \in G/B \times ( G \times^{N_0} \mathring T)\,\big|\, g't\sigma(g')^{-1} \in gB\sigma(g)^{-1}
\big\}.
$$
The last condition is equivalent to $g^{-1} g' t \sigma (g^{-1}g')^{-1} \in B$. By
Proposition~\ref{L: surj of reg} (4), $g^{-1} g' \in B \times^T N_0$. From this, we deduce that $\widetilde G^{\sigma\textrm{-rs}} \cong G \times^T \mathring T$ and therefore \eqref{E:pi restricted to Gsigmars} is an isomorphism.
\end{proof}

We return to the proof of Proposition~\ref{P: regular cartesian}.
Since $\mathbf A  \to \mathbf A /\!\!/ W_0$ is finite and flat by Lemma~\ref{L:branched locus}, the fiber product $G^{\sigma\textrm{-rs}} \times_{\mathring{\mathbf A}/\!\!/W_0}\mathring{\mathbf A}$ is open and dense in $G \times_{\mathbf A/\!\!/W_0}\mathbf A$. But $\pi$ is  proper as $G/B$ is, so $\pi$ is surjective. In particular, $G\times_{\mathbf A/\!\!/W_0}\mathbf A$ is irreducible. Moreover, since $\mathbf A$ and $G$ are smooth and $G\to \mathbf A/\!\!/W_0$ is flat,  $G\times_{\mathbf A/\!\!/W_0}\mathbf A$ is a complete intersection (in particular Cohen--Macaulay), and it follows from the above lemma that $G\to \mathbf A/\!\!/W_0$ is generically smooth. Therefore, $G\times_{\mathbf A/\!\!/W_0}\mathbf A$ is also reduced. So $G\times_{\mathbf A/\!\!/W_0}\mathbf A$ is a closed subvariety of $G\times \mathbf A$.

Since $\widetilde G\to G$ is surjective, every element in $G$ is $\sigma$-conjugate to an element in $B$. In particular, $\widetilde G^{\sigma\textrm{-reg}}= G\times^BB^{\sigma\textrm{-reg}}$. By Proposition~\ref{L: surj of reg} (4), the map $G\times^BB^{\sigma\textrm{-reg}}\to G\times \mathbf A$ is injective on points.
By Lemma~\ref{E:diff of csigma} and Corollary~\ref{C: unr tangent needed}, the tangent map for $G\times^BB^{\sigma\textrm{-reg}}\to G\times \bfA$ is injective at every point, and therefore the morphism is unramified. It follows that $\widetilde G^{\sigma\textrm{-reg}}\to G^{\sigma\textrm{-reg}}\times \mathbf A$ is a closed embedding, with image (as topological space) $G^{\sigma\textrm{-reg}}\times_{\mathbf A/\!\!/W_0}\mathbf A$. But since $G^{\sigma\textrm{-reg}}\times_{\mathbf A/\!\!/W_0}\mathbf A$ is also reduced, the map in the proposition is indeed an isomorphism.
\end{proof}

Here are some standard corollaries.

\begin{cor}
The morphism $\chi: G^{\sigma\textrm{-}\mathrm{reg}}\to \bfA/\!\!/W_0$ is smooth. 
\end{cor}
\begin{proof}
This follows from the smoothness of $\widetilde G^{\sigma\textrm{-reg}}\to \bfA$ and the flatness of $\bfA\to\bfA/\!\!/W_0$.
\end{proof}

\begin{cor}
\label{P: GS resolution}
The map $\pi$ induces an isomorphism
$k[G\times_{\mathbf A/\!\!/W_0}\mathbf A]\cong \Gamma(\widetilde G,\mO)$. 
\end{cor}
\begin{proof}
It is enough to show that $G\times_{\mathbf A/\!\!/W_0}\mathbf A$ is an integral normal scheme. Then it follows that $\pi$ induces an isomorphism between rings of global regular functions,  since it is proper birational.

We already seen that $G\times_{\mathbf A/\!\!/W_0}\mathbf A$ is integral, and a complete intersection. By Proposition~\ref{P: regular cartesian}, $G^{\sigma\textrm{-reg}}\times_{\mathbf A/\!\!/W_0}\mathbf A$ is smooth. By Corollary~\ref{L: codim of reg}, the complement of $G^{\sigma\textrm{-reg}}\times_{\mathbf A/\!\!/W_0}\mathbf A$ in $G\times_{\mathbf A/\!\!/W_0}\mathbf A$ has codimension at least two. 
It follows from this lemma that $G\times_{\mathbf A/\!\!/W_0}\mathbf A$ is normal. 
\end{proof}

\begin{cor}
\label{C:fiber regular}
\begin{enumerate}
\item For every $a\in\bfA/\!\!/W_0$, $\chi^{-1}(a)\cap G^{\sigma\textrm{-}\mathrm{reg}}$, is a single $G$-orbit, and the codimension of the complement of this $G$-orbit in $\chi^{-1}(a)$ at least two.
\item Each fiber of $\chi$ is a complete intersection and normal variety. 
\end{enumerate}
\end{cor}
\begin{proof}
(1) Pick $\tilde a \in \bfA(k)$ that lifts $a \in \bfA/\!\!/W_0$. By Proposition~\ref{P: regular cartesian}, the fiber $\chi^{-1}(a) \cap G^{\sigma\textrm{-reg}}$ is isomorphic to the fiber of $q: \widetilde G^{\sigma\textrm{-reg}} \cong G \times^B B^{\sigma\textrm{-reg}} \to \bfA$ at $\tilde a$,  which is clearly a single $G$-orbit by Proposition~\ref{L: surj of reg} (2). On the other hand, by Corollary~\ref{L: regular springer fiberII}, the fibers of $\widetilde{G}\to G$ over $G- G^{\sigma\textrm{-}\mathrm{reg}}$ have positive dimension. Since $q_B^{-1}(\tilde a)- (q^{-1}_B(\tilde a)\cap B^{\sigma\textrm{-}\mathrm{reg}})$
is a proper closed subset of $q^{-1}_B(\tilde a)$, $q^{-1}(a)-(q^{-1}(a)\cap G^{\sigma\textrm{-}\mathrm{reg}})$ has codimension at least two in $q^{-1}(a)$.

(2) Since $\chi: G \to \bfA /\!\!/ W_0$ is flat, each of its fiber is a complete intersection and hence Cohen--Macaulay. By (1), each fiber of $\chi$ contains a (smooth) $G$-orbit whose complement has codimension at least two. So the fiber is also normal.
\end{proof}

\section{The determinant of the pairing $\bfJ(V)\otimes\bfJ(V^*)\to \bfJ$}
\label{S:determinant}

\begin{assumption}
\label{H:hypo in section 6}
In this section, assume that $\on{char} k >2$. Let $G$ be a simply-connected semisimple group over $k$. Let $V$ be a finite dimensional representation of $G$ and $V^*$  the dual representation. Assume that both $V$ and $V^*$ admit  good filtrations.
\end{assumption}

We keep conventions and notations as in \S~\ref{S:automorphism}.
Then $\widetilde{V^*}$ is the dual of $\widetilde V$ as a vector bundle on $[G\sigma/G]$. The perfect pairing $\widetilde{V}\otimes\widetilde{V^*}\to\mO_{[G\sigma/G]}$ induces a $\bfJ$-bilinear pairing $\bfJ(V)\otimes\bfJ(V^*)\to \bfJ$ of global sections, which however is not perfect in general. Our main result (Theorem~\ref{T:determinant of intersection}) calculates the determinant of this pairing of two finite free $\bfJ$-modules. A main intermediate step is to study the failure of the surjectivity of the twisted Chevalley restriction homomorphism \eqref{AE:generalized HC morphism}.

\subsection{Main results}
\begin{dfn}
\label{D:zeta pm}
For each $\sigma$-orbit $\calO \subset \Phi(G, T)$, we choose $\alpha \in \calO$ and view it as a character of $T^\sigma$ by restriction, which is clearly independent of the choice of $\alpha$.
Define the number
\[
\zeta_{\mO} = \zeta_{\mO}(V): = \sum_{n \geq 1}
\dim V|_{ T^\sigma}( n \alpha ).
\]
\end{dfn}
The main theorem of this section is the following.

\begin{theorem}
\label{T:determinant of intersection}
Keep Assumption~\ref{H:hypo in section 6}.
The determinant of the natural $\bfJ$-bilinear pairing\footnote{Taking the determinant makes sense because Theorem~\ref{T:Kostant} shows that $\bfJ(V)$ and $\bfJ(V^*)$ are both free $\bfJ$-modules of the same rank.}
\begin{equation}
\label{E:intersection matrix}
\langle \cdot,\cdot \rangle_V: \bfJ(V)\otimes \bfJ(V^*)\to \bfJ
\end{equation}
is of the form
\begin{equation}
\label{E:determinant}
\big(\textrm{some unit in }k  \big)\cdot 
\prod_{\mO \textrm{ type }\sfA \textrm{ or }\sfB\sfC^-} (e^{\alpha_\mO}-1)^{\zeta_{\mO}}\cdot \prod_{\mO \textrm{ type }\sfB\sfC^+} (e^{\alpha_\mO}+1)^{\zeta_{\mO}}.
\end{equation}
\end{theorem}
\begin{rmk}
\label{R: pairing}
\begin{enumerate}
\item Note that  $\zeta_{\mO}$ depends only on the $W_0$-orbits of $\mO$, and therefore the product in \eqref{E:determinant} does belong to $\bfJ$.
\item When $\sigma$ is trivial, the rank of $\bfJ(V)$ as a $\bfJ$-module is $r_V = \dim V(0)$ by Theorem~\ref{T:Kostant}. In this case, then $\zeta_\calO = \zeta_\alpha$ can be alternatively computed via the multi-filtration on the weight spaces introduced in \S~\ref{S:the filtration} as
$$
\zeta_\alpha = \sum_{\nu \in \XX^\bullet(T)^+} \langle \nu, \check \alpha\rangle \cdot  \dim \gr_\nu V(0).
$$
\item Similarly we have the pairing $\bfJ_+(V)\otimes\bfJ_+(V^*)\to \bfJ_+$. Its restriction to $\Spec \bfJ$ is the pairing in the theorem, and restriction to the point $x_1\in\Spec\bfJ_0$ as defined in Remark~\ref{r:special point} is the natural pairing between $(V|_{T^\sigma})(0)$ and $(V|_{T^\sigma})^*(0)$.
\item When $\on{char} k=0$, we have explained in \S~\ref{SS: basis}  a construction of a basis of $\bfJ(V)$ (resp. $\bfJ(V^*)$) from a certain basis of $V$ (resp. $V^*$) (e.g. the MV basis used in \S~\ref{SS: fil via geomSat}). Then the $\langle \cdot,\cdot \rangle_V$ is represented by a square matrix. It seems to be an interesting (although probably difficult) question to calculate the entries of this matrix explicitly. See \S~\ref{SS:example of matrix} for some calculations and further discussions. We also refer to \cite{XZ} for the arithmetic and  geometric meaning of this square matrix. 
\end{enumerate}
\end{rmk}

In what follows, we shall relate the pairing \eqref{E:intersection matrix} to the twisted Chevalley restriction map \eqref{AE:generalized HC morphism}. As we will show in Lemma~\ref{L:modification =>determinant} that Theorem~\ref{T:determinant of intersection} follows from Proposition~\ref{P:modification of vector bundles over S} below.
More precisely, we will not study the Chevalley restriction homomorphism \eqref{AE:generalized HC morphism} itself, but rather the induced $\bfJ_T$-module homomorphism
\begin{equation}
\label{E:Res otimes 1}
\Res_V^\sigma \otimes 1: \bfJ_G(V) \otimes_{\bfJ_G} \bfJ_T \to \bfJ_T(V),
\end{equation}
where explicitly, we may write the target as
\begin{equation}
\label{E:explicit JT(V)}
\bfJ_T(V) = (k[T] \otimes V)^{c_\sigma(T)} \cong \bigoplus_{\xi \in (\sigma-1)\XX^\bullet(T)} k[A] e^{\nu_\xi} \otimes V(\xi)
\end{equation}
with $\nu_\xi \in \XX^\bullet(T)$ some weight such that $\sigma(\nu_\xi) - \nu_\xi = \xi$. In particular, if $\sigma = \id$, $\bfJ_T(V) \cong k[T] \otimes V(0)$.

We view $\Res_V^\sigma \otimes 1$ as a morphism of coherent sheaves over $A$.
By Lemma~\ref{L:generic Chevalley} (which relies on Lemma~\ref{E: sigmaregularinT}), it is an isomorphism over
$\mathring{A}$ (which is defined in Definition~\ref{D: divisor for roots}).

Now, let $\eta$ be a generic point of the divisor $\bigcup_{\mO} A_{\mO}$ (which is reduced since we assumed that $\on{char} k >2$; see Remark~\ref{R: irr and red of AO}). Then the complete local ring of $A$ at $\eta$ is isomorphic to $k(\eta)[[\varpi]]$, where $\varpi=e^{\al_\mO}-1$ or $e^{\al_\mO}+1$ for some $\sigma$-orbit $\mO\subset \Phi(G,T)$. 
Note that $\bfJ_G(V)$ is always a torsion free $\bfJ_G$-module (even if $V$ does not admit a good filtration). Therefore $\bfJ_G(V)\otimes_{\bfJ_G}k(\eta)[[\varpi]]$ is always free, and \eqref{E:Res otimes 1}, base changed to $k(\eta)[[\varpi]]$, is a map $\Res_\eta$ of finite free $k(\eta)[[\varpi]]$-modules which becomes an isomorphism when further base changed to $k(\eta)((\varpi))$. (Such map is called a \emph{modification} of vector bundles on $\Spec k(\eta)[[\varpi]]$ (in the sense as in \cite[\S~3.1.3]{XZ}). 
The top exterior power of this map $\Res_\eta$ is an element in $k(\eta)[[\varpi]]-\{0\}$, well-defined up to multiplying an element in $k[[\varpi]]^\times$, and therefore gives a well-defined element in 
$$(k(\eta)[[\varpi]]-\{0\})/k(\eta)[[\varpi]]^\times\cong \bZ_{\geq 0}.$$ 
We call this number \emph{the length of the modification}.

Here is the main result regarding the map \eqref{E:Res otimes 1}.

\begin{proposition}
\label{P:modification of vector bundles over S}
For every $G$-module $V$ with good filtration, the length of the modification \eqref{E:Res otimes 1} at every generic point of $A_{\mO}$ is exactly $\zeta_{\mO}$.
\end{proposition}

This proposition will be proved in \S~\ref{SS: proof of prop mod}. We note the following first.

\begin{lem}
\label{L:modification =>determinant}
Proposition~\ref{P:modification of vector bundles over S} implies Theorem~\ref{T:determinant of intersection}.
\end{lem}
\begin{proof}
Consider the following commutative diagram
\[\xymatrix@C=5pt{
\bfJ(V) \ar[d] & \times & \bfJ(V^*) \ar[d] \ar[rrrrr]^{\langle \cdot, \cdot\rangle_V} &&&&& \bfJ \ar[d] 
\\
\bfJ(V) \otimes_{\bfJ} \bfJ_T \ar[d]^{\Res^\sigma_V \otimes 1} & \times & \ar[d]^{\Res^\sigma_{V^*} \otimes 1} \bfJ(V^*) \otimes_{\bfJ} \bfJ_T \ar[rrrrr]^-{\langle \cdot, \cdot\rangle_V}&& &&& \bfJ_T \ar@{=}[d]
\\
\boldsymbol J_T(V)  & \times & \boldsymbol J_T(V^*) \ar[rrrrr]^-{\langle \cdot, \cdot \rangle_{V}} &&&&& \bfJ_T.}
\]
By Theorem~\ref{T:Kostant}, $\bfJ(V)$ and $\bfJ(V^*)$ are free $\bfJ$-modules. So second row is simply a base change of the first row, and hence the matrices for the top two pairings are the same (when choosing compatible bases).
Now the bottom row is a \emph{perfect} $\bfJ_T$-bilinear pairing (as can be easily seen from the explicit expression of $\bfJ_T(V)$ in \eqref{E:explicit JT(V)}). So the determinant of the middle row is the product of the determinant of the map $\Res_V^\sigma \otimes 1$ and the determinant of the map $\Res_{V^*}^\sigma \otimes 1$ (up to a unit).
Thus, Proposition~\ref{P:modification of vector bundles over S} would imply that the determinant of \eqref{E:intersection matrix}, as a divisor on $A$, is given by
$$
\sum_{\calO \subset \Phi(G, T)}
(\zeta_\mO(V) + \zeta_\mO(V^*)) \cdot A_\calO.
$$
By Lemma~\ref{L:zeta(V) = zeta(V*)} below, this gives the same expression as the formula \eqref{E:determinant} up to a unit in $k[A]$. But note that the product in \eqref{E:intersection matrix} is $W_0$-invariant, so is the determinant of \eqref{E:intersection matrix}. It follows that the ambiguous unit in $k[A]$ belongs to $(k[A]^{W_0})^\times$. As mentioned in the proof of Lemma~\ref{L:branched locus} (see also \cite[Corollary 2]{springer}), $k[A]^{W_0}$ is a polynomial algebra and hence its units are just $k^\times$. Therefore, the determinant of \eqref{E:intersection matrix} is given by \eqref{E:determinant} up to a unit in $k$.
\end{proof}

\begin{lem}
\label{L:zeta(V) = zeta(V*)}
For a representation $V$ of $G$ and a $\sigma$-orbit $\mO \subset \Phi(G, T)$, we have $\zeta_{\mO}(V) = \zeta_{\mO}(V^*)$.
\end{lem}
\begin{proof}
This follows from the following sequence of equalities
\[
\dim V^*|_{T^\sigma}(n \alpha) = \dim V|_{T^\sigma}(-n \alpha) =  \dim V|_{T^\sigma}(n \alpha),
	\]
where the first equality follows from the duality and the last equality follows from the fact that $\alpha$ and $-\alpha$ lie in the same $W$-orbit.
\end{proof}

\subsection{Proof of Proposition~\ref{P:modification of vector bundles over S}}
\label{SS: proof of prop mod} 
We will first reduce Proposition~\ref{P:modification of vector bundles over S} to the cases of $\SL_2$ and $\SL_3$.
\begin{lem}
\label{L:JG and JB}
Let $V$ be a representation of $G$.
The map $\bfJ_G(V)\otimes_{\bfJ_G}\bfJ_T\to \bfJ_B(V)$ is an isomorphism.
\end{lem}
\begin{proof}
Recall that by Corollary~\ref{P: GS resolution}, the pushforward of the structure sheaf along
$\pi: [B\sigma/B]\to [G\sigma/G]\times_{\mathbf A/\!\!/W_0}\mathbf A$ is the structure sheaf.
Then it follows from the projection formula that
\[\bfJ_G(V)\otimes_{\bfJ_G}\bfJ_T=\Gamma([G\sigma/G]\otimes_{\mathbf A/\!\!/W_0}\mathbf A,\widetilde V\boxtimes\mO)=\Gamma([B\sigma/B],\pi^*\widetilde V)=\bfJ_B(V)\]
is an isomorphism. 
\end{proof}

Using the $W_0$-action on $\Phi(G_\sigma, A)$, it suffices to prove  Proposition~\ref{P:modification of vector bundles over S} for a $\sigma$-orbit $\calO$ that is either a $\sigma$-orbit of \emph{simple} roots (of type $\sfA$ or $\sfB\sfC^-$), or a $\sigma$-orbit of roots of type $\sfB \sfC^+$ that are sum of two roots from a $\sigma$-orbit $\calO^-$ of simple roots.
In either case, let $G_\calO$
be the subgroup of $G$ generated by $T$ and $U_\al$ with $ \al\in\mO \cup -\mO$; and $B_\calO$ the Borel subgroup generated by $T$ and $U_\al$ with $ \al\in \mO$.
\begin{lem}
\label{L:reduction from G to Galpha}
The map $\bfJ_B(V)\to \bfJ_{B_{\mO}}(V)$ is an isomorphism over $\bfA^{[\mO]}$.
\end{lem}
\begin{proof}
It follows from Lemma~\ref{L: isomorphism locus} and Lemma~\ref{L: isomorphism locus2} that
$[B^{[\calO]}\sigma / B] \cong [B^{[\calO]}_{\calO}\sigma / B_{\calO}]$. So the map in the lemma induces an isomorphism when restricted to $\bfA^{[\mO]}$.
\end{proof}
\begin{lem}
\label{L:reduce to rank one groups}
To prove Proposition~\ref{P:modification of vector bundles over S}, it is enough to prove it for two cases:
\begin{itemize}
\item $G=\SL_2$ with $\sigma=\id$;
\item $G=\SL_3$ with $\sigma$ given as in Example~\ref{Ex:folding A2n} and $\calO$ being a $\sigma$-orbit consisting of two roots.
\end{itemize}
\end{lem}
\begin{proof}
By Lemma~\ref{L:JG and JB} and Lemma~\ref{L:reduction from G to Galpha}, we reduce Proposition~\ref{P:modification of vector bundles over S} to prove that, for the $\sigma$-orbit $\calO$ considered above and over every generic point $\eta$ of $\bfA_{\mO}$, the map $\bfJ_{B_{\mO}}(V)\to \bfJ_T(V)$ is a modification of length $\zeta_{\mO}$.  (Note that $V|_{G_\calO}$ also admits a good filtration by Theorem~\ref{T:prop of good fil} (1).)
We consider the central isogeny
\begin{equation*}
\label{E:G alpha def to G alpha}
G'_\calO:= G_{\mO,\on{der}}\times Z_{G_{\mO}}\to G_{\mO}.
\end{equation*}
with kernel $F = G_{\mO,\on{der}} \cap Z_{G_{\mO}}$. 
Let $T'$ denote the preimage of $T$ and $A' = T'_{\sigma}$. Note that the right exact sequence $F_\sigma\to T'_\sigma\to T_\sigma\to 1$ is also exact on the left, so the induced map $A' \to A$ is an $F_\sigma$-torsor. Note that $F_\sigma$ is indeed \'etale (given the classification of $G_{\calO,\der}$ by Lemma~\ref{L:rankonegroup}) since we assume $\cha k\neq 2$.
By Lemma~\ref{L: fpqc descent for JV} (applied to $B'_\calO= B_{\mO,\der}\times Z_{G_{\mO}}$ and $T_{\mO,\der}\times Z_{G_{\mO}}$) and the fact that modifications commute with \'etale base change, the length of modification of $\bfJ_{B_\mO}(V)\to \bfJ_T(V)$ at every generic point $\eta$ of $\bfA_\mO$ is the same as the length of modification of $\bfJ_{B'_\mO}(V)\to \bfJ_{T'}(V)$ (or equivalently $\bfJ_{G'_{\mO}}(V) \otimes_{\bfJ_{G'_\mO}}\bfJ_T \to \bfJ_T(V)$ by Lemma~\ref{L:JG and JB}) at any preimage of $\eta$ under the map $\pi:\bfA'\to\bfA$.

We decompose $V=\oplus_{\psi} V_\psi\otimes k_\psi$ according to the central character for the action of $Z_{G_\mO}$ on $V$ so that each $V_\psi$ is a $G_{\mO,\der}$-module. Then by \eqref{E:JV tensor}
\[\bfJ_{G'_{\mO}}(V)=\bigoplus_{\psi} \bfJ_{G_{\mO,\der}}(V_\psi)\otimes \bfJ_{Z_{G_\mO}}(k_\psi)=\bigoplus_{\psi|_{Z^\sigma_{G_\mO}}=\mathbf{1}} \bfJ_{G_{\mO,\der}}(V_\psi)\otimes \bfJ_{Z_{G_\mO}}(k_\psi).\]
We write $A_{\der}= (T_{\mO,\der})_\sigma$, where $T_{\mO,\der}=T\cap G_{\mO,\der}$. Then $A'= A_{\der}\times (Z_{G_{\mO}})_\sigma$, and one checks that $\pi:\bfA_{\der,\mO}\times (Z_{G_{\mO}})_\sigma\to \bfA_\mO$ is surjective. Clearly, the modification
\[\bfJ_{G'_{\mO}}(V)=\bigoplus_{\psi} \bfJ_{G_{\mO,\der}}(V_\psi)\otimes \bfJ_{Z_{G_\mO}}(k_\psi)\to \bfJ_{T'}(V)=\bigoplus_{\psi}\bfJ_{T_{\mO,\der}}(V_\psi)\otimes \bfJ_{Z_{G_\mO}}(k_\psi)\]
only happens on the first factor. Since the second factor is of rank one over $\bfJ_{Z_{G_\mO}}$, and since
\[\dim (V|_{T^\sigma})(n\al)=\sum_{\psi|_{Z^\sigma_{G_\mO}}=\mathbf{1}} (V_\psi|_{T_{\mO,\der}^\sigma})(n\al),\]
we see that to prove Proposition~\ref{P:modification of vector bundles over S}, it is enough to assume that $G= G_{\calO, \der}$. 

By Lemma~\ref{L:rankonegroup}, there are three cases for $G_{\mO,\der}$. By Lemma~\ref{L: isomorphism locus2}, and Lemma~\ref{L:reduce from G to G0}, we reduce the proof of Proposition~\ref{P:modification of vector bundles over S} to the case $G=\SL_2$ with $\sigma=\id$, and $G=\SL_3$ with $\sigma$ given as in Example~\ref{Ex:folding A2n} and $\calO$ being a $\sigma$-orbit consisting of two roots. This completes the proof of the lemma.
\end{proof}

Finally, we treat the above mentioned two cases.

\begin{lem}
\label{L:modification SL2}
Proposition~\ref{P:modification of vector bundles over S} holds for $G=\SL_2$.
\end{lem}
\begin{proof} 
It is enough to assume that $V$ is the Schur module $\ttS_n$ (of dimension $n+1$), in which case we make explicit computations. 
We set
$$
B = \big\{ \big( \begin{smallmatrix}
a & b \\ 0 & a^{-1}
\end{smallmatrix} \big)\; \big|\; a \in k^\times,\, b \in k\big\},\quad  T = \big\{ \big( \begin{smallmatrix}
a & 0 \\ 0 & a^{-1}
\end{smallmatrix} \big)\; \big|\; a \in k^\times\big\}, \quad \textrm{and} \quad U =  \big\{ \big( \begin{smallmatrix}
1 & b \\ 0 & 1
\end{smallmatrix} \big)\; \big|\;  b \in k\big\}.
$$
By Lemma~\ref{L:JG and JB},  we need to show that
\begin{enumerate}
\item $J_T(\ttS_n)$ is zero when $n$ is odd, and
\item 
the length of the modification of $J_B(\ttS_n) \to J_T(\ttS_n)$ at the generic point of $T$ is exactly $n/2$ when $n$ is even.
\end{enumerate}
(1) is obvious by looking at the action of $\big( \begin{smallmatrix}
-1 & 0 \\ 0 & -1
\end{smallmatrix} \big) \in \SL_2$. For (2) (and hence $n$ is even), we explicitly identify $\ttS_n \cong k[z]^{\deg \leq n}$ and $\calO_B \cong k[x^{\pm 1}, y]$, where $B$ acts on $\ttS_n$ by $\big( \begin{smallmatrix}
a & b \\ 0 & a^{-1}
\end{smallmatrix} \big) (h)(z) = a^n h\big(a^{-2}z -a^{-1} b \big)$ and the conjugation action of $B$ on $\calO_B$ is given by 
\begin{equation*}
\label{E:conjugation action of B}
\big( \begin{smallmatrix}
a & b \\ 0 & a^{-1}
\end{smallmatrix} \big) (h)(x,y) = h(x, a^{-2}y+ a^{-1}b(x-x^{-1})).
\end{equation*}
From this, we see that $z(x-x^{-1}) + y$ is invariant under the $U$-action and
explicitly $$(\calO_B \otimes \ttS_n)^U
\cong \bigoplus_{i=0}^{n-1} k[x^{\pm 1}] \cdot \big(z(x-x^{-1})+y \big)^i.
$$
But an element $\big( \begin{smallmatrix}
a & 0 \\ 0 & a^{-1}
\end{smallmatrix} \big) \in T$ acts on $(z(x-x^{-1})+y)^i$ by multiplication by $a^{n-2i}$. It follows that
$$
(\calO_B \otimes \ttS_n) ^B
\cong k[x^{\pm 1}] \cdot (z(x-x^{-1})+y)^{n/2}.
$$
Restricting to $T$, or equivalently setting $y =0$, we see that the map $\bfJ_B(V) \to \bfJ_T(V)$ can be identified with the inclusion
$$
k[x^{\pm 1}] \cdot (z(x-x^{-1}))^{n/2}
 \to k[x^{\pm 1}] \cdot  z^{n/2}.
$$
The length of the modification of the above map at each point of the divisor $x-x^{-1}$ is $n/2$. This  completes the proof of the lemma.
\end{proof}

\begin{lem}
\label{L:modification SL3}
Proposition~\ref{P:modification of vector bundles over S} holds for $G=\SL_3$ with the automorphism $\sigma$ given in Example~\ref{Ex:folding A2n}, at the generic points of $A_\calO$ for the $\sigma$-orbit $\calO$ consisting of two roots.
\end{lem}
\begin{proof}
Let $\alpha, \sigma \alpha, \beta =\alpha+ \sigma\alpha$ denote the positive roots and let $\calO = \{\alpha, \sigma \alpha\}$ so that $\alpha_\calO = \beta$.

We consider the principal $\SL_2$ of $\SL_3$. Its image is exactly $H=\SO_3=(\SL_3)^\sigma$, where $\SO_3$ is the orthogonal group defined by $\Big(\begin{smallmatrix}& & 1 \\ & -1& \\ 1&&\end{smallmatrix}\Big)$, regarded as a symmetric bilinear form. 
The map $T^\sigma\to T_\sigma=A$ is the square mapping on $\GG_m$. One checks immediately that when $\cha k\neq 2$, $\bfJ_{T}\to \bfJ_{T^\sigma}$ is finite \'etale of degree two and the induced map
\[\bfJ_{T}(V)\otimes_{\bfJ_T}\bfJ_{T^\sigma}\to \bfJ_{T^\sigma}(V)\]
is an isomorphism.

The embedding $H \to G$ induces $\Spec \bfJ_H = H/\!\!/c(H)\to\Spec \bfJ_G = G/\!\!/c_\sigma(G)$, where $c$ denotes the usual conjugation of $H$ on $H$.
Here $\bfJ_H=k[y]$ (resp. $\bfJ_G=k[x]$) is the space of conjugation invariant functions on $\SO_3$ (resp. twisted conjugation invariant functions on $\SL_3$), with the variable $y$ (resp. $x$) representing the character of the standard representation of $\SO_3$ (resp. the adjoint representation of $\SL_3$).
Note that $x=(y-1)^2$ when viewed as functions on $T^\sigma$. So $k[x, x^{-1}]\to k[y,(y-1)^{-1}]$ is finite \'etale of degree two.

The point $x=4$ is the image of the divisor $A_\calO$ under the isomorphism $A /\!\!/ W_0 \cong \Spec \bfJ_G$. Its preimages under the map $\Spec \bfJ_H \to \Spec \bfJ_G$ are points represented by $y = -1$ and $y=3$, where $y=3$ corresponds to the Weyl group orbit of the zero of $e^\beta -1$ on $T^\sigma$.
Let $u=x-4$, and $v=y-3$. Then the map $k[x]\to k[y]$ induces an isomorphism $k[[u]]\cong k[[v]]$. We claim that the natural map
\begin{equation}
\label{E:local principal SL2}
\bfJ_G(V)\otimes_{\bfJ_G}k[[u]]\to \bfJ_H(V)\otimes_{\bfJ_H}k[[v]]\end{equation}
is an isomorphism.

By \cite[Theorem 6.11 c) ]{St2} (or a direct computation), the preimage of  $y=3$ under the quotient map $\SO_3^\reg \to \Spec \bfJ_H = \Spec k[y]$ is a single $H$-orbit of  $g(0) = \Big(\begin{smallmatrix}1& 1& \frac 12 \\ & 1& 1\\ &&1 \end{smallmatrix}\Big)$.
One can easily check that the image of $g(0)$ in $\SL_3$, denoted by $h(0)$, is also $\sigma$-regular. Using the smoothness of the Chevalley map $\SO_3^{\textrm{reg}}\to \Spec k[y]$, we may lift $g(0)$ to $g:\Spec k[[v]]\to \SO_3^{\textrm{reg}}$. Since $\SL_3^{\sigma\textrm{-reg}}$ is open and naturally $\Spec k[[u]]\cong \Spec k[[v]]$, we get a composition of maps
$$
h: \Spec k[[u]] \cong \Spec k[[v]] \xrightarrow{g} \SO_3 \to \SL_3
$$
whose image lands in $\SL_3^{\sigma\textrm{-reg}}$. Now, viewing $g$ and $h$ as $k[[u]]$-valued points on $\SO_3^\reg$ and $\SL_3^{\sigma\textrm{-reg}}$ respectively, the centralizer $Z_H(g)$ of $g$ in $\SO_3$ is naturally contained in the twisted centralizer $I_h$ of $h$ in $\SL_3$. We claim that $Z_H(g) = I_h$. Indeed, this is true over the generic point of $\Spec k[[u]] \cong \Spec k[[v]]$ as both spaces are connected by Lemma~\ref{E: sigmaregularinT}. Over the special point, one can compute explicitly that $$Z_H(g(0)) = I_{h(0)} = 
\bigg\{ \bigg(
\begin{smallmatrix}1& x& \frac 12 x^2 \\ & 1& x\\ &&1 \end{smallmatrix} \bigg)\ ; \ x \in k
\bigg\}.
$$
Thus $Z_H(g) = I_h$.  Since a $V$-valued function is determined by its restriction to the regular locus, by Corollary~\ref{C:fiber regular} (1), we deduce that
\[\bfJ_H(V)\otimes_{\bfJ_H}k[[u]]= (V\otimes k[[u]])^{Z_H(g)}= (V\otimes k[[v]])^{I_h}= \bfJ_G(V)\otimes_{\bfJ_G}k[[v]].\]
This shows that \eqref{E:local principal SL2} is an isomorphism.

Now we immediately reduce the lemma to the $\SL_2$-case (Lemma~\ref{L:modification SL2}) by the same argument as in the last two paragraphs of the proof of Lemma~\ref{L:reduce to rank one groups} (based on Lemma~\ref{L: isomorphism locus2}), provided that we can show that $V|_{\SO_3}$ admits a good filtration. It suffices to check this for Schur modules $\ttS_{a\omega_1+  b\omega_2}$ with $a, b \geq 0$, where $\omega_1$ and $\omega_2$ are the fundamental weights of $\SL_3$. 
We check by induction on $a+b$.  This is true for $\omega_1$ and $\omega_2$ as they are standard representations of $\SO_3$, and therefore true for $\ttS_{\omega_1}^{\otimes a}\otimes \ttS_{\omega_2}^{\otimes b}$ by Theorem~\ref{T:prop of good fil} (2).
By the Frobenius reciprocity, there is a unique zero (up to scalar) map $\ttS_{\omega_1}^{\otimes a}\otimes \ttS_{\omega_2}^{\otimes b} \to \ttS_{a\omega_1 + b\omega_2}$. One easily identify the kernel $V'$ as $(\ttS_{\omega_1}^{\otimes a}\otimes \ttS_{\omega_2}^{\otimes b})_{\prec a\omega_1+b\omega_2}$ in Lemma~\ref {L:good sub} (2), and therefore admits a good filtration (whose graded pieces are Schur modules with highest weights strictly lower than $a\omega_1 + b\omega_2$). So the restriction of $V'$ to $\on{SO}_3$ admits a filtration by inductive hypothesis. Moreover, by Lemma~\ref{L:good sub} (2), $\ttS_{\omega_1}^{\otimes a}\otimes \ttS_{\omega_2}^{\otimes b} \to \ttS_{a\omega_1 + b\omega_2}$ is surjective.  By Theorem~\ref{T:prop of good fil} (3), we conclude that $\ttS_{a\omega_1 + b\omega_2}$ admits a good filtration when restricted to $\SO_3$.
\end{proof}

\begin{remark}
\label{R:y = -1}
As in the proof of the lemma, we have an analogous map $$\bfJ_G(V) \otimes_{\bfJ_G} k[[x-4]] \to \bfJ_H(V) \otimes_{ \bfJ_H} k[[y+1]].$$ But it is not an isomorphism in general.  The reason is that $y=-1$ lies on $\mathring T^\sigma /\!\!/ W_0$, so every regular element in $\SO_3^\reg$ that maps to the point $y=-1$ in $\Spec k[y]$ is conjugate to the \emph{semisimple} matrix $\Big(\begin{smallmatrix}-1&0& 0 \\ & 1& 0\\ &&-1 \end{smallmatrix}\Big)$; yet this matrix when viewed inside $\SL_3$ is \emph{not} $\sigma$-regular.
\end{remark}

\medskip

\subsection{The ring $\End(\widetilde V)$}
\label{SS:endoring}
Motivated by some applications in \cite{XZ}, in this subsection we study the endomorphism ring $\End(\widetilde V)$ of the vector bundle $\widetilde{V}$ on $[G\sigma/G]$. Here are some basic facts.

\begin{enumerate}
\item The ring $\End(\widetilde{V})$ is a $\bfJ$-algebra. In particular, if $W$ is a representation of $G\rtimes \langle\sigma\rangle$, its character $\chi_W$, restricted to $G\sigma$, defines a regular function on $[G\sigma/G]$ and therefore an element in $\End(\widetilde{V})$, still denoted by $\chi_W$.

\item  The Chevalley restriction map is the pullback map $\End(\widetilde{V})\to \End(\widetilde{V}|_{[T\sigma/T]})$ along $[T\sigma/T]\to [G\sigma/G]$. Note that
\[
\End(\widetilde{V}|_{[T\sigma/T]})\cong  (k[T]\otimes \End_{T^\sigma}V)^{T/T^\sigma}.
\] 
So \eqref{E:Res otimes 1} gives an injective map
\begin{equation*}
\label{E:res for end}
\Res_{V\otimes V^*}^\sigma\otimes 1: \End(\widetilde{V})\otimes_{\bfJ_G}\bfJ_T\to (k[T]\otimes \End_{T^\sigma}V)^{T/T^\sigma},
\end{equation*}
which is an isomorphism generically over $\bfJ_T$. It follows that $\End(\widetilde{V})$ is commutative if and only if $V$ is multiplicity free as a $T^\sigma$-module.

\item Assume that $V$ is the restriction of a representation $\rho: G\rtimes\langle \sigma\rangle\to \GL(V)$ to $G$. There is an element $\ga_{\on{taut}}\in \End(\widetilde V)$, given by a tautological automorphism of $\widetilde{V}$. Namely, consider the trivial vector bundle  $\mO_{G\sigma}\otimes V$ over $G\sigma$. There is the tautological automorphism of this bundle whose restriction to the fiber over $g\sigma\in G\sigma$ is given by the automorphism $\rho(g\sigma)$ of $V$. This automorphism commutes with the $G$-equivariant structure on $\mO_{G\sigma}\otimes V$, and therefore descends to
the desired $\ga_{\on{taut}}$. Note that
$$(\Res^\sigma_{V\otimes V^*}\otimes 1)(\gamma_{\on{taut}}) = \sum_{\mu}e^{\sigma(\mu)}\otimes  \sigma_\mu, $$  
where the sum is taken over all weights in $V$ and $\sigma_\mu: V(\mu)\to V(\sigma(\mu))$ is the natural isomorphism.
\end{enumerate}

\begin{prop}[Cayley--Hamilton]  
\label{P:CH}
Write $d=\dim V$. Define a degree $d$ polynomial $f(x)\in \bfJ[x]$ as 
\[
f(x)=\sum (-1)^i \chi_{\wedge^{d-i}V}x^i.
\]
Then there is an injective map of $\bfJ$-algebras defined by
\begin{equation}
\label{E: ring map}
\bfJ[x]/(f(x))\to \End(\widetilde{V}),\quad x\mapsto \ga_{\on{taut}}.
\end{equation}
In particular,
\begin{equation}
\label{E: CayHam}
\sum (-1)^i \chi_{\wedge^{d-i} V}\ga_{\on{taut}}^i=0
\end{equation}
as elements in $\End(\widetilde{V})$. 
\end{prop}
\begin{proof}For $G=\GL_n$, $\sigma=\id$ and $V=\on{std}$ being the standard representation, the usual Cayley--Hamilton theorem gives \eqref{E: CayHam}. 
In general, the representation defines a homomorphism $G\rtimes\langle\sigma\rangle\to\GL(V)$, which induces $[G\sigma/G]\to [\GL(V)/\GL(V)]$. The pullback of the bundle $\widetilde{\on{std}}$ on $[\GL(V)/\GL(V)]$ along this map is exactly $\widetilde{V}$,  the pullback of the tautological automorphism $\gamma_{\on{taut}}$ of $\widetilde{\on{std}}$ is the tautological automorphism of $\widetilde{V}$, and the pullback of the regular function $\chi_{\wedge^i \on{std}}$ is just $\chi_{\wedge^i V}$. Therefore, \eqref{E: CayHam} holds in general. Or equivalently, \eqref{E: ring map} is well-defined. It remains to show that it is injective.

If $G=\GL_n, \sigma=\id$ and $V=\on{std}$, by Lemma~\ref{L: strong minu} below, \eqref{E: ring map} is an isomorphism. On the other hand, for a representation $G\rtimes\langle\sigma\rangle\to \GL(V)$, the map
\[\End(\widetilde{\on{std}})\otimes_{\bfJ_{\GL(V)}}\bfJ_G\to \End_{\on{Coh}([G\sigma/G])}(\widetilde{V})\]
is injective. Therefore, the injectivity of \eqref{E: ring map} in the general case follows.
\end{proof}
\begin{rmk}
Note that the proof of the above proposition holds for \emph{any} algebraic group $G$.
\end{rmk}
\begin{rmk}
\label{R:generalization}
We also briefly explain how to study $\End(\widetilde V)$ if $V$ is not a representation of $G\rtimes\langle\sigma\rangle$, but only a representation of a subgroup $G\rtimes\langle\tau\rangle$, where $\tau=\sigma^f$. In this case, we have the map
\[[G\sigma/G]\to [G\tau/G],\quad g\sigma\mapsto (g\sigma)^f=g\sigma(g)\sigma^2(g)\cdots\sigma^{f-1}(g)\tau.\]
Then we can pullback everything from $[G\tau/G]$. Note that for a representation $W$ of $G\rtimes\langle\tau\rangle$, $W\otimes \sigma(W)\otimes\cdots\otimes \sigma^{f-1}(W)$ is a natural representation of $G\rtimes\langle\sigma\rangle$, usually called the tensor induction representation. Let us denote it by $\otimes_{\tau}^{\sigma}W$. Then
the pullback of $\chi_W$ to $[G\sigma/G]$ is $\chi_{\otimes_{\tau}^{\sigma}W}$. So we have
\[
\sum (-1)^i \chi_{\otimes_\tau^\sigma(\wedge^{d-i}V)}\ga_{\on{taut}}^i=0.
\]
\end{rmk}

As mentioned above, for a general $V$ which is not multiplicity free as a $T^\sigma$-module, the ring $\End(\widetilde{V})$ is non-commutative.
So $\End(\widetilde{V})$ is not generated by $\ga_{\on{taut}}$ as a $\bfJ$-algebra. Even if $\End(\widetilde{V})$ is commutative, \eqref{E: ring map} will only be an isomorphism away from a divisor on $\Spec \bfJ_G$.
However, we can say more in the following special case. 

\begin{lem}
\label{L: strong minu}
The map \eqref{E: ring map} is an isomorphism if
\begin{itemize}
\item $G=\GL_n$ or $\SL_n$ equipped with the trivial $\sigma$-action, and $V$ is the standard representation or its dual.
\item $G=\on{Sp}_{2n}$ and $V$ is the standard representation;
\end{itemize}
\end{lem}
\begin{proof}
It is enough to show that $\bfJ_G[x]/(f(x))\otimes_{\bfJ_G}\bfJ_T=\bfJ_T[x]/(f(x))\to \End(\widetilde{V})\otimes_{\bfJ_G}\bfJ_T$ is an isomorphism. For this, we calculate the composition
\begin{equation}
\label{E:Cayley Hamilton tensor JT}
\bfJ_G[x]/(f(x))\otimes_{\bfJ_G}\bfJ_T=\bfJ_T[x]/(f(x))\to \End(\widetilde{V})\otimes_{\bfJ_G}\bfJ_T\to \bfJ_T\otimes \End_{T}(V).
\end{equation}
Under this map \eqref{E:Cayley Hamilton tensor JT}, the image of $\ga_{\on{taut}}$ is $\sum e^{\la_j}\otimes \id_{V(\la_j)}$, where the sum is taken over all weights $\{\la_j\}$ of $V$, and $\id_{V(\la_j)}$ is the identity map of $V(\la_j)$. Then the image of $\ga^i_{\on{taut}}$ is $\sum e^{i\la_j}\otimes \id_{V(\la_j)}$. 
Since $V$ is minuscule, $\{\id_{V(\la_j)}\}$ form a basis of $\bfJ_T\otimes \End_{T}(V)$ as a $\bfJ_T$-module. By Vandermonde, the determinant of the composed map \eqref{E:Cayley Hamilton tensor JT} is
\begin{equation*}
\label{E:determinant of CH}
\prod_{j<j'} (e^{\la_j}-e^{\la_{j'}})= e^{\mu}\prod_{j<j'} (1-e^{\la_j-\la_{j'}}),
\end{equation*}
from some $\mu \in \XX^\bullet(T)$. For the cases in the lemma,  the difference $\lambda_j - \lambda_{j'}$ is a root $\alpha$, which is equal to the determinant (up to a unit in $\bfJ_T$) of the second map in \eqref{E:Cayley Hamilton tensor JT} by Proposition~\ref{P:modification of vector bundles over S}. (Note that Proposition~\ref{P:modification of vector bundles over S} also holds for $\GL_n$.) Therefore, the first map in \eqref{E:Cayley Hamilton tensor JT} is an isomorphism.  
\end{proof}

\subsection{Some examples}
\label{SS:example of matrix}
In this subsection, we present some examples of the calculation of the determinant of \eqref{E:intersection matrix}.
\begin{example}
Consider the case where $ G = \prod_{i=1}^{d}  \SL_2$ and $\sigma$ permutes all $d$ factors, and let $V=V_{a_1} \boxtimes \cdots \boxtimes V_{a_d}$ be the exterior tensor product representation, where $V_a$ is the $a$th symmetric power representation of $\SL_2$ (i.e. the representation of $\SL_2$ on the space of homogeneous polynomials in $(x,y)$ of degree $a$). By Lemma~\ref{L:reduce from G to G0}, this is equivalent to the case where $G=\SL_2$, $\sigma=\id$, and $V= V_{a_\bullet}:=V_{a_1}\otimes\cdots \otimes V_{a_d}$ (so that $r_V  = V_{a_\bullet}(0) \neq 0$). Now
assume that $a_1+\cdots +a_d $ is even.
In this case, Theorem~\ref{T:determinant of intersection} says that the determinant of the matrix \eqref{E:intersection matrix} is equal to
\[
c\cdot (e^\alpha -1)^{\zeta}(e^{-\alpha}-1)^\zeta
\]
for some $c\in k^\times$, 
where 
$$
\zeta=\sum_{n \geq 1} \dim V_{a_\bullet}(n) = \frac{\dim V_{a_\bullet} - \dim V_{a_\bullet}(0)}2.
$$

In fact, the matrix for the pairing \eqref{E:intersection matrix} (for some appropriate basis) can be obtained via a combinatorial description in terms of periodic meanders, which carries a quantum deformation. 
In the special case when $ G = \SL_2$, $d$ is even, and $a_1 = \dots= a_d =1$, this matrix \eqref{E:intersection matrix} is the Gram matrix for the periodic meanders (see \cite{TX}). 
\end{example}

\begin{example}
Consider the case when $G = \SL_{2r+1}$ with the non-trivial pinned $\sigma$-action as explained in  Example~\ref{Ex:folding A2n}. We use freely the notation therein.
In particular, $\XX^\bullet(T)$ is generated by the characters $\varepsilon_{-r}, \dots, \varepsilon_r$ with the relation that $\sum_{i=-r}^r \varepsilon_i = 0$, and $\sigma$ acts by
$\sigma(\varepsilon_i) =- \varepsilon_{-i}$ for each $i$.
The Bruhat partial order on $\XX^\bullet(T)$ is generated by $\varepsilon_i \succeq \varepsilon_{i+1}$ for $i=-r, \dots, r-1$.
For $A =T / (\sigma-1)T$,  $
\XX^\bullet(A) = \xch( T)^{\sigma} = \bigoplus_{i=1}^r \ZZ \cdot (\varepsilon_{-i}-\varepsilon_i) $.
The absolute Weyl group $W\cong S_{2r+1}$  permutes $\varepsilon_{-r}, \dots, \varepsilon_r$, and its $\sigma$-invariant elements are $$W_0=W^{\sigma} = \langle (i,-i), (ij)(-i,-j)\;|\; i, j=1, \dots, r \rangle= (\ZZ/2\ZZ)^r \rtimes S_r .$$

Recall the $\sigma$-orbits of roots in \eqref{E:sigma orbits in SL2r+1}; in particular, $\alpha_{\calO_i^-}  = \alpha_{\calO_i^+} = \varepsilon_{-i} - \varepsilon_i$ for $i=1, \dots, r$. 
Write $\frakS_i$ for the $i$th elementary symmetric power in $e^{ \alpha_{\calO,1}}+e^{-\alpha_{\calO,1}}, \dots, e^{\alpha_{\calO,r}}+e^{-\alpha_{\calO,r}}$,  then
$$
\bfJ: = k[ G \sigma]^{G}\cong k[\xch(A)]^{W_0}\cong k[\frakS_1, \dots, \frakS_r]\subset k[\xch(T)]=k[e^{\pm\varepsilon_{-r}},\dots, e^{\pm\varepsilon_{-1}}, e^{\pm\varepsilon_1}, \dots,e^{\pm\varepsilon_r} ].$$
When $r = 1$, $ {\bfJ} = k[e^{\alpha_{\calO,1}}+e^{-\alpha_{\calO, 1}}]$.

\medskip
We consider the standard representation $V = \ttS_{\varepsilon_{-r}} = \std$ of $G = \SL_{2r+1}$. Observe that the only nonzero weight space of $V$ with weights in $(\sigma-1)\XX^\bullet(T)$ is $V(\varepsilon_0)$. We write
$$
\varepsilon_0 = \sigma(\nu) - \nu ,  \quad \textrm{for} \quad \nu = \varepsilon_{-r} + \dots + \varepsilon_{-1}.
$$
The representation $\ttS_{\nu} \cong \ttS_{\sigma(\nu)^*} \cong \wedge^r \std$ is the $r$th wedge product of the standard representation.
Then it is easy to see that 
$$
(\ttS_{\sigma(\nu^*)} \otimes \ttS_\nu  \otimes V)^G = (\wedge^r \std \otimes \wedge^r \std \otimes \std)^G \cong (\wedge^{2r+1}\std)^G \cong k,
$$
Let $\bbb$ denote the element on the left hand side that corresponds to $1 \in k$, or more precisely $e_{-r} \wedge\cdots \wedge e_r$.
Then the recipe in \eqref{E:Xi map} defines a function $f_\bbb \in \bfJ(V)$ (note that in our case $\ttW_{\nu} \cong \ttS_\nu$).  We will compute its restriction to $T\sigma$.

Let $e_{-r}, \dots, e_r$ denote the standard basis of the representation $\std$. Then  $\wedge^r\std$ has a basis $(e_I)_{I \subset\{-r, \dots, r\},|I|= r}$, where for $I = \{i_1,\dots, i_r\}$ with elements ordered increasingly, we write $e_I = e_{i_1} \wedge \cdots \wedge e_{i_r}$. The natural map $\sigma: \ttS_\nu = \wedge^r \std \to \ttS_{\sigma(\nu)} \cong (\wedge^r\std)^*$ is given by sending $e_i$ to $(-1)^ie_{-i}^*$.\footnote{This is in fact $(-1)^r$ times the natural map given by \eqref{Ex:sigma A2n}.} 
So $\sigma(e_I) =\sgn(I) \cdot e_{-I}^*$, where $-I = \{-i_r, \dots, -i_1\}$ and $\sgn(I) = (-1)^{i_1+\dots + i_r+ r(r-1)/2}$.
Using this notation, we write explicitly 
$$
\bbb  = 
\sum_{|I| = r, |I'| = r} e_I \otimes e_{I'} \otimes v_{I,I'}, \quad 
f_\bbb(t\sigma) =  \sum_{|I| = r, |I'| = r}    \langle  e_I, t\sigma e_{I'} \rangle \cdot v_{I, I'},
$$
where $v_{I,I'} \in V = \std$, $t \in T$, and
the sums run through all subsets $I, I'$ of $\{-r, \dots, r\}$ of cardinality $r$.
In the expression for $f_{\bbb}(t\sigma)$, the pairing $\langle e_I, t\sigma e_{I'} \rangle$ is nonzero precisely when $I' = -I$. 
In the expression for $\bbb$, the vector $v_{I, I'}$ is nonzero precisely when $I$ and $I'$ are disjoint.
So we need only to discuss the case when there is $\underline s = \{  s_1, \dots, s_r\} \in \{\pm 1\}^r$ such that  $I = I_{\underline s} = \{s_1\cdot 1, \dots, s_r \cdot r\}$ and $I' = -I$. In this case, $\sgn(- I_{\underline s}) = (-1)^{r(r+1)/2 + r(r-1)/2} = (-1)^r$ is independent of $\underline s$. Moreover, we may take $v_{I, I'} = (-1)^{r} \sgn(\underline s) \cdot e_0$, where the sign $\sgn(\underline s)=(-1)^u$ with $u$ equal to the number of $-1$'s in $\underline s$. Thus,
$$
\Res_V^\sigma(f_\bbb) = \sum_{\underline s\in \{\pm 1\}^r} \sgn (\underline s) e^{-\varepsilon_{-s_1}- \varepsilon_{-2s_2} - \cdots - \varepsilon_{-rs_r}}(t) \cdot e_0 = \prod_{i=1}^r \big(e^{-\varepsilon_{-i}} - e^{-\varepsilon_i}\big) \cdot e_0 \in \bfJ_T(V).
$$

In the dual picture,  we consider an element
$$
\bbb^\vee \in (V_{\sigma(\nu)} \otimes V_{\nu^*} \otimes V^*)^G  = \big( (\wedge^r\std^*) \otimes (\wedge^r \std^*) \otimes \std^*\big)^G \cong (\wedge^{2r+1}\std^*)^G \cong k.
$$
Explicitly, we write 
$$
\bbb^\vee = \sum_{|I| = r, |I'| = r} e_I \otimes e_{I'}^* \otimes v^\vee_{I,I'}, \quad f_{\bbb^\vee}(t\sigma) =   \sum_{|I| = r, |I'| = r}   \langle  e_I, t\sigma e^*_{I'} \rangle \cdot e_{I,I'}^\vee.
$$
Similar to above, we need only to discuss the case when $I = I_{\underline s}$ for some $\underline s \in \{\pm 1\}^r$ and $I' = -I$. In this case, we may choose $v_{I, I'} = (-1)^{r}\sgn(\underline s) \cdot e_0^*$. Thus,
$$
\Res_{V^*}^\sigma(f_{\bbb^\vee})= \sum_{\underline s  \in \{\pm 1\}^r} \sgn (\underline s) e^{\varepsilon_{-s_1}+ \varepsilon_{-2s_2} + \cdots + \varepsilon_{-rs_r}} (t)\cdot e^*_0 = \prod_{i=1}^r \big(e^{\varepsilon_{-i}} - e^{\varepsilon_i}\big) \cdot e_0^* \in \bfJ_T(V^*).
$$

Now we compute $\langle f_\bbb, f_{\bbb^\vee}\rangle$ under the pairing
$\bfJ(V) \otimes_\bfJ \bfJ(V^*) \to \bfJ \cong \bfJ_T^{W_0}$:
\[
\langle f_\bbb, f_{\bbb^\vee}\rangle
= \prod_{i=1}^r \big(e^{-\varepsilon_{-i}} - e^{-\varepsilon_i}\big) \cdot \prod_{i=1}^r \big(e^{\varepsilon_{-i}} - e^{\varepsilon_i}\big) = \prod_{i = 1}^r \big( (e^{\alpha_{\calO^-_i}} -1)( e^{-\alpha_{\calO^-_i}}-1) \big)
.\]
This agrees with Theorem~\ref{T:determinant of intersection} by noting that $\zeta_{\mO^-_i}(V) = 1$ and $\zeta_{\mO_i^+}(V) = \zeta_{\mO_{i,j}}(V) = 0$, for the $\sigma$-orbits in \eqref{E:sigma orbits in SL2r+1}. 
\end{example}

\begin{example}
Consider $\SL_4$ with row and column indices in $\{-2, -1, 1, 2\}$, the pinning $(B, T, e)$ given by the subgroup of standard upper triangular matrices, the subgroup of diagonal matrices, and $e = E_{-2, -1}+ E_{-1, 1} + E_{1, 2}$. The unique non-trivial pinned automorphism $\sigma$ is given by $\sigma(X) = J\, {}^t\!X^{-1} J^{-1}$ for $X \in \SL_4$, with $J = \bigg( \begin{smallmatrix}
0&0&0&1\\0&0&-1&0 \\ 0& 1&0&0\\ -1&0&0&0
\end{smallmatrix}\bigg)$.

Let $\varepsilon_i$ for $i = \pm 1, \pm2$ be the character of $T$ given by evaluating at the $(i,i)$-entry; so that $\XX^\bullet(T)$ is generated by these $\varepsilon_i$ with the relation $\varepsilon_{-2} + \varepsilon_{-1} + \varepsilon _1 + \varepsilon _2 =0$.
Under the $\sigma$-action $\sigma(\varepsilon_i) = -\varepsilon_i$, the following are the $\sigma$-orbits of $\Phi(G, T)$:
$$
\calO_{i} = \{\varepsilon_{-i} - \varepsilon_{i}\}, \quad \calO_{\pm 2, \pm 1} = \{\varepsilon_{\mp 2} - \varepsilon_{\mp1}, \varepsilon_{\pm 1} - \varepsilon_{\pm 2}\}, \quad \calO_{\mp 2, \pm 1} = \{\varepsilon_{\pm 2} - \varepsilon_{\mp1}, \varepsilon_{\pm 1} - \varepsilon_{\mp 2}\}
$$
with $i \in \{\pm 1, \pm 2\}$. They are all of type $\sfA$.
For $A = T/(\sigma-1)T$,
$$\XX^\bullet(A) = \XX^\bullet(T)^\sigma = \ZZ (\varepsilon_{-2} - \varepsilon_2) \oplus \ZZ (\varepsilon_{-1} +\varepsilon_{-2}).
$$
(Note that $(\varepsilon_{-2} - \varepsilon_2) + (\varepsilon_{-1} - \varepsilon_1) = 2(\varepsilon_{-2} + \varepsilon_{-1})$.)
The root system $\Phi(G_\sigma, A)$
consists of short roots $\alpha_{\calO_i}$ and long roots $\alpha_{\calO_{i,j}}$; the group $G_\sigma \cong \Spin_5$.
The absolute Weyl group of $\Phi(G, T)$ is $W \cong S_4$ given by permuting the $\varepsilon_i$'s, and its $\sigma$-invariant elements are
$W_0 = W^\sigma = \langle (-1, 1), (-2, 2), (1, 2)(-1, -2)\rangle= (\ZZ/2\ZZ)^2 \rtimes S_2$.
Write 
$$
\frakS_1 =
e^{\varepsilon_{-1}-\varepsilon_1}+ e^{\varepsilon_{1}-\varepsilon_{-1}} + 
e^{\varepsilon_{-2}-\varepsilon_2}+ e^{\varepsilon_{2}-\varepsilon_{-2}} , \quad \frakS_2 = e^{\varepsilon_{-2}+\varepsilon_{-1}} + e^{-\varepsilon_{-2}-\varepsilon_{-1}} + e^{\varepsilon_{1}+\varepsilon_{2}}+ e^{-\varepsilon_1-\varepsilon_2}.$$ Then 
\begin{align*}
\bfJ: = k[ G \sigma]^{G}\cong&\; k[\xch(A)]^{W_0}\cong k[\frakS_1, \frakS_2]
\\&
\subset k[\xch(T)]=k[e^{\pm\varepsilon_{-2}}, e^{\pm\varepsilon_{-1}}, e^{\pm\varepsilon_1},e^{\pm\varepsilon_2} ] /(e^{\varepsilon_{-2} + \varepsilon_{-1} + \varepsilon_1 + \varepsilon_2} -1).
\end{align*}

Consider the minuscule representation $V = \ttS_{\varepsilon_{-2}+ \varepsilon_{-1}}$ of $\SL_4$; it is isomorphic to $\wedge^2\std$ for the standard representation $\std$ of $\SL_4$. There are two nonzero weight spaces with weights in $(\sigma-1)\XX^\bullet(T)$: $V(\lambda_1)$ and $V(\la_2)$ with $\lambda_i = \varepsilon_{-i} + \varepsilon_i$. They are both of one-dimensional. So $r_V=2$.
The minimal dominant weights $\nu_i$ (in the sense as in Lemma~\ref{L: minimalnusimplyconnected}) satisfying $\sigma(\nu_i) - \nu_i = \lambda_i$ are
$$
\nu_1 = \varepsilon_{-2}, \quad  \nu_2 =-\varepsilon_2 .
$$
We remark here that for $\SL_n$ with $n\geq 5$ and the non-trivial automorphism $\sigma$, and for $V = \wedge^2  \std$, we have $r_V = \lfloor \frac n2\rfloor$ and some (or rather most) of the minimal dominant weights $\nu$ above are no longer minuscule. So the computation will be much more involved. 

Back to our case,  $\ttS_{\nu_1} \cong \ttS_{\sigma(\nu_1)^*}\cong \std$ and $\ttS_{\nu_2} \cong \ttS_{\sigma(\nu_2)^*} \cong \std^*$.
In what follows, we will describe nonzero elements
$$
\bbb_1 \in (\ttS_{\sigma(\nu_1)^*} \otimes \ttS_{\nu_1} \otimes V)^G = (\std \otimes \std \otimes \wedge^2 \std)^G, \quad \bbb_2 \in (\ttS_{\sigma(\nu_2)^*} \otimes \ttS_{\nu_2} \otimes V)^G = (\std^* \otimes \std^* \otimes \wedge^2 \std)^G,
$$
use the recipe in \eqref{E:Xi map} to defines a function $f_{\bbb_1}, f_{\bbb_2} \in \bfJ(V)$ (note that in our case $\ttW_{\nu_i} \cong \ttS_{\nu_i}$), and describe the restriction of $f_{\bbb_i}$ to $T\sigma$.

Write $\{e_{-2}, e_{-1}, e_1, e_2\}$ for the standard basis of $\std$ and $\{e_{-2}^*, e_{-1}^*, e_1^*, e_2^*\}$ the dual basis; and the isomorphism $\sigma: \std \cong \std^*$ is given by
$$
e_{-2} \mapsto e_2^*, \quad e_{-1} \mapsto -e_1^*, \quad e_1 \mapsto e_{-1}^*, \quad e_2 \mapsto -e_{-2}^*.
$$
With this notation, we write explicitly,
$$
\bbb_1  = 
\sum_{i,j \in \{\pm 1, \pm 2\}} e_i \otimes e_j \otimes v_{i,j}, \quad 
f_{\bbb_1}(t\sigma) =  \sum_{i,j \in \{\pm 1, \pm 2\}}    \langle  e_i, t\sigma e_j \rangle \cdot v_{i,j},
$$
where $v_{i,j} \in V = \wedge^2\std$ and $t \in T$.
In the expression for $f_{\bbb_1}(t\sigma)$, the pairing $\langle e_i, t\sigma e_j \rangle$ is nonzero precisely when $i=-j$. 
In the expression for $\bbb_1$, the vector $v_{i,j}$ is nonzero precisely when $i \neq j$.
So we need only to discuss the case when $i =- j \in \{\pm 1, \pm 2\}$. In this case, we may choose $v_{-2,2} = -v_{2, -2} =-  e_{-1} \wedge e_1$ and $v_{-1, 1}= -v_{1,-1} =- e_{-2} \wedge e_2$. Thus,
$$
\Res_V^\sigma(f_{\bbb_1}) = (e^{-\varepsilon_{-2}} + e^{-\varepsilon_2}) e_{-1} \wedge e_1 - (e^{-\varepsilon_{-1}} + e^{-\varepsilon_{1}}) e_{-2} \wedge e_2 \in \bfJ_T(V).
$$
In a similar way, we may compute (up to choosing a sign for $\bbb_2$)
$$
\Res_V^\sigma(f_{\bbb_2}) = (e^{\varepsilon_{-1}} + e^{\varepsilon_{1}}) e_{-1} \wedge e_1 - (e^{\varepsilon_{-2}} + e^{\varepsilon_{2}}) e_{-2} \wedge e_2 \in \bfJ_T(V).
$$

On the dual side, we choose elements
$$
\bbb_1^\vee \in (\ttS_{\sigma(\nu_1)} \otimes \ttS_{\nu_1^*} \otimes V^*)^G = (\std^* \otimes \std^* \otimes \wedge^2 \std^*)^G, \quad \bbb_2^\vee \in (\ttS_{\sigma(\nu_2)} \otimes \ttS_{\nu_2^*} \otimes V^*)^G = (\std \otimes \std \otimes \wedge^2 \std^*)^G,
$$
and a computation similar to above gives the explicit formulas (up to choosing a sign for $\bbb_i^\vee$)
\begin{align*}
\Res_{V^*}^\sigma(f_{\bbb_1^\vee}) &= (e^{\varepsilon_{-2}} + e^{\varepsilon_{2}}) e_{-1}^* \wedge e_1^* - (e^{\varepsilon_{-1}} + e^{\varepsilon_{1}}) e_{-2}^* \wedge e_2^* \in \bfJ_T(V^*)
\\
\Res_{V^*}^\sigma(f_{\bbb_2^\vee}) &= (e^{-\varepsilon_{-1}} + e^{-\varepsilon_{1}}) e_{-1}^* \wedge e_1^* - (e^{-\varepsilon_{-2}} + e^{-\varepsilon_{2}}) e_{-2}^* \wedge e_2^* \in \bfJ_T(V^*).
\end{align*}

From this, we deduce that under the natural $\bfJ$-linear pairing $\bfJ(V) \times \bfJ(V^*) \to \bfJ$, we have
\begin{eqnarray*}
&
\langle f_{\bbb_1}, f_{\bbb_1^\vee}\rangle = \langle f_{\bbb_2}, f_{\bbb_2^\vee}\rangle = 
(e^{\varepsilon_{-2}} + e^{\varepsilon_{2}})(e^{-\varepsilon_{-2}} +e^{-\varepsilon_{2}})+(e^{\varepsilon_{-1}} + e^{\varepsilon_{1}})(e^{-\varepsilon_{-1}}+e^{-\varepsilon_{1}}) = \frakS_1+4,
\\
& \langle f_{\bbb_1}, f_{\bbb_2^\vee}\rangle = \langle f_{\bbb_2}, f_{\bbb_1^\vee}\rangle =  2(e^{-\varepsilon_{-1}} + e^{-\varepsilon_{1}}) (e^{-\varepsilon_{-2}} + e^{-\varepsilon_{2}})
 = 2(e^{\varepsilon_{-1}} + e^{\varepsilon_{1}}) (e^{\varepsilon_{-2}} + e^{\varepsilon_{2}})  = 2\frakS_2.
\end{eqnarray*}
(Note that $e^{\varepsilon_{-2} + \varepsilon_{-1} + \varepsilon_1+\varepsilon_2} = 1$.)
We can compute the determinant of the pairing as
\begin{align*}
\det\begin{pmatrix}
\langle f_{\bbb_1} , f_{\bbb_1^\vee}\rangle & \langle f_{\bbb_1} , f_{\bbb_2^\vee}\rangle\\
\langle f_{\bbb_2} , f_{\bbb_1^\vee}\rangle & \langle f_{\bbb_2} , f_{\bbb_2^\vee}\rangle
\end{pmatrix} = \,&
\big( (e^{\varepsilon_{-2}} + e^{\varepsilon_{2}})(e^{-\varepsilon_{-2}} +e^{-\varepsilon_{2}})+(e^{\varepsilon_{-1}} + e^{\varepsilon_{1}})(e^{-\varepsilon_{-1}}+e^{-\varepsilon_{1}})\big)^2 \\ &- 4 (e^{\varepsilon_{-1}} + e^{\varepsilon_{1}}) (e^{\varepsilon_{-2}} + e^{\varepsilon_{2}})(e^{\varepsilon_{-1}} + e^{\varepsilon_{1}}) (e^{\varepsilon_{-2}} + e^{\varepsilon_{2}})
\\
=\, &
\big( (e^{\varepsilon_{-2}} + e^{\varepsilon_{2}})(e^{-\varepsilon_{-2}} +e^{-\varepsilon_{2}})-(e^{\varepsilon_{-1}} + e^{\varepsilon_{1}})(e^{-\varepsilon_{-1}}+e^{-\varepsilon_{1}})\big)^2
\\
=\, &
\big( (e^{\varepsilon_{-2} - \varepsilon_2} - e^{\varepsilon_{-1}-\varepsilon_{1}})( 1- e^{\varepsilon_{2}-\varepsilon_{-2}+\varepsilon_{1}-\varepsilon_{-1}})\big)^2
\\
= \, & 
(e^{\alpha_{\calO_{-2, -1}}}-1)(e^{\alpha_{\calO_{-2, 1}}}-1) (e^{\alpha_{\calO_{2, -1}}}-1)(e^{\alpha_{\calO_{2, 1}}}-1).
\end{align*}

One can compare this with the computation that $\zeta_{\calO_i} (V) = 0$ and $\zeta_{\calO_{\pm 2, \pm 1}}(V) = \zeta_{\calO_{\pm2, \mp1}} (V) = 1$.
\end{example}

\begin{example}
\label{E: even Spin}
Assume that $\on{char}k=0$. 
Consider the case of $G = \Spin_{2r}$.
More precisely, consider $Q: = k^{\oplus 2r}$ with basis $e_{-r}, \dots, e_{-1}, e_1, \dots, e_r$ and a symmetric quadratic form $\langle e_i, e_{-j}\rangle = \delta_{ij}$. Let $$Cl =Cl(Q) =  Cl^\mathrm{even} \oplus Cl^\mathrm{odd} =  \Big( \bigoplus_{n \geq 0}Q^{\otimes n} \Big) \Big / \big( w\otimes w- \tfrac 12 \langle w, w\rangle; \, w\in Q\big) $$ denote the associated ($\ZZ/2\ZZ$-graded) Clifford algebra. 
There is a natural (anti-commutative) involution $\bullet^t$ on $Cl$ generated by sending $w \in Q$ to $-w$.
The pin group $\mathrm{Pin}_{2r}$ is formed by all elements $x \in Cl$ such that $x \cdot x^t = 1$, and $x Q x^t \subset Q$. The spin group $\Spin_{2r} : = \mathrm{Pin}_{2r} \cap Cl^\mathrm{even}$ is the neutral connected component of $\mathrm{Pin}_{2r}$.
The maximal torus $T$ of $\Spin_{2r}$ is given by the image of
$$\xymatrix@R=0pt{\Gamma:\GG_m^r \ar[r] & \Spin_{2r}
\\
(z_1, \dots, z_r) \ar@{|->}[r] & \prod_{i=1}^r \big( z_i e_ie_{-i} + z_i^{-1} e_ie_{-i}\big),}$$
where the kernel is $\{(z_1, \dots, z_r) \in \{\pm 1\}^r\; |\; z_1 \cdots z_r = 1\}$.
Let $\frac{\varepsilon_1}2, \dots, \frac{\varepsilon_r}2$ denote the usual basis of the character group of $\GG_m^r$; then
$$
\XX^\bullet(T) = \ZZ \cdot \tfrac 12 (\varepsilon_1 + \cdots + \varepsilon_r) \oplus  \bigoplus_{i=1}^{r-1} \ZZ \cdot \varepsilon_i.
$$
A weight $\sum_{i=1}^r a_i \varepsilon_i$ is dominant if $a_1 \geq \cdots \geq a_{r-1} \geq |a_r|$.

We equip with $\Spin_{2r}$ with the outer automorphism $\sigma$, induced by the conjugation by $e_r + e_{-r}$.
A simple computation shows that $$(e_r+e_{-r})( z_i e_ie_{-i} + z_i^{-1} e_ie_{-i})(e_r+e_{-r}) = \begin{cases}
z_i e_ie_{-i} + z_i^{-1} e_ie_{-i} & \textrm{if }i \neq r\\
z_r^{-1} e_re_{-r} + z_r e_re_{-r} & \textrm{if }i =r
\end{cases}
$$
So $\sigma$ fixes $\varepsilon_1, \dots, \varepsilon_{r-1}$ and maps $\varepsilon_r $ to $-\varepsilon _r$.
So for $A = T/(\sigma-1)T$,
its character group $
\XX^\bullet(A) = \XX^\bullet(T)^\sigma$ is the free abelian group with basis $\varepsilon_1, \dots, \varepsilon_{r-1}$.

The absolute Weyl group $W \cong H_r \rtimes S_r \subset \{\pm 1\}^r \times S_r$, where $H_r \subset \{\pm 1\}^r$ is the subgroup consisting of even number of $-1$'s,  $S_r$ permutes $\varepsilon_1, \dots, \varepsilon_r$ and $(h_1, \dots, h_r) \in H_r$ sends $\varepsilon_i$ to $h_i \varepsilon_i$ for $i=1, \dots, r$.
The $\sigma$-invariant elements of $W$ are $W_0 : = W^{\sigma}= \{\pm 1\}^{r-1} \rtimes S_{r-1}$,  where $S_{r-1}$ permutes $\varepsilon_1, \dots, \varepsilon_{r-1}$ and $(h_1, \dots, h_r) \in \{ \pm 1\}^{r-1}$ sends  $\varepsilon_i$ to $h_i \varepsilon_i$ for each $i = 1, \dots, r-1$.
It follows that the invariants of $\CC[\XX^\bullet(A)]$ under $W_0$ are
\begin{equation*}
\label{E:J GSpin}
\bfJ : = \CC[\XX^\bullet(A)]^{W_0} = \CC[\frakS_1, \dots, \frakS_{r-1}],
\end{equation*}
where $\frakS_i$ for $i=1, \dots, r-1$ is the $i$th elementary symmetric polynomial in $e^{\varepsilon_1} + e^{-\varepsilon_1}, \dots, e^{\varepsilon_{r-1}} + e^{-\varepsilon_{r-1}}$.

The root system $\Phi(G, T)$ consists of $\pm \varepsilon_i\pm \varepsilon_j$ and $\pm \varepsilon_i\mp\varepsilon_j$ for $i \in \{1, \dots, r\}$ and $j \in \{1, \dots, i-1\}$. The $\sigma$-orbits of these roots are
$$
\calO_{i,j}^{\pm} = \{ \pm\varepsilon_i\pm \varepsilon_j\}, \quad \calO_{j,i}^{\pm} = \{ \pm\varepsilon_i\mp \varepsilon_j\},\quad \calO_{i}^{\pm} = \{ \pm\varepsilon_i+ \varepsilon_r, \pm \varepsilon_i - \varepsilon_r\},
$$
with $i \in \{1, \dots, r\}, j \in \{1, \dots, i-1\}$; they are all of type $\sfA$.
Thus the root system $\Phi(G_\sigma, A)$ consists of short roots $\alpha_{\calO_{i,j}^\pm} = \pm \varepsilon_i \pm \varepsilon_j$ and long roots $\alpha_{\calO_i^\pm} = \pm 2\varepsilon_i$.

\medskip
We shall consider $V = \ttS_{\varepsilon_1} \cong Q$,  the vector representation of $G= \Spin_{2r}$. Its nonzero weight spaces with weights in $(\sigma-1)\XX^\bullet(T)$ are $V(-\varepsilon_r)$ and $V(\varepsilon_r)$.
We write
$$
\mp \varepsilon_r = \sigma(\nu_
\pm) - \nu_\pm  \quad\textrm{for} \quad \nu_\pm = \tfrac 12(\varepsilon_1+\cdots+ \varepsilon_{r-1} \pm \varepsilon_r).
$$
(Be careful with the signs here.)
So $\sigma(\nu_\pm) = \nu_\mp$.
We shall see that
\begin{equation}
\label{E:Spin invariants}
(\ttS_{\sigma(\nu_\pm^*)} \otimes \ttS_{\nu_\pm} \otimes V)^G \cong k.
\end{equation}
Indeed, since $V$ is minuscule, $\ttS_{\nu_\pm} \otimes V $ is a direct sum of $ \ttS_{\nu_\pm + \tau}$ with all $\tau \in  \{\pm \varepsilon_1, \dots, \pm \varepsilon_r\}$ such that $\nu_\pm +\tau$ is dominant. In particular, $\nu_\mp$ is among those weights $\nu_\pm + \tau$.
Dually, we have
\begin{equation}
\label{E:Spin invariants dual}
(\ttS_{\sigma(\nu_\pm)} \otimes \ttS_{\nu_\pm^*} \otimes V^*)^G \cong k.
\end{equation}

\end{example}

\begin{lem}
There exist bases $\bbb_\pm$ of \eqref{E:Spin invariants} and bases $\bbb_\pm^\vee$ of \eqref{E:Spin invariants dual} such that if we use $f_{\bbb_\pm} \in \bfJ(V)$ and $f_{\bbb_\pm^\vee} \in \bfJ(V^*)$ to denote the associated class functions via the recipe \eqref{E:Xi map}, then
\begin{equation}
\label{E:matrix for Spin}\frakM: = 
\begin{pmatrix}
\langle f_{\bbb_+}, f_{\bbb^\vee_+} \rangle & \langle f_{\bbb_+}, f_{\bbb^\vee_-}\rangle \\ \langle f_{\bbb_-}, f_{\bbb^\vee_+} \rangle& \langle f_{\bbb_-}, f_{\bbb^\vee_-}\rangle
\end{pmatrix} = \begin{pmatrix}
\sum_{i\geq 0 \mathrm{\ even}} 2^{r-1-i} \frakS_i & 
\sum_{i>0 \mathrm{\ odd}} 2^{r-1-i} \frakS_i
\\
\sum_{i>0 \mathrm{\ odd}} 2^{r-1-i} \frakS_i & \sum_{i\geq 0 \mathrm{\ even}}2^{r-1-i} \frakS_i
\end{pmatrix}.
\end{equation}
The determinant of $\frakM$ is $$ \prod_{i=1}^{r-1}(1- e^{2\varepsilon_i} )(1- e^{2\varepsilon_{-i}}).$$
\end{lem}

This agrees with our Theorem~\ref{T:determinant of intersection} because, for $i\neq j \in \{1, \dots, r-1\}$, one can check easily that $\zeta_{\calO_i^\pm}(V) =1$, $\zeta_{\calO_{i,j}^\pm}(V)= 0$.

\begin{proof}
Consider the (isotropic) subspace $U = \bigoplus_{i=1}^r k e_i$ of $Q$ and its ($\ZZ/2\ZZ$-graded) wedge product space $\wedge^*U= (\wedge^*U)^\mathrm{even} \oplus (\wedge^*U)^\mathrm{odd}$. A basis of the latter is given by $e_I = e_{i_1} \wedge \cdots \wedge e_{i_m}$ for $I =\{i_1, \dots, i_m\}$ in which elements are ordered so that $i_1 > \dots > i_m$. (Our unusual ordering is to avoid the appearance of unnecessary signs  later.)
The pin group $\mathrm{Pin}_{2r}$ acts on $\wedge^*U$ as, for $1\leq i \leq m$,
\begin{eqnarray}
\label{E:Pin 2r action on wedge U}
&e_i \bullet (u_1 \wedge \cdots \wedge u_m) =  e_i \wedge u_1 \wedge \cdots \wedge u_m, \\ \nonumber
&
e_{-i} \bullet (u_1 \wedge \cdots \wedge u_m) = \sum_{i=1}^m (-1)^{i-1} \langle e_{-i}, u_i\rangle \cdot  u_1 \wedge \cdots \wedge \widehat{u_i} \wedge \cdots \wedge u_m.
\end{eqnarray}
Restricting to $\mathrm{Spin}_{2r}$, both $(\wedge^*U)^\even$ and $(\wedge^*U)^\odd$ are irreducible representations; when $r$ is even, their highest weights are $\nu_+ = \frac 12(\varepsilon_1+\cdots + \varepsilon_r)$ and $\nu_- = \frac12(\varepsilon_1+ \cdots + \varepsilon_{r-1} - \varepsilon_r)$, respectively; when $r$ is odd, the highest weights are exchanged.
In what follows, we assume that $r$ is even. The other case can be treated similarly. We henceforth identify $\ttS_{\nu_+} \cong (\wedge^*U)^\even$ and $\ttS_{\nu_-} \cong (\wedge^*U)^\odd$.
More generally, the weight of $e_I$ (with $I = \{i_1, \dots, i_m\}$) is 
\begin{equation*}
\label{E:weight of wedge U}\varepsilon_{i_1}+\cdots + \varepsilon_{i_m} - \nu_+.
\end{equation*}
Moreover, the action \eqref{E:Pin 2r action on wedge U} also defines natural morphisms
$$
V \otimes (\wedge^*U)^\even \to (\wedge^*U)^\odd, \quad V \otimes (\wedge^*U)^\odd \to (\wedge^*U)^\even.
$$
The isomorphism $\sigma : \ttS_{\nu_+} \to \ttS_{\sigma(\nu_+)} = \ttS_{\nu_-}$ may be identified with the  natural map $$
({\wedge}^*U)^\even \to ({\wedge}^*U)^\odd, \quad x \mapsto (e_r+e_{-r})\bullet x,$$
intertwining the natural action of $\Spin_{2r}$ on the source and the $\sigma$-twisted action of $\Spin_{2r}$ on the target. Explicitly, this map sends $e_I$ to $e_{I \triangle \{r\}}$, where $I\triangle\{r\}$ is the symmetric difference.

With this discussion, we may express a basis element of
$$
(\ttS_{\sigma(\nu_+)^*} \otimes \ttS_{\nu_+} \otimes V)^G = \big( (\wedge^*U)^{\odd,*} \otimes (\wedge^*U)^\even \otimes V \big)^G
$$
and its class function as
$$
\bbb_+ = \sum_{|I| \textrm{ odd}}\sum_{|I'|\textrm{ even}} e_{I}^* \otimes e_{I'} \otimes v_{I, I'}, \quad f_{\bbb_+}(t\sigma)  = \sum_{|I| \textrm{ odd}}\sum_{|I'|\textrm{ even}} \langle e_{I}^* \otimes t\sigma e_{I'} \rangle \cdot v_{I, I'}
$$
for $v_{I, I'} \in V$ and $t\in T$, where the sum is taken over all subsets $I, J\subset \{1, \dots, r\}$ whose sizes are of the specified parity.
In the expression for $f_{\bbb}(t\sigma)$, the pairing $\langle e_I^*, \sigma e_{I'} \rangle$ is nonzero precisely when $I' = I\triangle \{r\}$; in this case, we may rewrite $\{I,I'\}$ as $\{J, J \cup\{r\}\}$ for some $J \subset \{1, \dots, r-1\}$, and we may take the vector $v_{I, I'}$ to be $e_r$ if $r \in I$ and to be $e_{-r}$ if $r \notin I$.
Therefore, we have
\begin{align*}
f_{\bbb_+}(t\sigma)& = \sum_{|J| \textrm{ even}} \langle e_{J \cup \{r\}}^* , e^{\varepsilon_{J}+\varepsilon_r-\nu_+}(t) \sigma  e_{J} \rangle \cdot e_r + \sum_{|J|  \textrm{ odd}} \langle e_{J}^* , e^{\varepsilon_{J}-\nu_+}(t) e_{J\cup \{r\}} \rangle \cdot  e_{-r}
\\
& = \sum_{|J| \textrm{ even}}  e^{\varepsilon_{J}+\varepsilon_r-\nu_+}(t) \cdot e_r + \sum_{|J|  \textrm{ odd}}  e^{\varepsilon_{J}-\nu_+}(t) \cdot  e_{-r},
\end{align*}
where the sums are taken over subsets $J \subset \{1, \dots, r-1\}$ whose sizes are of the given parity, and for $J = \{i_1, \dots, i_m\}$, $\varepsilon_J = \varepsilon_{i_1} + \cdots + \varepsilon_{i_m}$.
To simplify notation, we put
$$
C_+ = e^{-\nu_+}\cdot \sum_{|J| \textrm{ even}}  e^{\varepsilon_{J}}, \quad C_- =e^{-\nu_+}\cdot  \sum_{|J| \textrm{ odd}}  e^{\varepsilon_{J}}
$$
so that $\Res_V^\sigma(f_{\bbb_+}) = C_+ e^{\varepsilon_r} \cdot e_r + C_- \cdot e_{-r}\in \bfJ_T(V)$.

Similarly, we may choose a basis element
$$
\bbb_- 
\in (\ttS_{\sigma(\nu_-)^*} \otimes \ttS_{\nu_-} \otimes V)^G = \big( (\wedge^*U)^{\even,*} \otimes (\wedge^*U)^\odd \otimes V \big)^G,
$$
with class function
$$
\Res_V^\sigma(f_{\bbb_-}) =C_- e^{\varepsilon_r} \cdot e_r + C_+ \cdot  e_{-r} \in \bfJ_T(V).
$$

On the dual side, we choose elements
\begin{eqnarray*}
&\bbb_+^\vee
\in (\ttS_{\sigma(\nu_+)} \otimes \ttS_{\nu_+^*} \otimes V^*)^G = \big( (\wedge^*U)^{\odd} \otimes (\wedge^*U)^{\even,*} \otimes V^* \big)^G,\\
&\bbb_-^\vee 
\in (\ttS_{\sigma(\nu_-)} \otimes \ttS_{\nu_-^*} \otimes V)^G = \big( (\wedge^*U)^{\even} \otimes (\wedge^*U)^{\odd,*} \otimes V \big)^G.
\end{eqnarray*}
If we write $D_+ = e^{\nu_+}\cdot \sum_{|J| \textrm{ even}}  e^{-\varepsilon_{J}}$, $ D_- =e^{\nu_+}\cdot  \sum_{|J| \textrm{ odd}}  e^{-\varepsilon_{J}}$, the class functions $f_{\bbb_+^\vee}, f_{\bbb_-^\vee} \in \bfJ(V^*)$ have restrictions
$$
\Res_{V^*}^\sigma(f_{\bbb_+^\vee}) =D_+  e^{-\varepsilon_r} \cdot e_r^* + D_- \cdot  e_{-r}^*, \quad 
\Res_{V^*}^\sigma(f_{\bbb_-^\vee}) =D_- e^{-\varepsilon_r} \cdot e_r^* + D_+ \cdot  e_{-r}^* \quad \in \bfJ_T(V^*).
$$

Combining all above, we deduce that
$$
\frakM = 
\begin{pmatrix}
\langle f_{\bbb_+}, f_{\bbb^\vee_+} \rangle & \langle f_{\bbb_+}, f_{\bbb^\vee_-}\rangle \\ \langle f_{\bbb_-}, f_{\bbb^\vee_+} \rangle& \langle f_{\bbb_-}, f_{\bbb^\vee_-}\rangle
\end{pmatrix} = \begin{pmatrix}
C_+ D_+ + C_-D_-
 & 
C_+ D_- + C_-D_+
\\
C_+ D_- + C_-D_+ & C_+ D_+ + C_-D_-
\end{pmatrix}.
$$
Now the formula \eqref{E:matrix for Spin} follows from the following identity
\begin{align*}
C_+D_++ C_-D_- = \sum_{i \geq 0 \textrm{ even}} 2^{r-1-i} \frakS_i, \quad C_+D_-+ C_-D_+ = \sum_{i \geq 0 \textrm{ odd}} 2^{r-1-i} \frakS_i.
\end{align*}
It is easier to  compute the determinant of $\frakM$ in the following way:
$$
\det \frakM = (C_+ D_+ + C_-D_-)^2 - (C_+ D_- + C_-D_+)^2 = (C_+ + C_-) (D_++D_-) (C_+ - C_-)(D_+-D_-).
$$
On the other hand, it is easy to see that
$$
C_+ \pm C_- = e^{-\nu_+} \prod_{i = 1}^{r-1} (1\pm e^{\varepsilon_i}), \quad D_+ \pm D_- = e^{\nu_+} \prod_{i = 1}^{r-1} (1\pm e^{-\varepsilon_i}).
$$
So
\[
\det \frakM = \prod_{i=1}^{m-1} \big(( 1- e^{2\varepsilon_i} )( 1- e^{-2\varepsilon_i} )\big).\qedhere
\]
\end{proof}

\newpage

\appendix

\section{A remark on freeness} 

\medskip 
\begin{center}
{\it Stephen Donkin}\footnote{Department of Mathematics,  University of York, York YO10 5DD.
\ E-mail address: \tt stephen.donkin@york.ac.uk}
\end{center}

\medskip

Let $G$ be a semisimple, simply connected algebraic group over an  algebraically closed field $k$. Then conjugation gives rise to a $G$-module structure on the coordinate algebra $k[G]$ and the algebra of class functions  $C(G)$ is the invariant algebra. 
Let $V$ be a finite dimensional rational $G$-module which admits a good filtration.    In Theorem~\ref{T:main theorem} it is shown that the space of invariants  $(k[G]\otimes V)^G$ is free as a $C(G)$-module. 

We here obtain this as a special case of Theorem~\ref{T:Donkin theorem}  below. However, we  would like to stress that this does not cover the freeness of $\bfJ_0(V)$ and $\bfJ_+(V)$ established  in Theorem~\ref{T:main theorem}, nor does it cover the cases  in which the action on the coordinate algebra  is twisted via an automorphism,  considered  by the authors.

In \cite{DonkCR} we considered the situation in which $G$ acts rationally on a finitely generated $k$-algebra $A$, see \cite{DonkCR}, 1.5 Theorem and deduced freeness over the algebra of invariants, in some cases, in particular  in the case $A=k[G]$ with the conjugation action. 
The argument we give below is essentially a re-run of the argument from  \cite{DonkCR}, but where we now consider invariants of a finitely generated $A$-module and rational $G$-module. Our argument   relies on flatness and we should mention that  in the cases considered in \S~\ref{S: fil vect and Rees}--\S~\ref{S:vector-valued twisted invariant functions} a more constructive approach is taken and flatness is obtained as a corollary of freeness.

\subsection{Some General Recollections}

By a $G$-algebra we mean a commutative $k$-algebra on which $G$ acts rationally as $k$-algebra automorphisms. We will denote the action of $g\in G$-algebra on  $a\in A$ by $g\cdot a$.  Let $A$ be a $G$-algebra. By an $(A,G)$-module we mean a $k$-space $M$ which is an $A$-module and a rational $G$-module in such a way that $g(am)=(g\cdot a) gm$, for all $g\in G$, $a\in A$, $m\in M$.   The space of invariants of a $G$-module $M$ will be denoted $M^G$ and $C$ denotes $A^G$.

We here  modify the arguments of \cite{DonkCR},  pp139--140,  to  give a generalization to $(A,G)$-modules of the main result of \cite{DonkCR}.   We begin by describing some properties of $(A,G)$-modules. These are no doubt well known but we include them for the convenience of the reader.

Recall that, from the Mumford Conjecture (proved by Haboush in \cite{Haboush})  if $A$ is finitely generated then so is $C$. Moreover, if $\pi:A\to B$ is a surjection of $G$-algebras then for $b\in B^G$ we have $b^n=\pi(a)$ for some $a\in A$ and positive integer $n$ (see \cite{Newstead},  p54). In particular $B^G$ is integral over $\pi(A^G)$ and so if $A$ is finitely generated then  $B^G$ is a finitely generated $A^G$-module.

Now suppose that $A$ is finitely generated and that $M$ is a finitely generated $(A,G)$-module. We claim that $M^G$ is finitely generated as a   $C=A^G$-module.  If this holds we say that $M$ has the finite generation property.  Note that a submodule of an $(A,G)$-module with the finite generation property has the finite generation property.  Note also that an extension of $(A,G)$-modules with this property also has this property: Suppose $0\to X\to Y\to Z\to 0$ is a short exact sequence of $(A,G)$-modules and $X,Z$ have  the property. Then we have a short exact sequence of $C$-modules 
$$0\to X^G\to Y^G\to (X+Y^G)/X\to 0.$$
  Now $X+Y^G/X$ embeds in $(Y/X)^G$ and so is finitely generated as a $C$-module. Thus $X^G$ and $(X+Y^G)/X$ are finitely generated $C$-module so that $Y^G$ is finitely generated too.

We consider the case in which $M$ is generated  by a non-zero  invariant $m_0\in M^G$. Thus  $M=Am_0$ and $M$ is isomorphic to $A/I$, where $I$ is the annihilator of $m_0$. We  have the natural map $\pi:A\to B=A/I$ and by the above $B^G$ is integral over $B_0=\pi(A^G)$.  Hence $B^G$ is finitely generated as a $C$-module and $M$ has the finite generation property. 

Now consider the general case. Let $M$ be a finitely generated $(A,G)$-module and let $X$ be a submodule maximal subject  to the condition that it has the finite generation property. If $(M/X)^G=0$ then $M^G=X^G$ so that $M$ has the required property. Otherwise we choose $m_0\in M\backslash X$ with $(m_0+X)\in (M/X)^G$.  Put $Y=X+Am_0$. Then $Y/X$ is generated by an invariant so has  the property. Hence $X$ and $Y/X$ have  the finite generation property and therefore so has $Y$.  But $X$ is strictly contained in $Y$ so we have a contradiction.

Thus we have the following.

\begin{itemize}
\item [(1)]{\it If $A$ is a finitely generated $G$-algebra and $M$ is a finitely generated $(A,G)$-module then $M^G$ is a finitely generated $A^G$-module.}
\end{itemize}

We need to improve (1) so that we can take invariants for the action of a subgroup of $G$. Now suppose that $H$ is a Grosshans subgroup of $G$, i.e., a closed subgroup such that the  the algebra 
$$k[H\backslash G]=\{f\in k[G]\vert f(hg)=f(g) \hbox{ for all } h\in H, g\in G\}$$
 is finitely generated. (In fact we only in which  need the case $H$ is  a maximal unipotent subgroup of $G$.).   Suppose that $A$ is a finitely generated $G$-algebra and $M$ is a finitely generated $(A,G)$-module. We shall need that $M^H$ is finitely generated as an $A^H$-module. 
By Frobenius reciprocity and the tensor identity one has  $M^H=(M\otimes k[H\backslash G])^G$ as $k$-spaces
 and $(A\otimes k[H\backslash G])^G=A^H$ 
 so the argument should be to regard $(M\otimes k[H\backslash G])$ as an  $A\otimes k[H\backslash G]$-module and take invariants. 
 
To see that this really works we  write down explicitly the maps involved in identifying $H$-invariants and $G$-invariants. This is a slight extension of the context of \cite{Popov}, Theorem 4.

Let $\sigma:k[G]\to k[G]$ be the antipode, so $\sigma(f)(x)=f(x^{-1})$, for $f\in k[G]$, $x\in G$. Let $A$ be a $G$-algebra and let $M$ be an $(A,G)$-module.   We choose a $k$-basis $(m_i)_{i\in I}$ of $M$. Let  $f_{ij}\in k[G]$ be the corresponding coefficient functions so that 
$$gm_i=\sum_{j\in I} f_{ji}(g)m_j$$
for $g\in G$, $i\in I$.

It is  easy to check that there are
 inverse  $k$-linear isomorphisms $\phi_M:(M\otimes k[H\backslash G])^G\to M^H$ and $\psi_M:M^H\to (M\otimes k[H\backslash G])^G$ satisfying 
 $$\phi_M(\sum_{i\in I} m_i\otimes b_i)=\sum_{i\in I} b_i(1) m_i$$
 for $\sum_{i\in I} m_i\otimes b_i\in (M\otimes k[H\backslash G])^G$ and 
 $$\psi_M(\sum_{\in I} \lambda_im_i)=\sum_{i,j\in I} 
 \lambda_i m_j\otimes \sigma(f_{ji})$$
 for $\sum_{i\in I}\lambda_im_i\in M^H$.
 
 For  $a\in A^H$ and $m\in M^H$ one checks that $\psi_M(am)=\psi_A(a)\psi_M(m)$.    We regard $M\otimes k[H\backslash G]$ as an $(A\otimes k[H\backslash G],G)$-module in the natural way.  By (1),  $(M\otimes k[H\backslash G])^G$  is a finitely generated $(A\otimes k[H\backslash G])^G$-module, i.e., 
$$(M\otimes k[H\backslash G])^G=\sum_{i=1}^n (A\otimes k[H\backslash G])^G y_i$$
for some $y_1,\ldots,y_n\in (M\otimes k[H\backslash G])^G$.  Let $x_i\in M^H$ be such that $\psi_M(x_i)=y_i$ for $1\leq i\leq n$ and put  $D=\sum_{i=1}^n A^Hx_i\leq M^H$.  Then  
$$\psi_M(D)=\sum_{i=1}^n \psi_A(A^H)\psi_M(x_i)=\sum_{i=1}^n (A\otimes k[H\backslash G])^G y_i=(M\otimes k[H\backslash G])^G$$
and so $D=M^H$ and $M^H$ is finitely generated.

To summarise, we have the following.

\begin{itemize}
\item[(2)] {\it If $A$ is a finitely generated $G$-algebra, $H$ is a Grosshans subgroup of $G$ and $M$ is a finitely generated $(A,G)$-module then $M^H$ is a finitely generated $A^H$-module.}
\end{itemize}

\subsection{Freeness}

We are now in a position to extend the main result of \cite{DonkCR} to the context of $(A,G)$-modules.  We adopt the notation of \cite{RAG}. In particular we have the maximal torus $T$  and Weyl group $W$. We  attach a root system to $G$, with respect to $T$, and let $B$ denote the negative Borel subgroup. We have the character group $X(T)$ with natural partial order $\leq $  and set of dominant weight  $X^+(T)$.  For $\lambda\in X^+(T)$ we write $L(\lambda)$ for the irreducible rational $G$-module with highest weight $\lambda$, write $k_\lambda$ for the one dimensional $B$-module on which $T$ acts via $\lambda$ and write $\nabla(\lambda)$ for the induced module $\ind_B^Gk_\lambda$.  A subset $\pi$ of $X^+(T)$ is said to be saturated if whenever $\lambda\in \pi$ and $\mu \in X^+(T)$ with $\mu\leq \lambda$ then $\mu\in \pi$.  For $\pi$ a saturated subset of $X^+(T)$ and $V$ a rational $G$-module the set of submodules of $V$ which have all composition factors belonging to $\{L(\lambda)\vert \lambda\in \pi\}$ has a unique maximal element, which we denote $O_\pi(V)$.

Theorem 1.5  in \cite{DonkCR}, which we now extend to the context  of $(A,G)$-modules,  is obtained  via  three Propositions.  The first,  Proposition 1.2.  and the third,   Proposition 1.4,   require no modification.

The modified  Proposition 1.3 and its proof   require  only  minimal changes,  given the remarks on finite generation  above,  but we give it again for the sake of completeness,  in the form of the following lemma. 

\begin{lemma}  
\label{L:Donkin lemma}
Let $A$ be a finitely generated $G$-algebra and let $C=A^G$. Suppose that $M$ is a finitely generated $(A,G)$-module and that $M$ has a good filtration.  Then, for every finite saturated subset $\pi$ of $X^+(T)$, the $C$-module $O_\pi(M)$ is a finitely generated. 
\end{lemma}

\begin{proof} Let $\lambda$ be a maximal element of $\pi$ and let $\pi'=\pi\backslash \{\lambda \}$. Then \\
$O_\pi(M)/O_{\pi'}(M)\cong C\otimes O_\pi(M)^\lambda$, by \cite{DonkCR}, 1.2 Proposition. By   induction on the size of  $\pi$  it therefore suffices  to prove that $O_\pi(M)^\lambda$ is a finitely generated $C$-module.   Moreover, multiplication by a coset representative of $w_0$ (the longest element of the Weyl group $W$) induces  an isomorphism $O_\pi(M)^\lambda\to O_\pi(M)^{w_0\lambda}$,  so it suffices to show that $O_\pi(M)^{w_0\lambda}$ is finitely generated.

Let $M_0=M^U$ and $A_0=A^U$. Since $w_0\lambda$ is a lowest weight of $O_\pi(M)$, we have 
$O_\pi(M)^{w_0\lambda}\leq M_0^{w_0\lambda}$. On the other hand  ${\bar M}^{w_0\lambda}=0$, by \cite{D4}, (12.1.6) and (1.5.2),  where  where ${\bar M}=(M/O_\pi(M))^U$, so that $O_\pi(M)^{w_0\lambda}=M_0^{w_0\lambda}$.   Furthermore, $M_0$ is a $T$-module and $A_0^T=A^B=A^G$, by \cite{CPSvdK}, (2.1) Theorem. Moreover, $U$ is a Grosshans subgroup, e.g., by \cite{Grosshans}, so that $A_0$ is finitely generated and, by (2),  $M_0$ is a finitely generated $A_0$-module.  Hence we may (and do) replace $M$ by  $M^U$ and $G$ by $T$.  So it suffices to prove that if $A$ is  finitely generated $T$-algebra and $M$ is a finitely generated $(A,T)$-module then for any $\mu\in X(T)$ the weight space $M^\mu$ is a finitely generated $A^T$-module.  We choose a finite dimensional $T$-invariant subspace of $M$ which generates $M$, as a $(A,T)$-module. Then we have a surjective $(A,T)$-module map $A\otimes V\to M$ inducing a surjection on invariants (by complete reducibility of $T$-modules). Hence we may assume $M=A\otimes V$ and, by complete reducibility of $T$-modules again, that $V$ is one dimensional, with weight $\tau$ say. Then $(A\otimes V)^\mu$ is isomorphic to $A^{\mu-\tau}$. Hence it suffices to prove that $A^\lambda$ is a finitely generated $A^T$-module, for $\lambda\in X(T)$.  So let $\chi=-\lambda$. Then $A\otimes k[\chi]$ is a finitely generated $T$-algebra and so $(A\otimes k[\chi])^T$ is finitely generated  by  $a_i\otimes \chi^{d_i}$ say, for $1\leq i\leq n$, $d_i\geq 0$. Then $A^\lambda$ is generated as a $A^T$-module by $\{a_i : 1\leq i\leq n \hbox{ and } d_i=1\}$. 
\end{proof}

Now the proofs of \cite{DonkCR} 1.5 Theorem and its Corollary go through in the context of $(A,G)$-modules with  Lemma~\ref{L:Donkin lemma} replacing \cite{DonkCR},  Proposition 1.3, and we obtain the following.

\begin{theorem}
\label{T:Donkin theorem}
Suppose that $A$ is a finitely generated $G$-algebra whose algebra of invariants $C=A^G$ is a free polynomial algebra. Suppose that  $M$ is a finitely generated $(A,G)$-module such that $M$ is flat as an $A$-module and has a good filtration  (as a $G$-module).  

 Let $\pi$ be any finite saturated subset $\pi$ of $X^+(T)$,  let $\lambda$ be a maximal element and put $\pi'=\pi\backslash\{\lambda\}$.  Then, as a $(C,G)$-module, $O_\pi(M)/O_{\pi'}(M)$ is isomorphic to $E\otimes C$, where $E$ is isomorphic to a direct sum of finitely many copies of $\nabla(\lambda)$.
 \end{theorem}

Given a $C$-module $N$ and a $k$-space of $G$-module  $E$ we write $|E|\otimes N$ for the vector space $E\otimes N$ viewed as a $C$-module with action $c(e\otimes n)=e\otimes cn$, for $c\in C$, $e\in E$, $n\in N$.  

For a finite dimensional $G$-module $E$ admitting a good filtration and  $\lambda\in X^+(T)$, we write $(E:\nabla(\lambda))$ for the multiplicity of $\nabla(\lambda)$ as a section in a good filtration of $E$.

\begin{cor}
\label{C:Donkin cor 1}  Under the hypotheses of the Theorem, $M$ has an ascending $(C,G)$-module filtration $0=M_0,M_1,\ldots$, where $M_i/M_{i-1} \cong |E_i|\otimes C$, $E_i$ is a finite direct sum of copies of $\nabla(\lambda_i)$ ($i\geq 1$) and $\lambda_1,\lambda_2,\ldots$ is a labelling of the elements of $X^+(T)$ such that $i<j$ whenever $\lambda_i<\lambda_j$. For a given labelling the multiplicity $(E_i:\nabla(\lambda_i))$ is independent  of the choice of such a filtration.

In particular $M$ is  a free $C$-module.
\end{cor}

We specialize further to the situation of this paper, Theorem~\ref{T:main theorem}.

\begin{cor}
Regard  $k[G]$ as a $G$-module via the conjugating action and let $V$ be a finite dimensional rational $G$-module which admits a good filtration. Then $(k[G]\otimes V)^G$ is free over $C(G)$.

\end{cor} 

Note that $C(G)$ is a free polynomial algebra by \cite{Steinberg}, 6.1 Theorem. Moreover,    $k[G]\otimes V$ is free, and hence  flat over $k[G]$ and $k[G]$ is flat over $C(G)$, by \cite{Richardson-appendix}, Proposition 2.3, so that $k[G]\otimes V$ is flat over $C(G)$.
We take $A=k[G]$, considered as a $G$-algebra via the conjugating representation  and  $M=k[G]\otimes V$ and $\pi=\{0\}$.   Then $M$ has a good filtration, by \cite{RAG}, II, 4.20 Proposition and 4.21 Proposition so we may apply Corollary~\ref{C:Donkin cor 1} .


\end{document}